\documentclass[11pt]{amsart}

\usepackage{color}
\usepackage{comment}
\usepackage[margin=1in]{geometry}
\usepackage{graphicx}

\def\eq{\begin{equation}}
\def\endeq{\end{equation}}
\def\bbm{\begin{bmatrix}}
\def\ebm{\end{bmatrix}}
\newcommand{\C}{\mathbb{C}}
\newcommand{\Hankel}{\mathcal{H}}
\newcommand{\bigO}{\mathcal{O}}
\newcommand{\Prob}{\mathbb{P}}
\newcommand{\realR}{\mathbb{R}}
\newcommand{\T}{{\mathbb T}}
\newcommand{\Z}{\mathbb{Z}}
\newcommand{\be}{\beta}
\newcommand{\ga}{\gamma}
\newcommand{\de}{\delta}
\newcommand{\ep}{\epsilon}
\newcommand{\Sg}{\Sigma}
\newcommand{\sg}{\sigma}

\newcommand{\om}{\omega}
\newcommand{\hsgn}[1]{\epsilon(#1)}
\newcommand{\lattice}{L_{n, \tau}}
\renewcommand{\d}{\partial}

\newtheorem{theo}{Theorem}[section]
\newtheorem{prop}[theo]{Proposition}
\newtheorem{rhp}[theo]{Riemann--Hilbert Problem}
\newtheorem{lem}[theo]{Lemma}

\theoremstyle{remark}
\newtheorem{rmk}{Remark}

\newcommand{\NIBMT}{\ensuremath{\text{NIBM}^\mu_{0 \to T}}}

\numberwithin{equation}{section}

\title[Nonintersecting Brownian bridges]{Nonintersecting Brownian bridges on the unit \\ circle with drift}

\author{Robert Buckingham}

\author{Karl Liechty}

\thanks{The authors thank Dong Wang, who was involved in an early stage of this project, and Peter Miller for useful discussions.  
Robert Buckingham was supported by by the National Science Foundation through 
grant DMS-1615718 and by the Charles Phelps Taft Research Center through a 
Faculty Release Fellowship.  Karl Liechty was supported by the Simons Foundation through grant \#357872.}

\begin{document}

\begin{abstract} 
Nonintersecting Brownian bridges on the unit circle form a determinantal 
stochastic process exhibiting random matrix statistics for large numbers of 
walkers.  We investigate the effect of adding a drift term to walkers on the 
circle conditioned to start and end at the same position.  For each return 
time $T<\pi^2$ we show 
that if the absolute value of the drift is less than a critical value then 
the expected total winding number is asymptotically zero.  In addition, we 
compute the asymptotic distribution of total winding numbers in the 
double-scaling regime in which the expected total winding is finite.  
The method of proof is Riemann--Hilbert analysis of a certain family of 
discrete orthogonal polynomials with varying complex exponential weights.  
This is the first asymptotic analysis of such a class of polynomials.  We 
determine asymptotic formulas and demonstrate the emergence of a second band 
of zeros by a mechanism not previously seen for discrete orthogonal 
polynomials with real weights.

\vspace{-.1in}

\end{abstract}


\maketitle

\tableofcontents

\section{Introduction}


In 1962, Dyson \cite{Dyson:1962} made the remarkable observation that the eigenvalues of an $n\times n$ GUE random matrix obey the same statistics as a certain process comprised of nonintersecting Brownian motions (NIBM).  Since then,  it has been shown that NIBM models give rise to a plethora of universal stochastic 
processes, including the sine, Airy, Pearcey, and tacnode processes \cite{TracyW:2004, TracyW:2006, BleherK:2007, DelvauxKZ:2011, FerrariV:2012, Johansson:2013}, which appear in a wide range of problems in probability and mathematical physics. In addition, NIBM models and related models of nonintersecting paths serve as tractable models for the study of subjects as diverse as transportation engineering \cite{BaikBDS:2006}, wetting and melting \cite{Fisher:1984}, polymers and random interfaces \cite{Schehr:2012}, Yang--Mills theory \cite{ForresterMS:2011}, dynamics of quantum systems \cite{leDoussalMS:2017}, etc.
 Technically it is often easiest to study nonintersecting Brownian bridges, in which the starting and ending points of the Brownian paths are fixed.
 
 A natural generalization is  to consider nonintersecting Brownian bridges on the 
circle. This model was studied in depth recently by Wang and the second author in \cite{LiechtyW:2016}. To be specific, they considered an ensemble of $n$ nonintersecting Brownian bridges on the unit circle $\T = \{e^{i\theta} | -\pi \le \theta <\pi\}\subset \C$ with diffusion parameter $n^{-1/2}$, conditioned to begin at a common point at time $t=0$ and to end at the same common point at time $t=T>0$. We can summarize the global properties of that model as the number of particles $n\to\infty$ obtained in \cite{LiechtyW:2016} as follows; see Figure \ref{no-drift-simulation}. The asymptotic behavior of the particles as $n\to\infty$ depends on the return time $T$. In particular it is shown that there is a critical value $T_c=\pi^2$ of the return time $T$ separating the {\it subcritical} return times from the {\it supercritical} return times. In the subcritical case $T<\pi^2$, the particles do not have enough time to wrap around the circle and the asymptotic behavior is identical to nonintersecting Brownian bridges on the real line. In this case the limiting density of particles at any fixed time $t\in (0,T)$ converges to a properly rescaled semi-circle distribution and the boundary of the convex hull of the paths in space-time converges to an ellipse. In the supercritical case $T>\pi^2$, as $n\to\infty$ there is a nonvanishing probability that some particles wrap around the circle, and the distribution of the sum of the winding numbers of the particles converges to a discrete normal distribution. This case is considerably more complicated, and both the density of particles at fixed times and the boundary of the convex hull of the paths in space-time are expressed by elliptic functions.

In the current work we extend the model discussed above to nonintersecting Brownian bridges on $\T$ with a drift. For planar Brownian bridges, adding a drift to the Brownian motion has no effect, but for Brownian bridges on the unit circle the drift has the effect of encouraging particles to wrap around the circle. The analysis of \cite{LiechtyW:2016} is based on the fact that the nonintersecting Brownian bridges on $\T$ form a determinantal process whose kernel can be expressed in terms of a system of discrete Gaussian orthogonal polynomials which may be studied asymptotically using Riemann--Hilbert methods. When a drift is added the structure of the determinantal process is the same, but now the discrete orthogonal polynomials we need to consider have a complex exponential weight. In general discrete orthogonal polynomials with complex weights are difficult to study asymptotically, and we are not aware of any such previous work.

\begin{figure}[h]\label{no-drift-simulation}
\begin{center}
\includegraphics[height=1.8in]{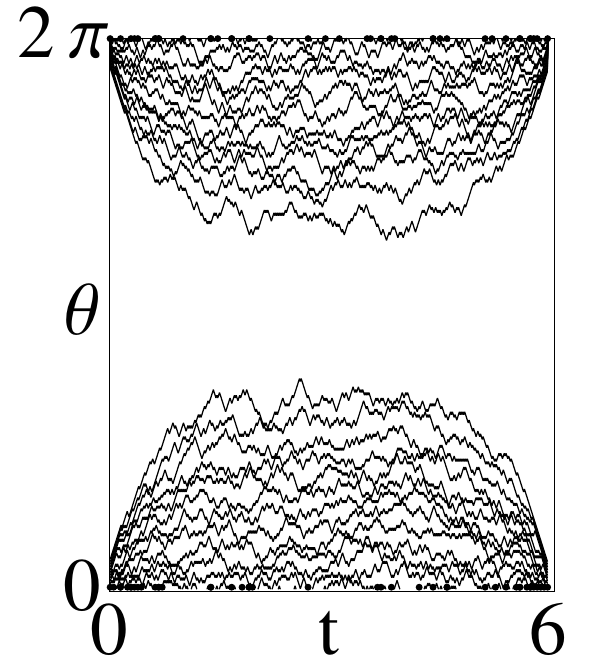}
\includegraphics[height=1.8in]{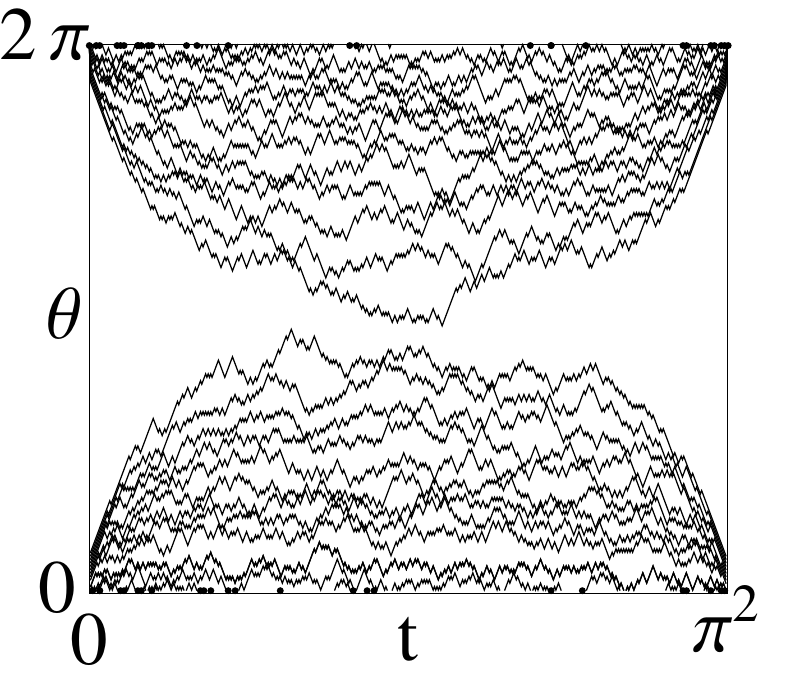}
\includegraphics[height=1.8in]{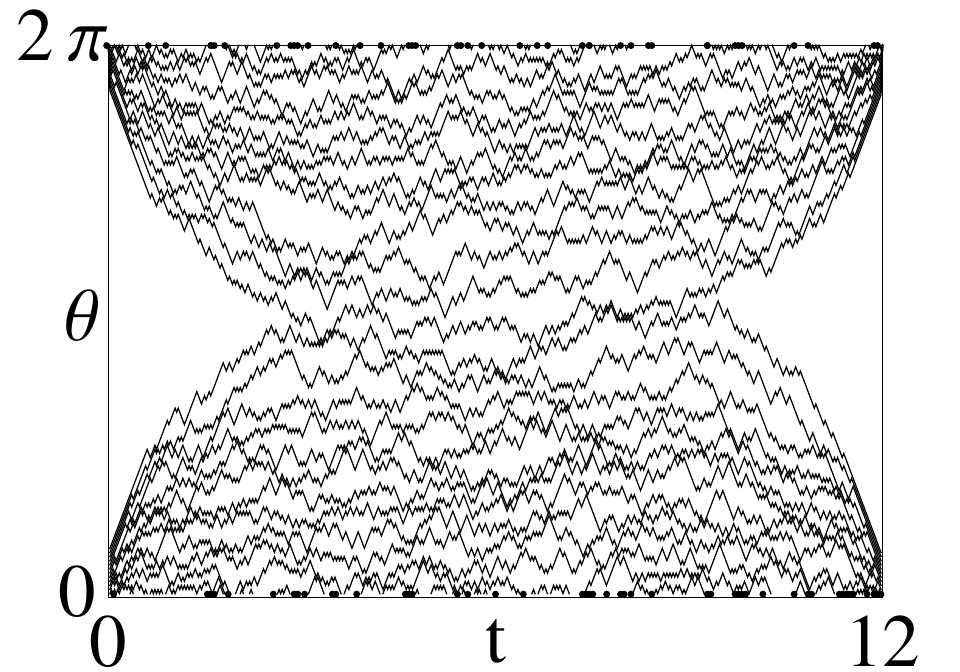}
\caption{Random walk simulations of a Dyson Brownian bridge on the circle with 24 walkers in the subcritical (left), critical (center), and supercritical (right) regimes.}
\label{random-walk-plots}
\end{center}
\end{figure}
\begin{figure}[h]
\begin{center}
\includegraphics[height=1.9in]{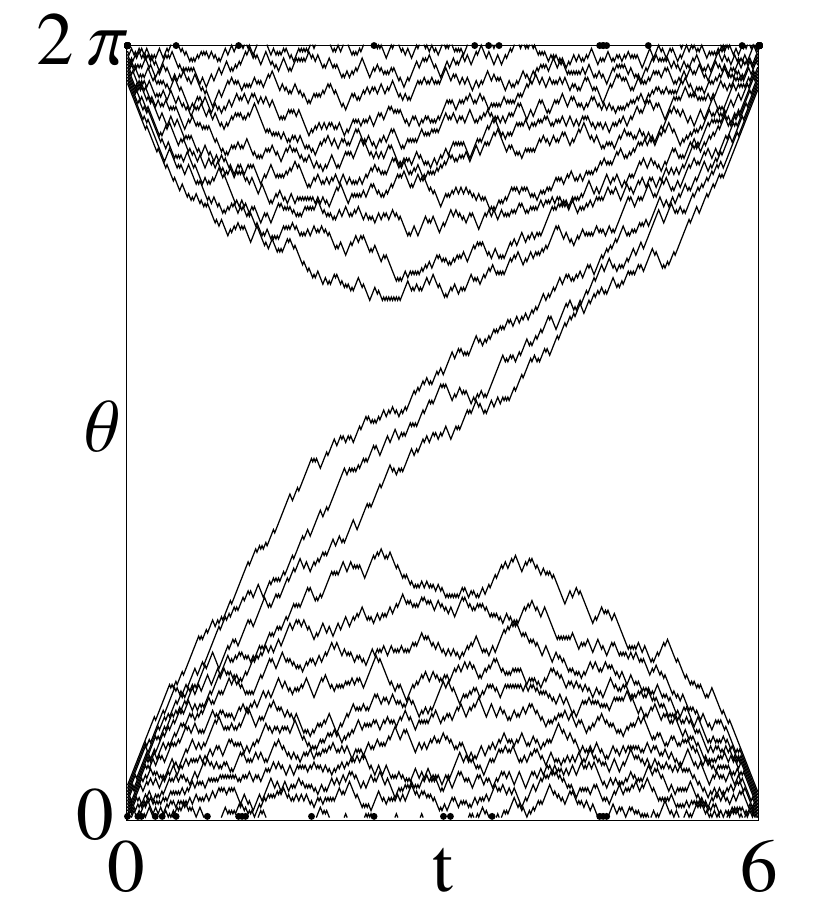}
\includegraphics[height=1.9in]{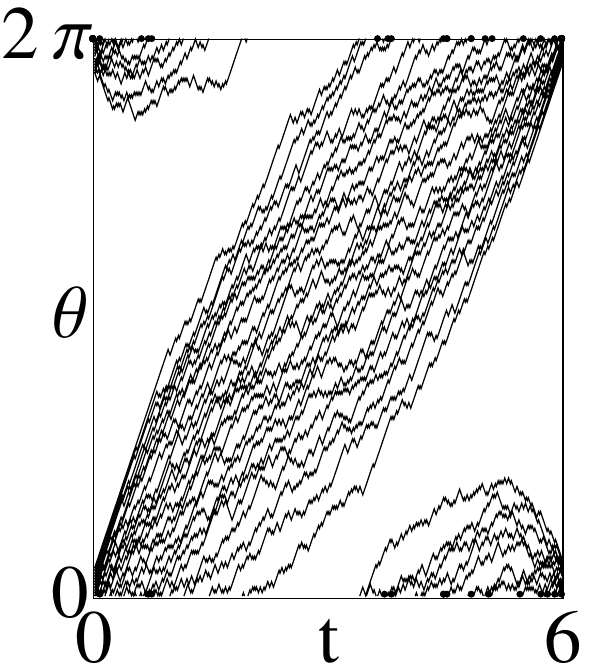}
\caption{Random walk simulations of a Dyson Brownian bridge on the circle 
with 24 walkers and total winding 3 (left) and 24 (right).}
\label{random-walk-crossers-plot}
\end{center}
\end{figure}

In the remainder of the introduction we define the model we study and present our 
results on winding numbers and orthogonal polynomials. 

\subsection{Definition of the model}
Consider a Brownian motion on $\realR$ with drift $\mu$ and diffusion 
parameter $\sg$. By definition, the probability density for the particle to move from 
position $x$ to position $y$ in time $t$ is
\begin{equation}
  P_\realR(x,y;t,\sigma,\mu) =  \frac{1}{\sqrt{2\pi t} \sigma}  \exp\left(-\frac{(y - x-t\mu )^2}{2t \sigma^2}\right).
   \end{equation}
Now consider a Brownian motion on the unit circle $\T$.  We refer to a 
particle at $e^{i\varphi}$ as being at position $\varphi$ and use the principal value of the argument $-\pi \le \varphi <\pi$.  Taking into account 
that the particle can wrap around the circle multiple times, the transition 
probability density for a particle starting at $\varphi$ to move to $\theta$ in time $t$ 
is
\begin{equation}
  P_\T(\varphi,\theta;t,\sigma,\mu) =  \frac{1}{\sqrt{2\pi t}\sigma}  \sum_{k=-\infty}^\infty \exp\left(-\frac{(\theta - \varphi-t\mu+2k\pi )^2}{2t\sigma^2}\right).
   \end{equation}
We will consider an ensemble of $n$ Brownian motions on the unit 
circle and set the diffusion parameter to be $\sg=n^{-1/2}$.  We introduce 
a phase parameter $\tau$ and define
\begin{equation}
\label{Pn-sum}
P_n(\varphi,\theta;t,\mu,\tau) :=  \sqrt{\frac{n}{2\pi t}}  \sum_{k=-\infty}^\infty \exp\left(-n\frac{(\theta - \varphi-t\mu+2k\pi )^2}{2t}+2\pi i k \tau\right).
\end{equation}
This is no longer a probability density in general, but the parameter $\tau$ allows us to keep track of exactly how many times 
particles wraps around the circle.  In particular, 
\eq
\mathbb{P}(\text{offset }\omega\,|\,\text{particle starts at }\varphi \text{ and ends at }\theta) = \int_0^1 \frac{P_n(\varphi,\theta;t,\mu,\tau) e^{-2\pi i \om \tau}}{P_n(\varphi,\theta;t,\mu,0)}d\tau.
\endeq
This can be seen by noting the integral isolates the $k=\omega$ term in 
\eqref{Pn-sum}.
   
Now consider the transition probability density for $n$ such Brownian 
particles on $\T$ conditioned not to intersect.  Let $A_n=\{a_1, \dots, a_n\}$ 
and $B_n=\{b_1, \dots, b_n\}$ be two sets of $n$ distinct points in $\T$ such 
that $-\pi \leq a_1 < a_2 < \dotsb < a_n < \pi$ and 
$-\pi \leq b_1 < b_2 < \dotsb < b_n < \pi$, and denote by 
$\mathcal{P}_n(A_n,B_n,t;\mu)$ the transition probability density of nonintersecting Brownian motions with the particles starting at the points $A_n$ and ending at the points $B_n$ after time $t$.  Note that we do not require that the particle that started at point $a_k$ ends at point $b_k$, but only that it ends at point $b_j$ for some $j=1,\dots, n$. 
 Introduce the notation
\begin{equation}\label{eq:def_hsgn}
  \hsgn{n} :=
  \begin{cases}
    0 & \text{if $n$ is odd}, \\
    \frac{1}{2} & \text{if $n$ is even}.
  \end{cases}
\end{equation}
We define the $n\times n$ determinant 
\begin{equation}
\mathcal{P}_n(A_n,B_n,t;\mu,\tau) := \det\bigg(P_n\big(a_j,b_k,t;\mu,\tau\big)\bigg)_{j,k=1}^n\,.
\end{equation}
Following \cite{LiechtyW:2016}, we then have that the transition probability 
density function for the NIBM on $\T$ with starting points $A_n$ and ending points $B_n$ is given exactly by $\mathcal{P}_n(A_n,B_n,t;\mu,\hsgn{n})$.  

At this point we fix a return time $T$ and take the limit as all of the starting and ending points go to zero.  Then the joint density of the particles at a fixed time $t\in(0,T)$ is given by
 \begin{equation}\label{eq:jpdf1}
 \lim_{\substack{a_1, \dots, a_n \to 0 \\ b_1,\dots, b_n \to 0}} \frac{\mathcal{P}_n(A_n,\Theta_n,t;\mu,\hsgn{n}) \mathcal{P}_n(\Theta_n,B_n,t;\mu,\hsgn{n})}{\mathcal{P}_n(A_n,B_n,t;\mu,\hsgn{n})}\,,
 \end{equation}
where $\Theta_n = \{ \theta_1, \dotsc, \theta_n \}$  describes the locations of the $n$ particles at time $t$. It is not difficult to see that such a limit exists, and so that our model is well defined (see \cite[Section 2.2]{LiechtyW:2016}).
This defines a model we denote as $\NIBMT$. 

The model \NIBMT\ is a determinantal process, meaning that for any fixed time $t\in (0,T)$, the correlation functions of the particles may be described by a particular determinantal formula. There exists some kernel function $K_n(x,y; t)$ such that the $k$-point correlation function for the positions of the particles at time $t$ is given by
\begin{equation} \label{eq:defn_Rmelon_special}
  R_{0\to T}^{(n)}(\theta_1, \dotsc, \theta_{k} ; t) = \det \left( K_n\left(\theta_i, \theta_j; t\right) \right)_{i, j = 1}^k\,.
\end{equation}
The kernel function $K_n\left(\theta_i, \theta_j; t\right)$ may be expressed in terms of orthogonal polynomials. Namely,  let $p^{(T,\mu,\tau)}_{n, j}(x)$ be the monic polynomial of degree $j$ that 
satisfies
\begin{equation} \label{eq:defn_of_discrete_Gaussian_OP}
 \frac{1}{n} \sum_{x \in \lattice} p^{(T,\mu,\tau)}_{n, j}(x) p^{(T,\mu,\tau)}_{n, k}(x)e^{-\frac{Tn}{2}(x^2-2i\mu x)} = h_{n,j}^{(T,\mu,\tau)}\de_{jk}
\end{equation}
for a sequence $\{h_{n,j}^{(T,\mu,\tau)}\}_{j=0}^\infty$ of normalizing constants, 
where the lattice $\lattice$ is defined as
 \begin{equation}
\label{lattice-def}
  \lattice := \left\{ \frac{k + \tau}{n} \mid k \in \Z \right\}.
\end{equation}
Also define the auxiliary function
\eq
S_{j,a}(\varphi;T,\mu,\tau,n)\equiv S_{j,a}(\varphi):=\frac{1}{n}\sum_{x\in L_{n,\tau}}p_{n,j}^{(T,\mu,\tau)}(x)e^{-an(x^2-2i\mu x)/2}e^{i\varphi nx}.
\endeq
The $\tau$-deformed correlation kernel is given by 
\eq
\label{tau-kernel}
K_n(\varphi,\theta; t,T,\mu,\tau):= \frac{n}{2\pi}\sum_{j=0}^{n-1}\frac{1}{h_{n,j}^{(T,\mu,\tau)}}S_{j,T-t}(\varphi)S_{j,t}(-\theta),
\endeq
and the correlation kernel in \eqref{eq:defn_Rmelon_special} is this kernel with $\tau=\hsgn{n}$:
\eq
\label{kernel}
K_n\left(\theta_i, \theta_j; t\right):=K_n(\theta_i,\theta_j; t,T,\mu,\hsgn{n}).
\endeq
Even though only the special case $\tau=\hsgn{n}$ in \eqref{tau-kernel} defines a correlation kernel, the $\tau$-deformed kernel was useful in \cite{LiechtyW:2016} for keeping track of the winding numbers in the model. 

\subsection{Distribution of winding numbers in \NIBMT\ }

In the process \NIBMT, let $\mathcal{W}_n(T,\mu)$ be the total winding number of the $n$ particles.
The distribution of $\mathcal{W}_n(T,\mu)$ can also be expressed in terms of the orthogonal polynomials \eqref{eq:defn_of_discrete_Gaussian_OP}. Introduce the Hankel determinant
\begin{equation}
\Hankel_n(T,\mu,\tau) := \det \left( \frac{1}{n} \sum_{x \in \lattice} x^{j + k - 2} e^{-\frac{Tn}{2}(x^2-2i\mu x)} dx \right)^n_{j, k = 1}.
\end{equation}
It is a standard result (see e.g. \cite{BleherL:2014}) that $\Hankel_n(T,\mu,\tau)$ is given in terms of the normalizing constants $h_{n,j}^{(T,\mu,\tau)}$ in \eqref{eq:defn_of_discrete_Gaussian_OP} as
\begin{equation}
\label{Hn-ito-hn}
\Hankel_n(T,\mu,\tau) := \prod_{j=0}^{n-1} h_{n,j}^{(T,\mu,\tau)}.
\end{equation}
The distribution of $\mathcal{W}_n(T,\mu)$ is given by
\begin{equation}
\label{eq:total_offset_formula}
\Prob(\mathcal{W}_n(T,\mu)=\om) = e^{2\pi i\omega\hsgn{n}} \int_0^1  \frac{\Hankel_n(T,\mu,\tau)}{\Hankel_n(T,\mu,\hsgn{n})} e^{-2\pi i \om \tau}d\tau, \quad \omega\in\mathbb{Z}.
\end{equation}
This formula is presented in \cite[Equation (185)]{LiechtyW:2016} in the case $\mu=0$, and its extension to general $\mu$ is straightforward.

We plot $\Prob(\mathcal{W}_3(1,\mu)=\omega)$, $\omega=0,...,3$ and
$\Prob(\mathcal{W}_6(1,\mu)=\omega)$, $\omega=0,...,6$ in Figure 
\ref{winding-plots}.  
\begin{figure}[h]
\begin{center}
\includegraphics[height=2.1in]{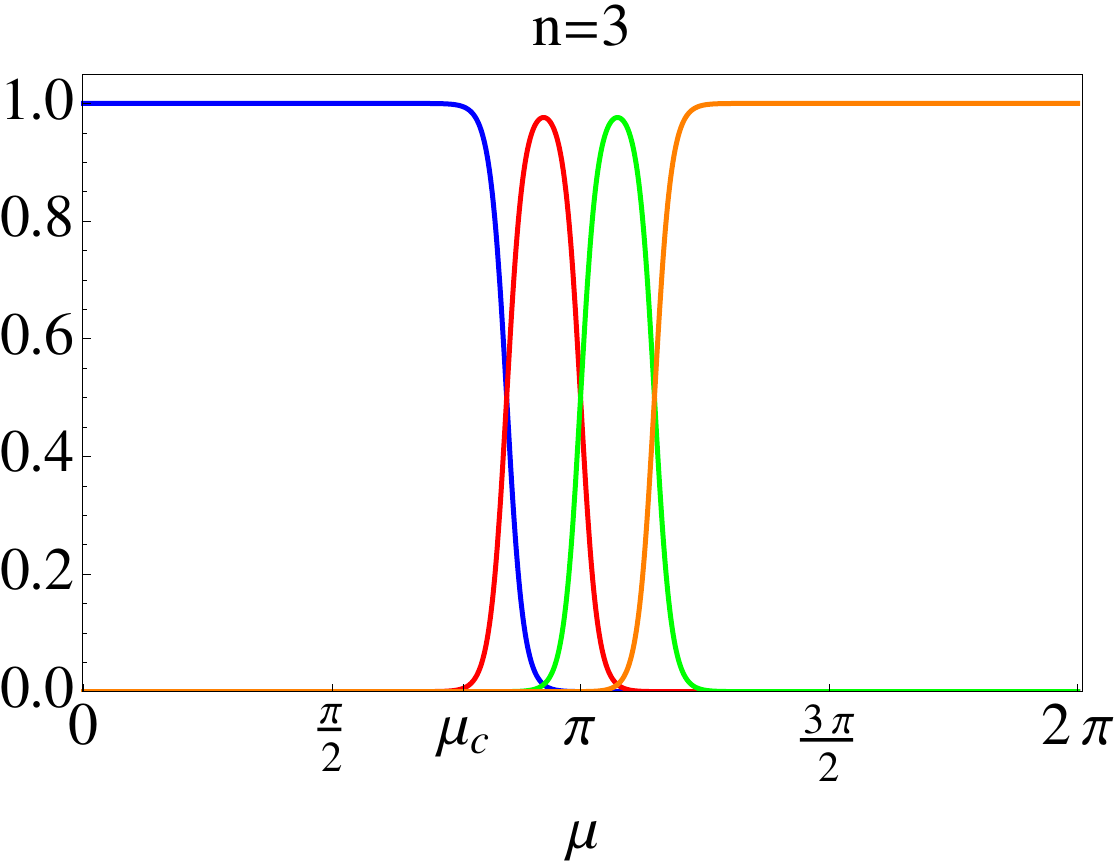}
\includegraphics[height=2.1in]{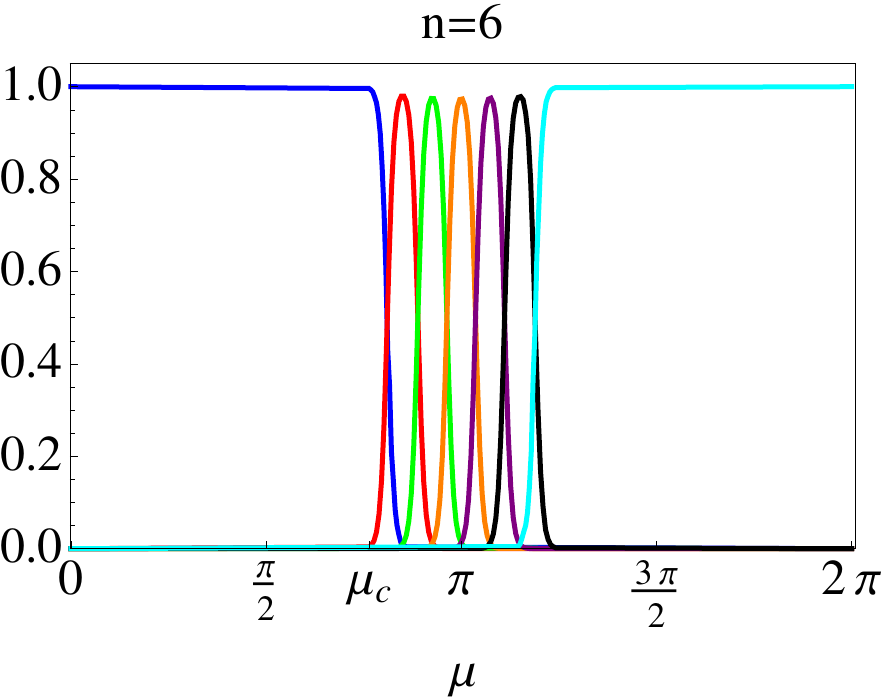}
\caption{The probability of each winding number as a function of $\mu$ for 
$n=3$ and $n=6$.  
Left plot: $\Prob(\mathcal{W}_3(1,\mu)=\omega)$ for $\omega=0$ (blue), 
1 (red), 2 (green), and 3 (orange).  Right plot:  
$\Prob(\mathcal{W}_6(1,\mu)=\omega)$ for $\omega=0$ (blue), 1 (red), 
2 (green), 3 (orange), 4 (purple), 5 (black), and 6 (cyan).  The critical 
drift $\mu_c(1)\approx 2.4016$ is indicated, with $\mu_c(T)$ defined in 
\eqref{mu-crit}.
}
\label{winding-plots}
\end{center}
\end{figure}
We prove in Theorem \ref{thm-subcrit-winding} that for each $T\in (0,\pi^2)$ there is a critical drift value $\mu_c(T)$ such that 
(asymptotically as $n\to\infty$) the expected winding is zero for 
$-\mu_c(T)<\mu<\mu_c(T)$.  The exact formula for $\mu_c(T)$ is given in 
\eqref{mu-crit}.  From plots such as those in Figure \ref{winding-plots}, it is 
possible to formulate further conjectures concerning the expected winding number 
for other values of $\mu$ and $T$.  For $\mu_c(T)<\mu<\frac{2\pi}{T}-\mu_c(T)$, 
there appears to be a transition region in which the expected winding number 
increases from 0 to $n$.  After this, the expected winding number appears to be 
(asymptotically as $n\to\infty$) $n$ for 
$\frac{2\pi}{T}-\mu_c(T)<\mu<\frac{2\pi}{T}+\mu_c(T)$.  This pattern appears to 
continue with the expected winding number increasing by $n$ when $\mu$ is 
increased by $\frac{2\pi}{T}$.  This conjectured behavior is illustrated in 
Figure \ref{winding-regions}.
\begin{figure}[h]
\begin{center}
\includegraphics[height=2.1in]{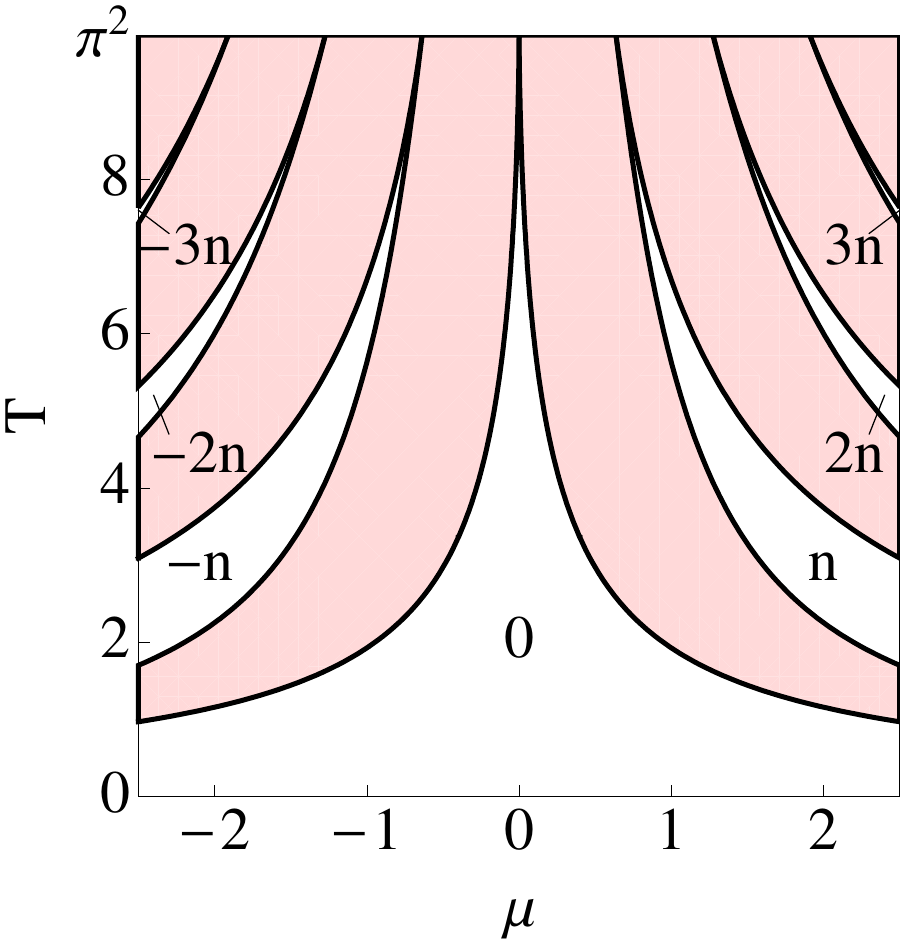}
\caption{The conjectured expected winding number (in the large-$n$ limit) for 
different values of $\mu$ and $T$.  The white regions have constant expected 
winding as indicated, while the transition regions are shaded.  The figure 
could be extended arbitrarily far in either direction in $\mu$, but it is 
necessary to restrict $T$ to the interval $[0,\pi^2]$.  Theorem 
\ref{thm-subcrit-winding} establishes the behavior in the zero-winding 
region, while Theorem \ref{thm-hermite-winding} details the behavior at the 
edge of this region.}
\label{winding-regions}
\end{center}
\end{figure}

We now state our results on the winding numbers.  Fix $T\in (0,\pi^2)$ and define 
the critical drift 
\begin{equation}
\label{mu-crit}
\mu_c \equiv \mu_c(T) := \frac{\sqrt{\pi^2-T}}{T} - \frac{\log T}{2\pi} + \frac{\log(\pi-\sqrt{\pi^2-T})}{\pi} \quad (0<T<\pi^2).
\end{equation}
For $|\mu|<\mu_c$, the asymptotic distribution of the random variable $\mathcal{W}_n(T,\mu)$, which represents the total winding number of the particles in \NIBMT, is almost surely 1 in the limit $n\to\infty$.
\begin{theo}[Winding numbers in the subcritical regime]
\label{thm-subcrit-winding}
Fix a return time $T\in(0,\pi^2)$ and a drift $\mu$ such that $|\mu|<\mu_c(T)$, as 
defined in \eqref{mu-crit}.  Then there is a constant $c>0$ such that 
\eq
\Prob(\mathcal{W}_n(T,\mu) = \om) = \begin{cases} 1 + \mathcal{O}\left(e^{-cn}\right), & \omega=0, \\ \mathcal{O}\left(e^{-cn}\right), & \omega\neq 0.\end{cases}
\endeq
\end{theo}
We will prove Theorem \ref{thm-subcrit-winding} in 
\S\ref{subsec-subcrit-winding}.

We now consider what we call the \emph{Hermite regime} when $\mu$ is 
close to $\mu_c$ (but $T$ is still bounded away from $\pi^2$).  Specifically, we 
fix a non-negative integer $k$ and choose $\mu$ so that 
\eq
\label{k-condition}
\mu_c + \left(k-\frac{1}{2}\right)\frac{\log n}{2\pi n} < \mu \leq \mu_c + \left(k+\frac{1}{2}\right)\frac{\log n}{2\pi n}.
\endeq
The continuous 
Hermite polynomials will play an important role in the analysis in this situation.  
\begin{theo}[Winding numbers in the Hermite regime]
\label{thm-hermite-winding}
Fix a return time $T\in(0,\pi^2)$ and a non-negative integer $k$, and choose $\mu$ 
satisfying \eqref{k-condition}.  Then
\eq
\label{hermite-winding}
\Prob(\mathcal{W}_n(T,\mu) = \om) = \begin{cases} 
\displaystyle \frac{F_{k-1}^{-1}}{1+F_k+F_{k-1}^{-1}} + \mathcal{O}\left(\frac{1}{n}\right), & \omega=k-1, \\
\displaystyle \frac{1}{1 + F_k + F_{k-1}^{-1}} + \mathcal{O}\left(\frac{1}{n}\right), & \omega=k, \\ 
\displaystyle \frac{F_k}{1 + F_k + F_{k-1}^{-1}} + \mathcal{O}\left(\frac{1}{n}\right), & \omega=k+1, \vspace{.1in} \\ 
\displaystyle \mathcal{O}\left(\frac{1}{n}\right), & \text{otherwise}, \end{cases}
\endeq
where
\eq
\label{Fk-scalar-def}
F_j\equiv F_j(T,\mu,n) := \begin{cases} \displaystyle \frac{j!}{2^{j+1}\sqrt{\pi}} \left(\frac{T}{(2\pi)^{3/2}(\pi^2-T)^{1/4}}\right)^{2j+1}\frac{e^{2\pi n(\mu-\mu_c(T))}}{n^{j+(1/2)}}, & j\in\mathbb{N}, \\ 0, & j=-1.
\end{cases}
\endeq
Here $F_{k-1}=o(1)$ except at 
$\mu=\mu_c+(k-\frac{1}{2})\frac{\log n}{2\pi n}$, and $F_k=o(1)$ except at 
$\mu=\mu_c+(k+\frac{1}{2})\frac{\log n}{2\pi n}$.  
\end{theo}
Theorem \ref{thm-hermite-winding} is proved in \S\ref{subsec-hermite-winding}.

\begin{rmk}
In addition to information about the winding number, it is also natural to ask about the asymptotic distribution of particles in \NIBMT\ . Using the kernel \eqref{kernel}, one can approach global questions about the asymptotic distribution of particles at any fixed time or the limiting shape of the complex hull of the paths in space-time, as well as local questions about the statistics of particles in various scaling limits. For $\mu=0$, these questions were answered in \cite{LiechtyW:2014, LiechtyW:2016} by rewriting \eqref{kernel} as a double contour integral, using the asymptotic expansion of the orthogonal polynomials and their discrete Cauchy transforms obtained from Riemann--Hilbert analysis, and performing classical steepest-descent analysis on the double contour integrals. This approach should also apply for general $\mu$, and for $T<\pi^2$ and $|\mu|$ small enough, the approach of \cite{LiechtyW:2016} can be followed exactly to obtain identical results. Namely, the paths collect inside an ellipse in space-time, at any fixed time $t$ the density of particles converges to a semi-circle, and there are local scaling limits to the sine process in the bulk and to the Airy process at the edge. However, for certain values of $\mu<\mu_c(T)$, we have  a technical difficulty in carrying out this program. The reason is that, although the asymptotic formulas for the orthogonal polynomials $p^{(T,\mu,\tau)}_{n, n}(z)$ are simple, see Theorem \ref{thm:polynomials_asymptotics_sub} below, the asymptotic formulas for their discrete Cauchy transforms become more complicated in certain regions of the complex plane.
\end{rmk}

\subsection{Discrete orthogonal polynomials with complex weight}
We prove our results on winding numbers by asymptotically analyzing the 
(meromorphic) Riemann--Hilbert problem associated to the discrete orthogonal 
polynomials \eqref{eq:defn_of_discrete_Gaussian_OP}.  Orthogonal polynomials are 
a powerful tool for asymptotic analysis of nonintersecting particle systems (see 
the survey article \cite{Konig:2005} for an overview).  A non-exhaustive list 
of applications includes the analysis of nonintersecting random walks by 
orthogonal polynomials on the unit circle \cite{Baik:2000}, discrete Krawtchouk 
and Charlier polynomials \cite{KonigOR:2002}, and Stieltjes--Wigert polynomials 
\cite{BaikS:2007}; nonintersecting Brownian motions on the line with multiple 
starting and ending points by multiple orthogonal polynomials \cite{DaemsK:2007};
nonintersecting Brownian bridges on the line with distinct drifts by biorthogonal 
Stieltjes--Wigert polynomials \cite{TakahashiK:2012}; and widths of 
nonintersecting Brownian bridges by Hermite polynomials and their discrete Gaussian counterparts \cite{BaikL:2014}.

Discrete Gaussian orthogonal polynomials with a real weight are the main tool in the 
analysis of the zero-drift version of the process \NIBMT\ \cite{LiechtyW:2016}.  
Discrete orthogonal polynomials display some remarkably different properties from 
their continuous analogues.  One well-known example for real weights is  
\emph{saturation of zeros}, which follows from the upper limit on the 
density of zeros due to the fact that there can be at most one zero between any 
two lattice points on which the weight is supported (see, e.g., 
\cite{BaikKMM:2007,Liechty:2012,LiechtyW:2016}).  Saturation occurs in the 
zero-drift case for $T>\pi^2$ and corresponds to the two edges of the sea of 
walkers touching at $\theta=\pi$ (see the third panel in Figure 
\ref{random-walk-plots}).  Saturation also plays a role for $T>\pi^2$ when the 
drift is non-zero, although that is beyond the scope of this paper.  Instead, we 
describe and analyze a new band formation phenomenon.  For sufficiently small 
drift values the locus of accumulating zeros is a single band, but at the
critical value $\mu_c(T)$ zeros also appear at a new point (each new zero 
corresponds to an increase in the expected winding number).  Similar behavior 
arises in the so-called \emph{birth of a cut} 
for random matrix eigenvalues, which has been studied using continuous orthogonal 
polynomials with a real but non-convex weight 
\cite{Claeys:2008,BertolaL:2009,Mo:2008}.  In the Riemann--Hilbert analysis it 
is necessary to insert a local parametrix built from continuous Hermite 
polynomials and modify the outer parametrix.  This construction also arises 
for the solitonic edge of an oscillatory zone in the solution of the 
small-dispersion Korteweg--de Vries equation \cite{ClaeysG:2010}, for a 
librational-rotational transition region in the solution of the small-dispersion 
sine-Gordon equation \cite{BuckinghamM:2012}, and at the edge of the pole region 
for rational Painlev\'e-II functions \cite{BuckinghamM:2015}.

This new phenomenon arises from the particular combination of a complex weight 
and discrete orthogonal polynomials.  No band opens at a single point in the 
zero-drift problem analyzed by discrete Gaussian orthogonal polynomials with a real weight 
\cite{LiechtyW:2016}.  To see that this behavior also depends on the fact that 
we are using discrete orthogonal polynomials, 
consider the continuous analogue of \eqref{eq:defn_of_discrete_Gaussian_OP}. That is, let $H_{n,k}^{(T,\mu)}(x)$ be the monic polynomial of degree $k$ satisfying 
\eq\label{def:continuous_OPs}
\int_\realR H_{n,k}^{(T,\mu)}(x)H_{n,j}^{(T,\mu)}(x)e^{-\frac{nT}{2}(x^2-2i\mu x)}dx = h_{n,k}^{(T,\mu)} \de_{jk} ,
\endeq
for a sequence $\{h_{n,k}^{(T,\mu)}\}_{k=0}^\infty$ of normalizing constants. For $\mu=0$ these are the rescaled monic versions of the classical Hermite polynomials. For general $\mu$, we can complete the square in the exponent to obtain
\eq
\int_{\realR} H_{n,k}^{(T,\mu)}(x)H_{n,j}^{(T,\mu)}(x)e^{-\frac{nT}{2}(x-i\mu)^2}dx = e^{nT\mu^2/2}h_{n,k}^{(T,\mu)} \de_{jk} ,
\endeq
or equivalently
\eq
\int_{\realR-i\mu} H_{n,k}^{(T,\mu)}(x+i\mu)H_{n,j}^{(T,\mu)}(x+i\mu)e^{-\frac{nT}{2}x^2}dx = e^{nT\mu^2/2}h_{n,k}^{(T,\mu)} \de_{jk} .
\endeq
Using Cauchy's theorem, the contour of integration can be deformed back to the real line, giving the orthogonality condition for $H_{n,k}^{(T,0)}(x)$. We thus find that the orthogonal polynomials \eqref{def:continuous_OPs} with complex weight are related to the ones with real weight ($\mu=0$) simply by a shift:
\eq
H_{n,k}^{(T,\mu)}(x) \equiv H_{n,k}^{(T,0)}(x-i\mu).
\endeq
Therefore the complexified weight cannot lead to new behavior for the continuous 
Hermite polynomials.  This argument relies on a deformation of the contour of 
integration, which clearly is not possible for the discrete orthogonal 
polynomials \eqref{eq:defn_of_discrete_Gaussian_OP}.  For discrete orthogonal 
polynomials with the complexified weight we will show in Theorem 
\ref{thm:polynomials_asymptotics_sub} an analogous relation 
holds, but only for sufficiently small drift (see \eqref{pnn-pnn-Hnn}).  The asymptotics of 
$p^{(T,\mu,\tau)}_{n,n}(z)$ near the critical drift are given in Theorem 
\ref{thm:polynomials_asymptotics_crit}.  


The current work is, to the best of our knowledge, the first investigation 
of discrete orthogonal polynomials with a complex weight.  However, there are 
examples of asymptotic studies of continuous orthogonal polynomials with complex 
weights, including a weight supported on a real compact interval \cite{Deano2014}
and on arcs in the complex plane \cite{HuybrechsKL:2014}.  In both cases band 
splitting was studied but the birth of a new band at a point was not observed.  
A case where a new band appears at a point is the analysis of orthogonal polynomials on the 
unit circle to study rational Painlev\'e-II solutions \cite{BertolaB:2015}.

Below we describe the asymptotic behavior of the polynomials $p^{(T,\mu,\tau)}_{n,n}(z)$ for $0<T<\pi^2$ as $n\to\infty$ in various regions of the complex plane separately in the cases $\mu<\mu_c(T)$ and $\mu\approx\mu_c(T)$. In this paper we do not consider non-scaling regimes when $\mu>\mu_c(T)$, which are considerably more complicated and can involve elliptic functions. We first introduce a few notations. Set  
\eq\label{eq:LWg-definition}
g_0(z):=\frac{T}{4}z\left(z-\sqrt{z^2-\frac{4}{T}}\right)-\log\left(z-\sqrt{z^2-\frac{4}{T}}\right) - \frac{1}{2} + \log\frac{2}{T}, \qquad \ell_0 := -1-\log T,
\endeq
where the principal branches of the logarithm and square root are chosen.  
These are the $g$-function and Lagrange multiplier associated to the rescaled continuous Hermite 
polynomials \cite{DeiftKMVZ:1999} and are also used for the Brownian bridge on 
the circle when $\mu=0$ and $0<T<\pi^2$ \cite{LiechtyW:2016}.  
Now we define the $g$-function and the Lagrange multiplier for nonzero $\mu$ 
as
\eq\label{eq:def-g-function}
g(z):=g_0(z-i\mu), \quad \ell := \ell_0 -\frac{T}{2} \mu^2.
\endeq
Define the function $\ga(z)$ to be 
\eq
\label{gamma-def}
\gamma(z):=\left(\frac{z-a}{z-b}\right)^{1/4}, \quad \textrm{where} \quad a:=-\frac{2}{\sqrt{T}}+i\mu, \quad b:=\frac{2}{\sqrt{T}}+i\mu,
\endeq
%
with a cut on the horizontal line segment $(a,b)$, taking the branch such that $\ga(\infty) = 1$. Since both $g(z)$ and $\ga(z)$ have  cuts on $(a,b)$, for $z\in(a,b)$ we define $g_\pm(z)$ and $\ga_\pm(z)$ to be the limiting values of these functions as $z$ approaches $(a,b)$ from above (respectively, below) the 
interval.

\begin{theo}[Orthogonal polynomial asymptotics in the subcritical regime]
\label{thm:polynomials_asymptotics_sub}
Let $T\in (0,\pi^2)$ be fixed, and $\mu\in \realR$ with $|\mu|<\mu_c(T)$, where $\mu_c(T)$ is defined in \eqref{mu-crit}.
\begin{enumerate}
\item[(a)] Let $z$ be bounded away from the band $(a,b)$. As $n\to\infty$, the polynomial $p^{(T,\mu,\tau)}_{n, n}(z)$ satisfies
\eq
\label{sub-crit_asymptotics_outer}
p^{(T,\mu,\tau)}_{n, n}(z) = e^{ng(z)}\left(\frac{\ga(z)+\ga(z)^{-1}}{2}\right)\left(1+\bigO(n^{-1})\right),
\endeq
and the error is uniform on subsets of $\C$ that are bounded away from the interval $(a,b)$.
\item[(b)] Let $z$ be in a compact subset of the band $(a,b)$. As $n\to\infty$, the polynomial $p^{(T,\mu,\tau)}_{n, n}(z)$ satisfies
\eq
\label{sub-crit_asymptotics_band}
p^{(T,\mu,\tau)}_{n, n}(z) =\left[ e^{ng_+(z)}\left(\frac{\ga_+(z)+\ga_+(z)^{-1}}{2}\right)+e^{ng_-(z)}\left(\frac{\ga_-(z)+\ga_-(z)^{-1}}{2}\right)\right]\left(1+\bigO(n^{-1})\right),
\endeq
and the error is uniform on compact subsets of $(a,b)$. 
\end{enumerate}
\end{theo}
Theorem \ref{thm:polynomials_asymptotics_sub} is proved in 
\S\ref{subsec-op-subcrit}.  

\begin{rmk}
Notice that the functions involved in Theorem 
\ref{thm:polynomials_asymptotics_sub} are the same ones involved in the asymptotics of $p^{(T,\mu,\tau)}_{n, n}(z)$ with $\mu=0$, but shifted by $i\mu$. Indeed, this theorem implies that as $n\to\infty$,
\eq
\label{pnn-pnn-Hnn}
 p^{(T,\mu,\tau)}_{n, n}(z)={p^{(T,0,\tau)}_{n, n}(z-i\mu)}(1+\bigO(n^{-1})) = H_{n,n}^{(T,\mu)}(z)(1+\bigO(n^{-1})),
\endeq
provided $0<T<\pi^2$ and $|\mu|<\mu_c$. In Theorem 
\ref{thm:polynomials_asymptotics_crit}, we see that this no longer holds when 
$\mu$ is scaled close to $\mu_c$. 
\end{rmk}
We introduce the function
\eq\label{def:D}
D(z;\mu,T):=\frac{\sqrt{T}}{2i}(z-i\mu+R(z;\mu,T)),
\endeq
where 
\eq\label{def:R}
R(z;\mu,T):=((z-a(\mu,T))(z-b(\mu,T)))^{1/2}
\endeq
is chosen with branch cut $[a,b]$ so that $R(z)=z+\mathcal{O}(1)$ as 
$z\to\infty$.   Also introduce the point on the imaginary axis
\eq
\label{zc-def}
\widetilde{z}_c(\mu,T):=i\mu - \frac{2i}{T}\sqrt{\pi^2-T},
\endeq
the constant
\eq\label{def:alpha_infty}
\alpha_\infty:=\frac{-i}{D(\widetilde{z}_c)} = \frac{\sqrt{\pi^2-T}+\pi}{\sqrt{T}}i,
\endeq 
and the function
\eq\label{def:alpha}
\alpha(z;\mu,T) := i\frac{D(\widetilde{z}_c;\mu,T) - D(z;\mu,T)}{1+D(\widetilde{z}_c;\mu,T)D(z;\mu,T)}.
\endeq
Finally, let $\mathfrak{h}_k(\zeta)$ be the (orthonormal) Hermite polynomial 
of degree $k$, where $k$ is a non-negative integer.  The Hermite polynomials 
satisfy 
\eq
\label{cont-hermite-def}
\int_\mathbb{R}\mathfrak{h}_j(\zeta)\mathfrak{h}_k(\zeta)e^{-z^2}d\zeta = \delta_{jk},
\endeq
and $\mathfrak{h}_0(\zeta)\equiv 1$.  Denote the leading coefficient by 
$\kappa_k$:  
\eq
\label{eq:hk}
\mathfrak{h}_k(\zeta) = \kappa_k\zeta^k + \mathcal{O}(\zeta^{k-1}), \quad \kappa_k := \frac{2^{k/2}}{\pi^{1/4}\sqrt{k!}}.
\endeq

\begin{theo}[Orthogonal polynomials in the Hermite regime]
\label{thm:polynomials_asymptotics_crit}
Fix $T\in(0,\pi^2)$ and a non-negative integer $k$, and choose $\mu$ satisfying 
\eqref{k-condition}.  
\begin{enumerate}
\item[(a)] Let $z$ be bounded away from the band $(a,b)$ and the point $\widetilde{z}_c(\mu,T)$. As $n\to\infty$, the polynomial $p^{(T,\mu,\tau)}_{n, n}(z)$ satisfies
\eq
\label{hermite_asymptotics_outer}
p^{(T,\mu,\tau)}_{n, n}(z) = e^{ng(z)}\left(\frac{\alpha(z)}{\alpha_\infty}\right)^k\left(\frac{\ga(z)+\ga(z)^{-1}}{2}\right)\left(1 + \mathcal{O}\left(\frac{e^{2\pi n(\mu-\mu_c)}}{n^{k+\frac{1}{2}}}\right) + \mathcal{O}\left(\frac{n^{k-\frac{1}{2}}}{e^{2\pi n(\mu-\mu_c)}} \right) \right),
\endeq
and the error is uniform on subsets of $\C$ which are bounded away from the interval $(a,b)$ and the point $\widetilde{z}_c(\mu,T)$.
\item[(b)] Let $z$ be in a compact subset of the band $(a,b)$. As $n\to\infty$, the polynomial $p^{(T,\mu,\tau)}_{n, n}(z)$ satisfies
\eq
\label{hermite_asymptotics_band}
\begin{split}
p^{(T,\mu,\tau)}_{n, n}(z) = & \left[\frac{\alpha_+(z)^k}{\alpha_\infty^k} e^{ng_+(z)}\left(\frac{\ga_+(z)+\ga_+(z)^{-1}}{2}\right)+\frac{\alpha_-(z)^k}{\alpha_\infty^k}e^{ng_-(z)}\left(\frac{\ga_-(z)+\ga_-(z)^{-1}}{2}\right)\right] \\ 
  & \times \left(1 + \mathcal{O}\left(\frac{e^{2\pi n(\mu-\mu_c)}}{n^{k+\frac{1}{2}}}\right) + \mathcal{O}\left(\frac{n^{k-\frac{1}{2}}}{e^{2\pi n(\mu-\mu_c)}} \right) \right),
\end{split}
\endeq
and the error is uniform on compact subsets of $(a,b)$. 
\item[(c)] There is a neighborhood $\mathbb{D}_{\widetilde{z}_c}$ of 
$\widetilde{z}_c$ and a local coordinate $W(z)$ (defined in \eqref{phi-to-W}) 
such that, for $z\in\mathbb{D}_{\widetilde{z}_c}$,
\eq
\label{hermite_asymptotics_zc}
\begin{split}
p^{(T,\mu,\tau)}_{n, n}(z) = & e^{ng(z)}\left(\frac{\alpha(z)}{\alpha_\infty W(z)}\right)^k\left(\frac{\gamma(z)+\gamma(z)^{-1}}{2}\right)\frac{\mathfrak{h}_k(n^{1/2}W(z))}{\kappa_k n^{k/2}} \\
  & \times \left(1 + \mathcal{O}\left(\frac{e^{2\pi n(\mu-\mu_c)}}{n^{k+\frac{1}{2}}}\right) + \mathcal{O}\left(\frac{n^{k-\frac{1}{2}}}{e^{2\pi n(\mu-\mu_c)}} \right) \right).
\end{split}
\endeq
\end{enumerate}
\end{theo}
Theorem \ref{thm:polynomials_asymptotics_crit} is proved in \S\ref{subsec-hermite-norms}.

\begin{rmk}
The complicated form of the error term arises from the fact that, although $\mu$ 
is a continuous parameter, the asymptotic behavior in the Hermite regime is 
naturally discretized in the sense that the choice of interval in 
\eqref{k-condition} determines the qualitative behavior.  This discretization is 
visually evident for the winding numbers in Figure \ref{winding-plots}, and for 
the orthogonal polynomials it manifests itself in the number of outlying zeros 
(see Remark \ref{rmk-zeros}).  Simply keeping the leading-order terms as we have 
done in Theorem \ref{thm:polynomials_asymptotics_crit} has the advantage of 
giving simpler formulas with more straightforward derivations.  However, it 
necessarily gives non-uniform errors due to the different behavior on different 
intervals.  In fact, the first error term is actually $\mathcal{O}(1)$ if 
$\mu = \mu_c + (k+\frac{1}{2})\frac{\log n}{2\pi n}$ and the second error term 
is $\mathcal{O}(1)$ if $\mu = \mu_c + (k-\frac{1}{2})\frac{\log n}{2\pi n}$.  
A uniform expression can be obtained if desired using the analysis in this paper 
by keeping additional terms, and we carry out this procedure completely in 
Theorems \ref{thm-hermite-winding} and \ref{thm-hermite-norms}.
\end{rmk}
\begin{rmk}
\label{rmk-zeros}
Equation \eqref{hermite_asymptotics_zc} shows that 
$p^{(T,\mu,\tau)}_{n, n}(z)$ asymptotically has $k$ zeros (possibly counting 
multiplicity) near $\widetilde{z}_c$ in the Hermite regime.  Indeed, 
$\alpha(z)$ and $W(z)$ both have simple zeros at $\widetilde{z}_c$ (see 
\eqref{alpha-on-W-at-zc}), so the term
$$e^{ng(z)}\left(\frac{\alpha(z)}{\alpha_\infty W(z)}\right)^k\left(\frac{\gamma(z)+\gamma(z)^{-1}}{2}\right)$$ 
is well-defined and non-zero in a neighborhood of $\widetilde{z}_c$.  
Furthermore, $\mathfrak{h}_k$ has exactly $k$ simple zeros on the real line, 
and $W$ is a conformal map in this neighborhood, so 
$\mathfrak{h}_k(n^{1/2}W(z))$ also has $k$ simple zeros in a neighborhood 
of $\widetilde{z}_c$ for $n$ large enough.
\end{rmk}

\begin{rmk}
In part (b) of Theorems \ref{thm:polynomials_asymptotics_sub} and \ref{thm:polynomials_asymptotics_crit} all of the functions involved in the asymptotic formulas have analytic extensions into a neighborhood of the band $(a,b)$, so these asymptotics can be extended to a neighborhood of compact subsets of $(a,b)$.
\end{rmk}

\begin{rmk}
Neither Theorem \ref{thm:polynomials_asymptotics_sub} nor \ref{thm:polynomials_asymptotics_crit} describe the asymptotics behavior of the orthogonal polynomials close to the band endpoints $a$ and $b$. Near these points the asymptotic behavior is described by Airy functions. We omit this analysis, which is standard, from the current paper. 
\end{rmk}

The orthogonal polynomials \eqref{eq:defn_of_discrete_Gaussian_OP} satisfy the three term recurrence equation (see \cite{Szego:1975})
\begin{equation} \label{eq:three_term_recurrence}
xp^{(T,\mu,\tau)}_{n, k}(x)=p^{(T,\mu,\tau)}_{n, k+1}(x)+\be^{(T,\mu,\tau)}_{n, k}p^{(T,\mu,\tau)}_{n, k}(x)+\left(\ga^{(T,\mu,\tau)}_{n, k}\right)^2p^{(T,\mu,\tau)}_{n, k-1}(x)\,,
\end{equation}
where $\{\be^{(T,\mu,\tau)}_{n, j}\}_{j=0}^\infty$ is a sequence of constants, and
\begin{equation} \label{eq:defn_of_gamma_nk}
\left(\ga^{(T,\mu,\tau)}_{n, j}\right)^2 := \frac{h^{(T,\mu,\tau)}_{n, j}}{h^{(T,\mu,\tau)}_{n, j-1}}.
\end{equation}
The asymptotic behavior of the recurrence coefficients and normalizing constants is presented in the last two theorems.

\begin{theo}[Normalizing constants and recurrence coefficients in the subcritical regime]
\label{thm-normalizations-subcrit}
Let $T\in (0,\pi^2)$ be fixed, and $\mu\in \realR$ with $|\mu|<\mu_c(T)$.
The normalizing constants and recurrence coefficients satisfy the following 
asymptotics as $n\to\infty$.
\eq	
\label{h-subcrit}
  h^{(T,\mu,\tau)}_{n, n} = \frac{2\pi }{T^{n+\frac{1}{2}}e^{n(\frac{1}{2}T\mu^2+1)}}(1+\bigO(n^{-1})), \qquad   \left(h^{(T,\mu,\tau)}_{n, n-1}\right)^{-1} = \frac{T^{n-\frac{1}{2}}e^{n(\frac{1}{2}T\mu^2+1)}}{2\pi}(1+\bigO(n^{-1})),
\endeq
\eq	
\be^{(T,\mu,\tau)}_{n, n-1} = i\mu+\bigO(n^{-1}), \qquad \left(\ga^{(T,\mu,\tau)}_{n, n}\right)^2 = \frac{1}{T}+\bigO(n^{-1}).
\endeq
\end{theo}
Theorem \ref{thm-normalizations-subcrit} is proved in 
\S\ref{subsec-op-subcrit}.

We now define 
\eq
\label{lambda-def}
\lambda\equiv\lambda(T):=i\frac{\pi+\sqrt{\pi^2-T}+i\sqrt{T}}{\pi-\sqrt{\pi^2-T}-i\sqrt{T}}
\endeq
and
\eq
\label{Gk-def}
G_j \equiv G_j(T,\mu,\tau,n) := (-1)^{n+1}F_j(T,\mu,n)e^{2\pi i\tau}, \quad j\in\{-1\}\cup\mathbb{N}. 
\endeq
Then the asymptotic behavior of the normalization constants in the Hermite regime 
is as follows.
\begin{theo}[Normalizing constants and recurrence coefficients in the Hermite regime]
\label{thm-hermite-norms}
Fix $T\in(0,\pi^2)$ and a non-negative integer $k$, and choose $\mu$ satisfying 
\eqref{k-condition}.  Then 
\eq
\label{hnn-asymptotics}
h_{n,n}^{(T,\mu,\tau)} =\left\{
\begin{aligned}
& \frac{2\pi }{T^{n+\frac{1}{2}}e^{n(\frac{1}{2}T\mu^2+1)}}\left(1 + \frac{2\pi}{\sqrt{T}}\frac{G_k}{1+G_k}\lambda + \frac{2\pi}{\sqrt{T}}\frac{G_{k-1}^{-1}}{1+G_{k-1}^{-1}}\lambda^{-1} + \mathcal{O}(n^{-1})\right)\alpha_\infty^{-2k}, \ &{\rm for} \ k>0, \\
 & \frac{2\pi }{T^{n+\frac{1}{2}}e^{n(\frac{1}{2}T\mu^2+1)}}\left(1 + \frac{2\pi}{\sqrt{T}}\frac{G_k}{1+G_k}\lambda + \mathcal{O}(n^{-1})\right)\alpha_\infty^{-2k}, \ &{\rm for} \ k=0,
 \end{aligned}
 \right.
\endeq
\eq
\label{hnnm1-asymptotics}
\begin{split}
&\left(h_{n,n-1}^{(T,\mu,\tau)}\right)^{-1} = \\
&\,\,\,\,\left\{
\begin{aligned}
&\frac{T^{n-\frac{1}{2}}e^{n(\frac{1}{2}T\mu^2+1)}}{2\pi } \left(1 +\frac{2\pi}{\sqrt{T}}\frac{G_k}{1+G_k}\lambda^{-1} +\frac{2\pi}{\sqrt{T}}\frac{G_{k-1}^{-1}}{1+G_{k-1}^{-1}}\lambda + \mathcal{O}(n^{-1})\right)\alpha_\infty^{2k},  \ &{\rm for} \ k>0, \\
& \frac{T^{n-\frac{1}{2}}e^{n(\frac{1}{2}T\mu^2+1)}}{2\pi } \left(1 +\frac{2\pi}{\sqrt{T}}\frac{G_k}{1+G_k}\lambda^{-1}  + \mathcal{O}(n^{-1})\right)\alpha_\infty^{2k},  \ &{\rm for} \ k=0, \\
 \end{aligned}
 \right.
\end{split}
\endeq
\begin{multline}
\label{gamma1-asymptotics}
 \left(\ga^{(T,\mu,\tau)}_{n, j}\right)^2 =\frac{1}{T} \left(1 +\frac{2\pi}{\sqrt{T}}\left(\frac{G_k}{1+G_k}+\frac{G_{k-1}^{-1}}{1+G_{k-1}^{-1}}\right)(\lambda+\lambda^{-1})\right. \\
\left. +\frac{4\pi^2}{T}\left(\left(\frac{G_k}{1+G_k}\right)^2+\left(\frac{G_{k-1}^{-1}}{1+G_{k-1}^{-1}}\right)^2\right) + \mathcal{O}(n^{-1})\right).
\end{multline}
Here $G_{k-1}^{-1}=o(1)$ except at 
$\mu=\mu_c+(k-\frac{1}{2})\frac{\log n}{2\pi n}$, and $G_k=o(1)$ except at 
$\mu=\mu_c+(k+\frac{1}{2})\frac{\log n}{2\pi n}$.  
\end{theo}
If desired, the asymptotic behavior of $\beta_{n,n-1}^{(T,\mu,\tau)}$ can also be 
extracted from the analysis in this paper, although we do not write down the 
lengthy but straightforward derivation.  Theorem \ref{thm-hermite-norms} is 
proved in \S\ref{subsec-hermite-norms}.

\subsection{Outline and notation}
From \eqref{eq:total_offset_formula}, the distribution of winding numbers is 
expressed in terms of the Hankel determinant $\mathcal{H}_n(T,\mu,\tau)$.  
We start in \S\ref{sec-integral-hankel} by representing the logarithm of this 
determinant as an integral involving the coefficient of the $x^{n-1}$ term of the 
orthogonal polynomial $p_{n,n}^{(T,\mu,\tau)}(x)$.  This connection allows us 
to use asymptotic analysis of the orthogonal polynomials to obtain information 
on the winding numbers.  In \S\ref{sec-rhp} we pose the Riemann--Hilbert problem 
encoding the discrete orthogonal polynomials and carry out several standard steps 
in the nonlinear steepest-descent analysis, namely interpolation of the poles, 
introduction of the $g$-function, and opening of lenses.  In 
\S\ref{sec-subcrit-winding} we complete the analysis in the subcritical case and 
compute the winding numbers and orthogonal polynomial asymptotics for 
$|\mu|<\mu_c$.  Finally, in \S\ref{sec-hermite-analysis} we consider the 
double-scaling limit $n\to\infty$, $\mu\to\mu_c$ and compute the winding 
numbers and orthogonal polynomial asymptotics in the Hermite regime.  

\noindent
{\it Notation}.  
With the exception of 
\eq
\mathbb{I}:=\bbm 1 & 0 \\ 0 & 1 \ebm, \quad \sigma_1:=\bbm 0 & 1 \\ 1 & 0 \ebm, \quad \sigma_3:=\bbm 1 & 0 \\ 0 & -1 \ebm,
\endeq
matrices are denoted by bold capital letters.  The non-negative integers are 
denoted by $\mathbb{N}$.  
We denote the $(jk)$th entry of a matrix ${\bf M}$ by $[{\bf M}]_{jk}$.
In reference to a smooth, oriented contour $\Sigma$, for $z\in\Sigma$ we denote by 
$f_+(z)$ (respectively, $f_-(z)$) the non-tangential limit of $f(\zeta)$ as 
$\zeta$ approaches $\Sigma$ from the left (respectively, the right).


%
%
%
%

\section{Integral representation for the Hankel determinant $\Hankel_n(T,\mu,\tau)$}
\label{sec-integral-hankel}
We begin by deriving a differential identity satisfied by the Hankel 
determinant $\Hankel_n(T,\mu,\tau)$. Introduce the notation for the 
coefficients of the orthogonal polynomials
\eq
p^{(T,\mu,\tau)}_{n, k}(x) = x^k + \sum_{j=0}^{k-1} c^{(T,\mu,\tau)}_{n, k,j} x^j.
\endeq
We then have the following proposition.
\begin{prop} The Hankel determinant $\Hankel_n(T,\mu,\tau)$ satisfies the 
differential equation
\begin{equation} \label{eq:diff_tau0}
\frac{\d}{\d \tau} \log \Hankel_n(T,\mu,\tau)  = inT\mu + Tc^{(T,\mu,\tau)}_{n, n,n-1}.
\end{equation}
\end{prop}
\begin{proof}

By \eqref{Hn-ito-hn},
\begin{equation} \label{eq:diff_tau1}
\frac{\d}{\d \tau} \log \Hankel_n(T,\mu,\tau) = \sum_{j=0}^{n-1} \frac{\frac{\d}{\d \tau}h^{(T,\mu,\tau)}_{n,j}}{h^{(T,\mu,\tau)}_{n,j}}.
\end{equation}
Therefore we compute $ \frac{\d}{\d \tau} h^{(T,\mu,\tau)}_{n, k}$.
We have
\begin{equation} \label{eq:diff_tau2}
  h^{(T,\mu,\tau)}_{n, k} = \frac{1}{n} \sum_{\ell \in \Z} p^{(T,\mu,\tau)}_{n, k}\left(\frac{\ell+\tau}{n}\right)^2 e^{-\frac{T}{2n}\left(\ell+\tau\right)^2}e^{iT\mu \left(\ell+\tau\right)}.
\end{equation}
Writing $x_\ell = \frac{\ell+\tau}{n}$ and taking the derivative with respect to $\tau$ gives
\begin{equation} \label{eq:diff_tau3}
\begin{aligned}
 \frac{\d}{\d \tau} h^{(T,\mu,\tau)}_{n, k} &= \frac{1}{n} \sum_{\ell \in \Z}\left[\frac{2}{n} p^{(T,\mu,\tau)}_{n, k}(x_\ell)\left(\frac{\d}{\d\tau}p^{(T,\mu,\tau)}_{n, k}(x_\ell)\right) -Tx_\ell + iT\mu \right]e^{-\frac{T}{2n}\left(\ell+\tau\right)^2}e^{iT\mu \left(\ell+\tau\right)} \\
 & = \frac{1}{n} \sum_{x \in \lattice}\left[\frac{2}{n} p^{(T,\mu,\tau)}_{n, k}(x)\left(\frac{\d}{\d\tau}p^{(T,\mu,\tau)}_{n, k}(x)\right) +(iT\mu-Tx)p^{(T,\mu,\tau)}_{n, k}(x)^2  \right]e^{-\frac{nT}{2}(x^2 - 2i\mu x)} \\
 & = iT\mu h^{(T,\mu,\tau)}_{n, k} -\frac{T}{n} \sum_{x \in \lattice} xp^{(T,\mu,\tau)}_{n, k}(x)^2e^{-\frac{nT}{2}(x^2 - 2i\mu x)} \\
  & = iT\mu h^{(T,\mu,\tau)}_{n, k} -Th^{(T,\mu,\tau)}_{n, k}  \be^{(T,\mu,\tau)}_{n, k}.
 \end{aligned}
\end{equation} 
Combining this with \eqref{eq:diff_tau1} gives 
\begin{equation} \label{eq:diff_tau4}
\frac{\d}{\d \tau} \log \Hankel_n(T,\mu,\tau) = \sum_{j=0}^{n-1}(iT\mu  -T  \be^{(T,\mu,\tau)}_{n, j}). 
\end{equation}
Comparing coefficients of the $x^k$ term in the three-term recurrence \eqref{eq:three_term_recurrence}, we find that
\eq
 \be^{(T,\mu,\tau)}_{n, k} = c^{(T,\mu,\tau)}_{n, k,k-1} - c^{(T,\mu,\tau)}_{n, k+1,k}.
\endeq
It follows that the second term in the right-hand side of 
\eqref{eq:diff_tau4} gives a telescoping sum, and we arrive at 
the formula \eqref{eq:diff_tau0}.
\end{proof}
We note the differential identity in 
\cite[Proposition 4.1]{LiechtyW:2016} used to study the zero-drift case 
involves two derivatives.  The fact that this differential identity
involves only one derivative simplifies the analysis somewhat.  
Integrating, we obtain the following proposition which, along with \eqref{eq:total_offset_formula}, we will subsequently use to prove Theorems \ref{thm-subcrit-winding} and \ref{thm-hermite-winding}.
\begin{prop}
\label{Hankel-integral-prop}
The Hankel determinant $\Hankel_n(T,\mu,\tau)$ has the integral 
representation 
\eq
\log \left(\frac{\Hankel_n(T,\mu,\tau) }{\Hankel_n(T,\mu,\hsgn{n})}\right) = \int_{\hsgn{n}}^\tau \left(inT\mu + Tc^{(T,\mu,v)}_{n, n,n-1}\right)\,dv.
\endeq
\end{prop}
\begin{proof}
This follows immediately from \eqref{eq:diff_tau0} via 
\eq
\int_{\hsgn{n}}^\tau \left(inT\mu + Tc^{(T,\mu,v)}_{n, n,n-1}\right)\,dv =\int_{\hsgn{n}}^\tau \frac{\d}{\d v} \log \Hankel_n(T,\mu,v) \,dv =  \log \left(\frac{\Hankel_n(T,\mu,\tau) }{\Hankel_n(T,\mu,\hsgn{n})}\right).
\endeq
\end{proof}

\section{Formulation and initial analysis of the Riemann--Hilbert problem}
\label{sec-rhp}

In this section we express the discrete orthogonal polynomials in terms of a 
Riemann--Hilbert problem, and then carry out several changes of variables that will 
be used in all the following analysis. For ease of exposition we assume throughout that $\mu> 0$, although the analysis is similar for $\mu<0$.

\subsection{The discrete Gaussian orthogonal polynomial Riemann--Hilbert problem}

The orthogonal polynomials \eqref{eq:defn_of_discrete_Gaussian_OP} are encoded 
in the following meromorphic Riemann--Hilbert problem.
\begin{rhp}[Discrete Gaussian orthogonal polynomial problem]
\label{rhp-dop}
Fix $n\in\{1,2,\ldots\}$ and find a $2\times 2$ matrix-valued 
function $\mathbf P_n(z)$ with the 
following properties:
\begin{itemize}
\item[]{\bf Analyticity:} $\mathbf P_n(z)$ is a meromorphic function of $z$ 
and is analytic for $z\in\C\setminus L_{n,\tau}$.
\item[]{\bf Normalization:}  There exists a function $r(x)>0$ on  $L_{n,\tau}$ 
such that 
\begin{equation} \label{IP2a}
\lim_{x\to\infty} r(x)=0,
\end{equation} 
and such that as $z\to\infty$, $\mathbf P_n(z)$ admits the asymptotic 
expansion
\begin{equation} \label{IP2}
\mathbf P_n(z) = \left(\mathbb{I}+\frac{\mathbf P_{n,1}}{z}+\frac{\mathbf P_{n,2}}{z^2}+ \mathcal{O}\left(\frac{1}{z^3}\right) \right)
\begin{pmatrix}
z^n & 0 \\
0 & z^{-n}
\end{pmatrix},\quad z\in \C\setminus \left[\bigcup_{x\in \lattice}^\infty D\big(x,r(x)\big)\right],
\end{equation}
where $D(x,r(x))$ denotes a disk of radius $r(x)>0$ centered at $x$.
\item[]{\bf Residues at poles:}  At each node $x\in L_{n,\tau}$, the elements 
$[\mathbf P_n(z)]_{11}$ and $[\mathbf P_n(z)]_{21}$ of the matrix 
$\mathbf P_n(z)$ 
are analytic functions of $z$, and the elements $[\mathbf P_n(z)]_{12}$ and
$[\mathbf P_n(z)]_{22}$ have a simple pole with the residues
\begin{equation} \label{IP1}
\underset{z=x}{\rm Res}\; [\mathbf P_n(z)]_{j2}=\frac{1}{n} e^{-\frac{nT}{2}(x^2-2i\mu x)} [\mathbf P_n(x)]_{j1},\quad j=1,2.
\end{equation}
\end{itemize}
\end{rhp}
The unique solution to Riemann--Hilbert Problem \ref{rhp-dop} 
(see \cite{FokasIK:1991,BleherL:2011}) is 
\begin{equation} \label{IP3}
\mathbf P_n(z) :=
\begin{pmatrix}
p_{n,n}^{(T,\mu,\tau)}(z) & \left(Cp_{n,n}^{(T,\mu,\tau)}\right)(z) \\
(h^{(T,\mu,\tau)}_{n,n-1})^{-1}p_{n,n-1}^{(T,\mu,\tau)}(z) & (h^{(T,\mu,\tau)}_{n,n-1})^{-1}\left(Cp_{n,n-1}^{(T,\mu,\tau)}\right)(z)
\end{pmatrix},
\end{equation}
where the weighted discrete Cauchy transform $C$ is 
\eq
\label{eq:def_Cauchy_trans}
Cf(z):=\frac{1}{n}\sum_{x\in L_{n,\tau}}\frac{f(x)e^{-\frac{nT}{2}(x^2-2i\mu x)}}{z-x}.
\endeq
The normalizing constants in \eqref{eq:defn_of_discrete_Gaussian_OP} and the recurrence 
coefficients \eqref{eq:three_term_recurrence} can be extracted from the 
matrices $\mathbf P_{n,1}$ and $\mathbf P_{n,2}$ in the expansion \eqref{IP2}.  
Specifically, 
\begin{equation}\label{IP4}
h^{(T,\mu,\tau)}_{n, n}=[\mathbf P_{n,1}]_{12}\,, \quad \left(h^{(T,\mu,\tau)}_{n, n-1}\right)^{-1}=[\mathbf P_{n,1}]_{21}\,,
\end{equation}
\begin{equation}\label{IP5}
\be^{(T,\mu,\tau)}_{n, n-1}=\frac{[\mathbf P_{n,2}]_{21}}{[\mathbf P_{n,1}]_{21}}-[\mathbf P_{n,1}]_{11}\,,
\end{equation}
and
\begin{equation}\label{IP5a}
c^{(T,\mu,\tau)}_{n, n,n-1} = [\mathbf P_{n,1}]_{11}.
\end{equation}
\begin{lem}
\label{winding-probs-ito-P}
Writing ${\bf P}_{n,1}\equiv{\bf P}_{n,1}(T,\mu,\tau)$, we have 
\eq
\Prob(\mathcal{W}_n(T,\mu)=\om) = e^{2\pi i\omega\hsgn{n}} \int_0^1  
\exp\left(\int_{\hsgn{n}}^\tau \left(inT\mu + T 
[\mathbf P_{n,1}(T,\mu,v)]_{11}
\right)\,dv\right)
e^{-2\pi i \om \tau}d\tau.
\endeq
\end{lem}
\begin{proof}
Simply combine \eqref{eq:total_offset_formula}, Proposition 
\ref{Hankel-integral-prop}, and \eqref{IP5a}.
\end{proof}

\subsection{Interpolation of poles}
We now carry out a sequence of changes of variables
$${\bf P}_n \to {\bf R}_n \to {\bf T}_n \to {\bf S}_n$$
in order to reduce Riemann--Hilbert Problem \ref{rhp-dop} to one that can 
approximated by exactly solvable problems.  In the first step we interpolate 
the poles, replacing the merophorphic function ${\bf P}_n$ with the 
sectionally analytic function ${\bf R}_n$.  
Define 
\begin{equation}
\mathbf D^u_{\pm}(z) := \begin{pmatrix} 1 & \displaystyle -\frac{1}{n\Pi(z)}e^{-nT(z^2-2i\mu z)/2}e^{\pm i\pi(n  z-\tau)} \\ 0 & 1 \end{pmatrix}, 
\end{equation}
where
\begin{equation}\label{red1}
\Pi(z):=\frac{\sin(n\pi z-\tau\pi)}{n\pi}.
\end{equation}
Fix two positive numbers $\epsilon<\mu$ and $\delta$.  Specifically, 
$\epsilon\equiv\epsilon(\mu,T)$ is chosen sufficiently small to satisfy 
\eqref{epsilon-condition1} and \eqref{epsilon-condition2}, while 
$\delta\equiv\delta(\mu,T)$ is defined by \eqref{eq:def-delta}.  Now set the 
matrix $\mathbf R_n(z)$ to be
\eq
\label{R-def}
\mathbf R_n(z) :=  \begin{cases} 
\begin{pmatrix} 1 & 0 \\ 0 & -2\pi i \end{pmatrix} \mathbf P_n(z) \mathbf D_+^u(z) \begin{pmatrix} 1 & 0 \\ 0 & -2\pi i \end{pmatrix}^{-1}, & \mu<\Im z<\mu+\epsilon, \\ 
\begin{pmatrix} 1 & 0 \\ 0 & -2\pi i \end{pmatrix} \mathbf P_n(z) \mathbf D_-^u(z) \begin{pmatrix} 1 & 0 \\ 0 & -2\pi i \end{pmatrix}^{-1}, & -\delta < \Im z < \mu, \\ 
\begin{pmatrix} 1 & 0 \\ 0 & -2\pi i \end{pmatrix} \mathbf P_n(z)\begin{pmatrix} 1 & 0 \\ 0 & -2\pi i \end{pmatrix}^{-1}, & \text{otherwise}.
\end{cases}
\endeq
Note that the singularities at the points $L_{n,\tau}$ are removable.  However, 
${\bf R}_n(z)$ now has jump discontinuities and satisfies the jump conditions
\eq
\mathbf R_{n+}(z) =  \mathbf R_{n-}(z) {\bf V}^{({\bf R})}(z),
\endeq
where 
\begin{equation}
{\bf V}^{({\bf R})}(z) := \begin{cases}
\begin{pmatrix} 1 & 0 \\ 0 & -2\pi i \end{pmatrix} \mathbf D^u_{+}(z)  \begin{pmatrix} 1 & 0 \\ 0 & -2\pi i \end{pmatrix}^{-1}, & z\in \realR + i(\mu+\ep), \\
\begin{pmatrix} 1 & e^{-nT(z^2-2i\mu z)/2} \\ 0 & 1 \end{pmatrix}, & z\in \realR+i\mu, \\
\begin{pmatrix} 1 & 0 \\ 0 & -2\pi i \end{pmatrix} \mathbf  D^u_{-}(z)  \begin{pmatrix} 1 & 0 \\ 0 & -2\pi i \end{pmatrix}^{-1}, & z\in \realR -i\delta,
\end{cases}
\end{equation}
and the orientation is left-to-right on $\realR+i\mu$ and $\realR-i\de$, but right-to-left on $\realR +i(\mu+\ep)$.

\subsection{The subcritical $g$-function}

A vital step in the nonlinear steepest-descent analysis is the introduction 
of an exponent $g(z)$ (traditionally called the \emph{$g$-function}) that 
in effect averages out the rapid oscillations in the jump and, after the 
subsequent lens-opening step, leaves constant jumps on one or more 
\emph{bands}.  The process of defining $g(z)$ involves determining these 
bands.  It turns out that in the subcritical case we can use a shifted 
version of the $g$-function $g_{0}(z)$ which appears in the asymptotic analysis of the rescaled monic Hermite polynomials (see 
\eqref{eq:def-g-function}).  Also recall the band endpoints $a$ and $b$ defined 
in \eqref{gamma-def}.  
We denote the potential $V(z)$ by
\eq
V(z):=\frac{Tz^2}{2}-iT\mu z.
\endeq
The function $g(z)$ satisfies the following properties.
\begin{enumerate}
\item[]{\bf Analyticity:} The function $g(z)$ is analytic for $\Im z\neq\mu$.
\item[]{\bf Normalization:}  As $z\to\infty$,
\begin{equation}
g(z) = \log(z) -\sum_{j=1}^\infty \frac{g_j}{z^j},
\end{equation}
where
\begin{equation}
g_j:= \int_{a}^{b} \frac{z^j}{j} d\nu(z)
\end{equation}
and $d\nu$ is the semi-circle law on the interval $[a,b]$, i.e.,
\eq\label{def:semi-circle}
d\nu(x+i\mu) = \frac{T}{2\pi} \sqrt{\frac{4}{T} - x^2}\,dx  \quad {\rm for} \  -\frac{2}{\sqrt{T}}<x<\frac{2}{\sqrt{T}}.
\endeq
In particular we have 
\eq
\label{g1}
g_1=i\mu.
\endeq
\item[]{\bf Variational condition:}  On the line $\Im z=\mu$,
\begin{equation} \label{eq:variational_condition}
  g_+(z) + g_-(z) - V(z) - \ell
  \begin{cases}
    = 0, & z\in\{\Im z = \mu\}\cap\{|\Re z| \le 2/\sqrt{T}\}, \\
    < 0, & z\in\{\Im z = \mu\}\cap\{|\Re z| > 2/\sqrt{T}\},
  \end{cases}
\end{equation}
where the Lagrange multiplier $\ell$ is defined in \eqref{eq:LWg-definition} and \eqref{eq:def-g-function}.

\item[]{\bf Jump condition:} For $z\in [a,b]$, $g(z)$ satisfies the jump condition
\eq\label{eq:jump_for_g}
g_+(z) - g_-(z) = 2\pi i \int_z^b d\nu(w).
\endeq
\end{enumerate}

Notice that since the measure $\nu$ has a real analytic density \eqref{def:semi-circle}, the right-hand-side of \eqref{eq:jump_for_g} extends to an analytic function in a neighborhood of the band $[a,b]$. For $z$ in this neighborhood, we will denote
\eq\label{eq:def-G}
 G(z):= g_+(z) - g_-(z).
\endeq
We can now define the matrix $\mathbf T_n(z)$ via the transformation
\eq
\label{T-def}
\mathbf T_n(z) := e^{-n\ell\sg_3/2} \mathbf R_n(z) e^{-n(g(z) - \ell/2)\sg_3}.
\endeq
It satisfies the jump conditions
\[ \mathbf T_{n+}(z) =  \mathbf T_{n-}(z) {\bf V}^{({\bf T})}(z),\]
where 
\begin{equation}
{\bf V}^{({\bf T})}(z) := \begin{cases} 
\begin{pmatrix} 1 & \displaystyle\frac{1}{2\pi in \Pi(z)}e^{-nT(z^2-2i\mu z)/2}e^{n(2g(z)-\ell)}e^{i\pi(nz-\tau)} \\ 0 & 1 \end{pmatrix}, & z\in \realR + i(\mu+\ep), \\
\begin{pmatrix} e^{-n(g_+(z) - g_-(z))} & e^{-nT(z^2-2i\mu z)/2}e^{n(g_+(z) + g_-(z)-\ell)} \\ 0 & e^{n(g_+(z) - g_-(z))} \end{pmatrix}, & z\in \realR+i\mu, \\
\begin{pmatrix} 1 & \displaystyle\frac{1}{2\pi in \Pi(z)}e^{-nT(z^2-2i\mu z)/2}e^{n(2g(z)-\ell)}e^{-i\pi(nz-\tau)} \\ 0 & 1 \end{pmatrix}, & z\in \realR -i\delta. 
\end{cases}
\end{equation}

\subsection{Opening of the lenses}

Recall the choices of $\epsilon$ and $\delta$ used to specify ${\bf R}_n(z)$ 
and define the rectangular lens regions
\eq
\Omega_+:= \left\{\left(-\frac{2}{\sqrt{T}},\frac{2}{\sqrt{T}}\right)\times(i\mu,i(\mu+\epsilon))\right\}, \quad \Omega_-:= \left\{\left(-\frac{2}{\sqrt{T}},\frac{2}{\sqrt{T}}\right)\times(i(\mu-\epsilon),i\mu)\right\}
\endeq
(see Figure \ref{fig-S-jumps}).
Now the function $g_+(z)-g_-(z)$ defined for $z\in[a,b]$ (recall we orient the 
band left-to-right) has an analytic extension to 
$\overline{\Omega_+\cup\Omega_-}$.  Adding \eqref{eq:variational_condition} and \eqref{eq:def-G}, we find that $G(z)$   can be written as
\eq
G(z) = \begin{cases} 2g(z)-V(z)-\ell, & z\in\Omega_+, \\ -2g(z)+V(z)+\ell, & z\in\Omega_-.  \end{cases}
\endeq
\begin{figure}[h]
\setlength{\unitlength}{1.5pt}
\begin{center}
\begin{picture}(100,80)(-50,-20)
\thicklines
\put(-80,65){\line(1,0){160}}
\put(-80,45){\line(1,0){160}}
\put(-50,25){\line(1,0){100}}
\multiput(-80,0)(4,0){40}{\line(1,0){2}}
\put(-80,-15){\line(1,0){160}}
\put(-50,25){\line(0,1){40}}
\put(50,25){\line(0,1){40}}
\put(82,-2){$\mathbb{R}$}
\put(82,43){$\mathbb{R}+i\mu$}
\put(82,63){$\mathbb{R}+i(\mu+\epsilon)$}
\put(82,-17){$\mathbb{R}-i\delta$}
\put(-50,45){\circle*{2}}
\put(50,45){\circle*{2}}
\put(-48,47){$a$}
\put(45,47){$b$}
\put(-5,53){$\Omega_+$}
\put(-5,33){$\Omega_-$}
\put(-5,18){$\sigma_-$}
\put(0,65){\vector(-1,0){5}}
\put(62,45){\vector(1,0){5}}
\put(-68,45){\vector(1,0){5}}
\put(67,65){\vector(-1,0){5}}
\put(-63,65){\vector(-1,0){5}}
\put(-4,45){\vector(1,0){5}}
\put(-4,25){\vector(1,0){5}}
\put(-4,-15){\vector(1,0){5}}
\put(50,53){\vector(0,1){5}}
\put(-50,58){\vector(0,-1){5}}
\end{picture}
\end{center}
\caption{\label{fig-S-jumps} The lens regions $\Omega_\pm$ and the jump 
contours $\Sigma^{({\bf S})}$ along with their orientations.}
\end{figure}
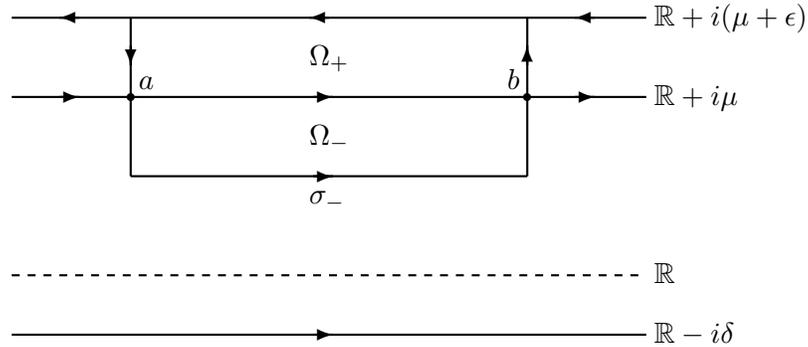

We now make the transformation
\begin{equation}\label{st2}
\mathbf S_n(z) := \begin{cases} 
\mathbf T_n(z)\bbm 1 & 0 \\ -e^{-nG(z)} & 1 \ebm, & z\in\Omega_+, \vspace{.05in} \\
\mathbf T_n(z)\bbm 1 & 0 \\ e^{nG(z)} & 1 \ebm, & z\in\Omega_-,\\
\mathbf T_n(z), & \textrm{otherwise}. \end{cases}
\end{equation}
The function ${\bf S}_n(z)$ has jumps on $\Sigma^{({\bf S})}$ (see 
Figure \ref{fig-S-jumps}) consisting of the three lines on which 
${\bf T}_n(z)$ has jumps along with the vertical line segments 
$(a+i\epsilon,a)$ and $(b,b+i\epsilon)$ as well as the contour
\eq
\sigma_-:=(a,a-i\epsilon)\cup(a-i\epsilon,b-i\epsilon)\cup(b-i\epsilon,b)
\endeq
forming part of the boundary of $\Omega_-$.
On this contour, the function $\mathbf S_n(z)$ satisfies the jump condition
\begin{equation}\label{st3}
\mathbf S_{n+}(z)=\mathbf S_{n-}(z) {\bf V}^{({\bf S})}(z),
\end{equation}
where
  \begin{equation}\label{st3a}
{\bf V}^{({\bf S})}(z):=
\begin{cases}
\bbm (1-e^{- 2i\pi(n z-\tau)})^{-1} &  \frac{e^{ nG(z)}}{1-e^{- 2\pi i( n z-\tau)}} \\ - e^{- nG(z)} & 1 \ebm, & z\in (a+i\epsilon,b+i\epsilon), \vspace{.03in} \\
\bbm 1 &  \frac{e^{n(2g(z)-\ell-V(z))}}{1-e^{- 2\pi i(nz-\tau)}} \\ 0 & 1 \ebm, & z\in (\mathbb{R}+i(\mu+\epsilon))\backslash(a+i\epsilon,b+i\epsilon), \vspace{.03in} \\
\bbm 1 & 0 \\ -e^{-nG(z)} & 1 \ebm, & z\in(a+i\epsilon,a)\cup(b,b+i\epsilon), \vspace{.03in} \\
\bbm 0 & 1 \\ -1 & 0 \ebm, & z\in (a,b),  \vspace{.03in} \\
\bbm 1 & e^{n(g_+(z)+g_-(z)-\ell-V(z))} \\ 0 & 1 \ebm, & z\in (\mathbb{R}+i\mu)\backslash(a,b), \vspace{.03in} \\
\bbm 1 &  0 \\ e^{nG(z)} & 1 \ebm, & z\in \sigma_-, \vspace{.03in} \\
\bbm 1 & -\frac{e^{n(2g(z)-V(z)-\ell)}}{1-e^{2\pi i(nz-\tau)}} \\ 0 & 1 \ebm, & z\in \realR -i\delta. 
\end{cases}
\end{equation}
The orientation of each contour is given in Figure \ref{fig-S-jumps}.
We claim that all of the jumps outside of the band $[a,b]$ are exponentially 
close to the identity matrix as $n\to\infty$.  The analysis on the jumps close 
to $\mathbb{R}+i\mu$ essentially follows from the zero-drift case 
\cite{LiechtyW:2016}.  However, on $\mathbb{R}-i\delta$ a different argument 
is required.  We begin with the following lemma.
\begin{lem}\label{lem:horizontal-line}
Fix $\delta$ to be
\eq\label{eq:def-delta}
\de\equiv\de(\mu,T):= -\mu + \frac{2}{T}\sqrt{\pi^2 - T},
\endeq
and define $\mu_c(T)$ by \eqref{mu-crit}.  Note that if $0<\mu<\mu_c$ then $\delta$ 
is positive.  Also define the phase function 
\eq\label{eq:def-phi}
\phi(z):= 2g(z) -V(z) - \ell -2\pi i z.
\end{equation}
Fix $T\in(0,\pi^2)$.  Then for all $0<\mu<\mu_c(T)$, we have 
$\Re \phi(z)\leq \Re \phi(-i\delta)<0$ for all $z\in \realR - i\delta$.
\end{lem}
\begin{proof}
We first prove that for all fixed $y\in \realR\setminus\{\mu\}$, the real part of $\phi(x+iy)$ attains its maximum at $x=0$. Using \eqref{eq:LWg-definition}--\eqref{eq:def-g-function} and \eqref{eq:def-phi}, after some simplifications we find that
\begin{equation}\label{eq:phi-deriv}
\frac{\d}{\d x} \phi(x-iy) = -T\sqrt{\left(x+i(y-\mu)\right)^2 - \frac{4}{T}}-2\pi i.
\end{equation}
The critical point of $\Re\phi(x+iy)$ is where the radicand is a negative number, which clearly occurs only at $x=0$. To see that this critical point is the location of a global maximum (in $x$ for fixed $y$), note that the branch of the square root in \eqref{eq:phi-deriv} is such that
\eq
\sqrt{(z-i\mu)^2 - \frac{4}{T}} = z+\bigO(1) \quad \textrm{as} \ z\to\infty.
\endeq
It follows that for large positive $x$, $\Re \frac{\d}{\d x} \phi(x+iy)  < 0$; and for large negative $x$, $\Re \frac{\d}{\d x} \phi(x+iy)  > 0$. Since there is only one critical point of $\Re\phi(x+iy)$ at $x=0$ it must be that $\Re \frac{\d}{\d x} \phi(x+iy)  < 0$ for all positive $x$ and $\Re \frac{\d}{\d x} \phi(x+iy)  > 0$ for all negative $x$, thus the maximum is attained at $x=0$.

To complete the proof of Lemma \ref{lem:horizontal-line} we must show that $\Re\phi(-i\delta)<0$, with $\delta$ given in \eqref{eq:def-delta}. This can be checked directly:
\eq
 \phi(-i\delta) = -\frac{2\pi \sqrt{\pi^2-T}}{T} + 2\pi \mu +\log(T) - 2\log(\pi-\sqrt{\pi^2-T})-i\pi,
 \endeq
 which clearly has negative real part for all $\mu<\mu_c(T)$.
\end{proof}

Let $\mathbb{D}_a$ and $\mathbb{D}_b$ be small fixed circular neighborhoods 
centered at $a$ and $b$, respectively.  The radii are chosen small enough so 
the closure of the disks do not intersect each other or 
$\mathbb{R}+i(\mu+\epsilon)$.

\begin{prop}
\label{prop-S-jump-bounds}
Fix $T\in(0,\pi^2)$ and $\mu\in[0,\mu_c(T))$.  Then there exists a constant 
$c>0$, independent of $n$ and $z$, such that as $n\to\infty$
\eq
{\bf V}^{({\bf S})}(z) = \mathcal{O}(e^{-cn}), \quad z\in\Sigma^{({\bf S})}\backslash([a,b]\cup\mathbb{D}_a\cup\mathbb{D}_b).
\endeq
\end{prop}
\begin{proof}
The variational conditions \eqref{eq:variational_condition} immediately imply this for the jumps on $(\mathbb{R}+i\mu)\backslash(a,b)$ and $(\mathbb{R}+i(\mu+\epsilon))\backslash(a+i\epsilon,b+i\epsilon)$. Now consider the lens boundaries $(a+i\epsilon,b+i\epsilon)$ and $\sigma_-$.
According to \eqref{eq:jump_for_g} and \eqref{eq:def-G} we have that for all $x\in[-\frac{2}{\sqrt{T}}, \frac{2}{\sqrt{T}}],$
\eq\label{eq:G-expansion}
G(x+i(\mu\pm \epsilon)) = 2\pi i \int_{x+i\mu}^b d\nu(w) \pm2\pi \ep \frac{d\nu}{d\sigma}\bigg(x+i\mu\bigg)+\bigO(\epsilon^2),
\endeq
where $\sigma$ is the Lebesgue measure on the interval $[a,b]$. 
The density $\frac{d\nu}{d\sigma}$ given in \eqref{def:semi-circle} is positive on the interior of $(a,b)$. Thus, 
$\epsilon$ may be chosen small enough so that
\eq 
\label{epsilon-condition1}
\pm \Re G(x+i(\mu\pm\epsilon))> 0,
\endeq
which shows that the jump on $\sigma_-$ is exponentially close to the identity matrix as $n\to\infty$,
and that the jump on $(a+i\epsilon,b+i\epsilon)$ is exponentially close to the identity matrix except perhaps in the $(12)$-entry. To examine this entry we consider
\eq 
 G(x+i(\mu+\epsilon))+2\pi i(x+i(\mu+\epsilon)) = 2\pi i\left(x+ \int_{x+i\mu}^b d\nu(w)\right) + 2\pi \ep\left( \frac{d\nu}{d\sigma}\bigg(x+i\mu\bigg)-1\right)+\bigO(\epsilon^2) 
 \endeq
for all $x\in \left[-\frac{2}{\sqrt{T}}, \frac{2}{\sqrt{T}}\right].$  Since we are in the subcritical regime $T<\pi^2$, the density $\frac{d\nu}{d\sigma}(x+i\mu)$ is strictly less than 1, see \eqref{def:semi-circle}. Thus $\ep$ may be chosen small enough so that
\eq 
\label{epsilon-condition2}
\Re [G(x+i(\mu\pm\epsilon))+2\pi i(x+i(\mu+\epsilon))] < 0,
\endeq
for all $x\in [-\frac{2}{\sqrt{T}}, \frac{2}{\sqrt{T}}].$ This proves that the jump on the upper lens is exponentially close to the identity function as $n\to\infty$.

Finally, the jump on the horizontal line $\realR - i\delta$ 
is exponentially close to the identity matrix as $n\to\infty$ since 
$\Re\phi(z)<0$ for all $z\in \realR-i\de$ by Lemma \ref{lem:horizontal-line}. 
\end{proof}

\section{Analysis in the subcritical regime}
\label{sec-subcrit-winding}

\subsection{The subcritical outer model problem}

While it is not possible to solve for ${\bf S}_n(z)$ exactly, it is true that 
the jumps ${\bf V}^{(\bf S)}(z)$ decay to the identity as $n\to\infty$ except on the 
band $[a,b]$.  It is therefore reasonable to expect that ${\bf S}_n(z)$ is 
approximated by the function (the outer model) with only this jump.  The 
leading-order error in this approximation can be traced to the other jumps in 
neighborhoods of the band endpoints $a$ and $b$, which decay but 
subexponentially.  In order to control this error it is necessary to 
construct functions in these neighborhoods (the local models) that satisfy 
exactly the same jumps as ${\bf S}_n(z)$.  A global model solution will be 
built from the outer and inner model solutions.  These steps are standard 
(see, for instance, \cite{DeiftKMVZ:1999,Liechty:2012}).

We denote the subcritical outer model solution by ${\bf M}_0(z)$.  The 
subscript 0 is because in the Hermite regime we will use a modified outer model 
solution ${\bf M}_k(z)$.
\begin{rhp}[Subcritical outer model problem]
\label{rhp-outer-subcrit}
Determine a $2\times 2$ matrix-valued function ${\bf M}_0^{(\rm{out})}(z)$ 
satisfying:
\begin{itemize}
\item[]{\bf Analyticity:} ${\bf M}_0^{(\rm{out})}(z)$ is analytic in $z$ 
off $[a,b]$ and is H\"older continuous up to $(a,b)$ with at worst quarter-root 
singularities at $a$ and $b$.  
\item[]{\bf Normalization:}  
\eq
\lim_{z\to\infty}{\bf M}_0^{(\rm{out})}(z) = \mathbb{I}.
\endeq
\item[]{\bf Jump condition:}  Orienting $[a,b]$ left-to-right, the solution 
satisfies 
\eq
\label{outer-band-jump}
{\bf M}_{0+}^{(\rm{out})}(z)={\bf M}_{0-}^{(\rm{out})}(z)\bbm 0 & 1 \\ -1 & 0 \ebm, \quad z\in[a,b].
\endeq
\end{itemize}
\end{rhp}
Recall $\gamma(z)$ defined in \eqref{gamma-def}.  Then
\eq
\label{M0-out}
{\bf M}_0^{(\text{out})}(z):= \bbm \displaystyle \frac{\gamma(z)+\gamma(z)^{-1}}{2} & \displaystyle \frac{\gamma(z)-\gamma(z)^{-1}}{-2i} \\ \displaystyle \frac{\gamma(z)-\gamma(z)^{-1}}{2i} & \displaystyle \frac{\gamma(z)+\gamma(z)^{-1}}{2} \ebm
\endeq
is the unique solution to Riemann--Hilbert Problem \ref{rhp-outer-subcrit}.  We 
note that $\det{\bf M}_0^{(\text{out})}(z)\equiv 1$.

\subsection{Airy parametrices}
Recall the small disks $\mathbb{D}_a$ and $\mathbb{D}_b$ introduced in 
Proposition \ref{prop-S-jump-bounds}.  We seek a local parametrix 
${\bf M}_0^{(a)}(z)$ 
defined in $\mathbb{D}_a$ with the following properties.
\begin{rhp}[The Airy parametrix Riemann--Hilbert problem]
\label{rhp-airy}
Determine 
a $2\times 2$ matrix-valued function ${\bf M}_0^{(a)}(z)$ satisfying:
\begin{itemize}
\item[]{\bf Analyticity:} ${\bf M}_0^{(a)}(z)$ is analytic in 
$\mathbb{D}_a\backslash\Sigma^{(\bf{S})}$.  In each wedge the solution can be 
analytically continued into a larger wedge, and is H\"older continuous up to 
the boundary in a neighborhood of $z=a$.
\item[]{\bf Normalization:}  
\eq
{\bf M}_0^{(a)}(z) = {\bf M}_0^{(\rm{out})}(z)(\mathbb{I}+\mathcal{O}(n^{-1})), \quad z\in\partial\mathbb{D}_a.
\endeq
\item[]{\bf Jump condition:}  For $z\in\mathbb{D}_a\cap\Sigma^{({\bf S})}$, 
${\bf M}_0^{(a)}(z)$ satisfies 
${\bf M}_{0+}^{(a)}(z) = {\bf M}_{0-}^{(a)}(z){\bf V}^{({\bf S})}(z)$.
\end{itemize}
\end{rhp}
The solution to this problem is standard (see, for example, 
\cite{DeiftKMVZ:1999,BleherL:2011,BuckinghamM:2014}) and is constructed from 
Airy functions.  
For our purposes it is sufficient to note that such a parametrix exists (and 
is unique), as the explicit form for the solution is only necessary for 
computing terms of $\mathcal{O}(n^{-1})$.  The corresponding Riemann--Hilbert 
problem for ${\bf M}_0^{(b)}(z)$ is obtained simply by replacing $a\to b$, 
and it also has a unique solution.

\subsection{The subcritical error problem}

The subcritical global model is
\eq
\label{M0-def}
{\bf M}^{(0)}(z):=\begin{cases} {\bf M}_0^\text{(out)}(z), & z\in\mathbb{C}\backslash(\mathbb{D}_a\cup\mathbb{D}_b), \\ {\bf M}_0^{(a)}(z), & z\in \mathbb{D}_a, \\ {\bf M}_0^{(b)}(z), & z\in \mathbb{D}_b.  \end{cases}
\endeq
The error function 
\eq
\label{X-subcrit}
{\bf X}_n(z) := {\bf S}_n(z){\bf M}^{(0)}(z)^{-1}
\endeq
measures how well the global model approximates the desired solution 
${\bf S}_n(z)$.
The function $\mathbf X_n(z)$ satisfies a RHP with jumps on the contour 
$\Sg^{({\bf X})}$ consisting of the circles $\partial\mathbb{D}_a$ and 
$\partial\mathbb{D}_b$, oriented counterclockwise, together with the parts of 
$\Sigma^{({\bf S})} \setminus [a,b]$ that lie outside of the disks 
$\mathbb{D}_a$ and $\mathbb{D}_b$.  See Figure \ref{fig-subcrit-error-jumps}.
By Proposition \ref{prop-S-jump-bounds} and the normalization condition in 
Riemann--Hilbert Problem \ref{rhp-airy}, the jumps for ${\bf X}_n(z)$ are 
uniformly within $\mathcal{O}(n^{-1})$ of the 
identity matrix, and $\mathbf X_n(\infty)=\mathbb{I}$.  
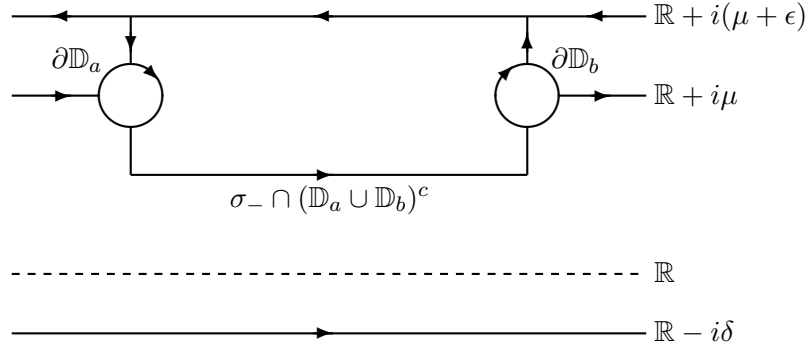
\begin{figure}[h]
\setlength{\unitlength}{1.5pt}
\begin{center}
\begin{picture}(100,80)(-50,-20)
\thicklines
\put(-80,65){\line(1,0){160}}
\put(-80,45){\line(1,0){22}}
\put(58,45){\line(1,0){22}}
\put(-50,25){\line(1,0){100}}
\multiput(-80,0)(4,0){40}{\line(1,0){2}}
\put(-80,-15){\line(1,0){160}}
\put(-50,25){\line(0,1){12}}
\put(-50,53){\line(0,1){12}}
\put(50,25){\line(0,1){12}}
\put(50,53){\line(0,1){12}}
\put(82,-2){$\mathbb{R}$}
\put(82,43){$\mathbb{R}+i\mu$}
\put(82,63){$\mathbb{R}+i(\mu+\epsilon)$}
\put(82,-17){$\mathbb{R}-i\delta$}
\put(-50,45){\circle{15}}
\put(50,45){\circle{15}}
\put(-70,52){$\partial\mathbb{D}_a$}
\put(56,52){$\partial\mathbb{D}_b$}
\put(-25,17){$\sigma_-\cap(\mathbb{D}_a\cup\mathbb{D}_b)^c$}
\put(1,65){\vector(-1,0){5}}
\put(66,45){\vector(1,0){5}}
\put(-70,45){\vector(1,0){5}}
\put(70,65){\vector(-1,0){5}}
\put(-65,65){\vector(-1,0){5}}
\put(-4,25){\vector(1,0){5}}
\put(-4,-15){\vector(1,0){5}}
\put(50,56){\vector(0,1){5}}
\put(-50,61){\vector(0,-1){5}}
\put(-44,50){\vector(1,-1){1}}
\put(45.5,51.5){\vector(1,1){1}}
\end{picture}
\end{center}
\caption{\label{fig-subcrit-error-jumps} The jump 
contours $\Sigma^{({\bf X})}$ along with their orientations.}
\end{figure}
The solution to the Riemann--Hilbert problem for ${\bf X}_n(z)$ is given 
explicitly in terms of a Neumann series.   $\mathbf X_n(z)$ satisfies 
\begin{equation}\label{tt19}
\mathbf X_n(z) = \mathbb{I} + \bigO\left(\frac{1}{n(|z|+1)}\right) \quad  \textrm{as} \quad n \to \infty
\end{equation}
uniformly for $z\in \C \setminus \Sigma^{({\bf X})}$, see e.g. \cite{BleherL:2014, DeiftKMVZ:1999}.

\subsection{Winding numbers in the subcritical regime}
\label{subsec-subcrit-winding}
We now unravel the steps of the steepest-descent analysis to recover 
$c^{(T,\mu,\tau)}_{n,n,n-1}$, which is encoded in the matrix 
$\mathbf P_{n,1}$ (recall \eqref{IP2}). We first prove the following lemma, which is the same as Theorem \ref{thm-subcrit-winding}, but with a weaker error, which is subsequently improved.
\begin{lem}
\label{lemma-subcrit-winding-algebraic-error}
Fix a return time $T\in(0,\pi^2)$ and a drift $\mu$ such that 
$|\mu|<\mu_c(T)$, as defined in \eqref{mu-crit}.  Then 
\eq
\Prob(\mathcal{W}_n(T,\mu) = \om) = \begin{cases} \displaystyle 1 + \mathcal{O}\left(\frac{1}{n}\right), & \omega=0, \\ \displaystyle \mathcal{O}\left(\frac{1}{n}\right), & \omega\neq 0.\end{cases}
\endeq
\end{lem}
\begin{proof}
For $z\in\mathbb{C}\backslash\{z|-\delta\leq \Im z \leq \mu+\epsilon\}$ we have 
\begin{equation}
\label{P-ito-XM-subcrit}
\mathbf P_n(z) =\bbm 1 & 0 \\ 0 & -2\pi i \ebm^{-1} e^{n\ell\sg_3/2} \mathbf X_n(z) \mathbf M^{(0)}(z) e^{n(g(z) - \ell/2)\sg_3} \bbm 1 & 0 \\ 0 & -2\pi i \ebm.
\end{equation}
Expanding at $z=\infty$ gives
\eq
\label{P-at-infinity}
\begin{split}
\mathbb{I}+\frac{\mathbf P_{n,1}}{z}+\mathcal{O}\left(\frac{1}{z^2}\right) = 
 & \bbm 1 & 0 \\ 0 & -2\pi i \ebm^{-1} e^{n\ell\sg_3/2} \left(\mathbb{I}+\frac{\mathbf X_{n,1}}{z}+\mathcal{O}\left(\frac{1}{z^2}\right)\right) \left(\mathbb{I}+\frac{\mathbf M_1}{z}+\mathcal{O}\left(\frac{1}{z^2}\right)\right) \\ 
 &  \times \exp\left(-n\left(\frac{g_1}{z}+\mathcal{O}\left(\frac{1}{z^2}\right)\right)\sg_3\right) e^{-n\ell\sg_3/2} \bbm 1 & 0 \\ 0 & -2\pi i \ebm.
\end{split}
\endeq
It is easy to check that $[\mathbf M_1]_{11} = 0$. We thus have that
\eq
\label{cnnnm1}
c^{(T,\mu,\tau)}_{n,n,n-1}=[\mathbf P_{n,1}]_{11} = -ng_1 + [\mathbf X_{n,1}]_{11} + [\mathbf M_1]_{11} = -in\mu+\bigO(n^{-1}),
\endeq
where we have used \eqref{IP5a} and \eqref{g1}.  Together, Proposition 
\ref{Hankel-integral-prop}, \eqref{IP5a}, and \eqref{cnnnm1} show 
\eq
\log \left(\frac{\Hankel_n(T,\mu,\tau) }{\Hankel_n(T,\mu,\hsgn{n})}\right) = \int_{\hsgn{n}}^\tau \left(inT\mu + T [\mathbf P_{n,1}(T,\mu,v)]_{11}\right)\,dv = \mathcal{O}\left(n^{-1}\right).
\endeq
Thus
\eq
\label{Hn-over-Hn}
\frac{\Hankel_n(T,\mu,\tau) }{\Hankel_n(T,\mu,\hsgn{n})} = 1 + \mathcal{O}\left(n^{-1}\right).
\endeq
From \eqref{eq:total_offset_formula},
\eq
\Prob(\mathcal{W}_n(T,\mu) = \om) = e^{2\pi i\hsgn{n}\omega} \int_0^1  e^{-2\pi i\tau\omega}d\tau + \mathcal{O}\left(n^{-1}\right) = \begin{cases} 1 + \mathcal{O}\left(n^{-1}\right), & \omega=0, \\ \mathcal{O}\left(n^{-1}\right), & \omega\neq 0,\end{cases}
\endeq
which establishes the lemma.
\end{proof}
To prove Theorem \ref{thm-subcrit-winding} we only need to improve the 
error term from $\mathcal{O}(n^{-1})$ to $\mathcal{O}(e^{-cn})$.  
\begin{proof}[Proof of Theorem \ref{thm-subcrit-winding} (Subcritical winding numbers)]
Let $H_{n,k}^{(T,\mu)}(z)$ be the monic polynomial of degree $k$ defined by the orthogonality condition
\eq
\int_{\realR+i\mu} H_{n,k}^{(T,\mu)}(z)H_{n,j}^{(T,\mu)}(z)e^{-\frac{nT}{2}(z^2-2i\mu z)}\,dz = h_{n,k}^{(T,\mu)} \de_{jk}.
\endeq
Notice that these are the same as the polynomials defined in 
\eqref{def:continuous_OPs}, even though the contour of integration is 
different (it can be deformed using Cauchy's theorem). These polynomials are 
encoded in the following Riemann--Hilbert problem.
\begin{rhp}[Shifted Hermite polynomial problem]
Find a $2\times 2$ matrix-valued 
function $\mathbf P_n^{(c)}(z)$ with the 
following properties:
\begin{itemize}
\item[]{\bf Analyticity:} $\mathbf P_n^{(c)}(z)$ is analytic for $z\in\C\setminus (\realR+i\mu)$.
\item[]{\bf Normalization:} As $z\to\infty$, $\mathbf P_n^{(c)}(z)$ admits the asymptotic 
expansion
\begin{equation} 
\mathbf P^{(c)}_n(z) = \left(\mathbb{I} + \mathcal{O}\left(\frac{1}{z}\right) \right)
\begin{pmatrix}
z^n & 0 \\
0 & z^{-n}
\end{pmatrix},\quad z\in \C\setminus (\realR+i\mu).
\end{equation}
\item[]{\bf Jump condition:}  On the line $\realR+i\mu$, the matrix function $\mathbf P_n^{(c)}(z)$ takes limiting values from above and below, and satisfies the jump condition
\begin{equation} 
\mathbf P_{n+}^{(c)}(z)=\mathbf P_{n-}^{(c)}(z)\begin{bmatrix}1 &  e^{-\frac{nT}{2}(z^2-2i\mu z)} \\ 0 & 1 \end{bmatrix}. 
\end{equation}
\end{itemize}
\end{rhp}
Proceeding with the steepest-descent analysis of this Riemann--Hilbert 
problem, we make the transformations
$${\bf P}^{(c)}_n\to {\bf T}^{(c)}_n \to {\bf S}^{(c)}_n$$
as
\eq
\mathbf T^{(c)}_n(z) := e^{-n\ell\sg_3/2}{\bf P}^{(c)}_n e^{-n(g(z) - \ell/2)\sg_3}
\endeq
and
\begin{equation}
\mathbf S^{(c)}_n(z) := \begin{cases} 
\mathbf T^{(c)}_n(z)\bbm 1 & 0 \\ -e^{-nG(z)} & 1 \ebm, & z\in\Omega_+, \vspace{.05in} \\
\mathbf T^{(c)}_n(z)\bbm 1 & 0 \\ e^{nG(z)} & 1 \ebm, & z\in\Omega_-,\\
\mathbf T^{(c)}_n(z), & \textrm{otherwise}, \end{cases}
\end{equation}
where the functions $g(z)$ and $G(z)$ are defined in 
\eqref{eq:def-g-function} and \eqref{eq:def-G}.
The function $\mathbf S^{(c)}_n(z)$ has exactly the same jumps as $\mathbf S_n(z)$ given in \eqref{st3a} except on the lines $\realR+i(\mu+\ep)$ and $\realR-i\delta$. On these lines there is no jump at all except on the interval $(a+i\ep, b+i\ep)$ (oriented right-to-left), where the jump is
\eq
\label{VXc}
 {\bf V}^{(\bf X^{(c)})}(z) = \begin{bmatrix} 1 & 0 \\ -e^{-nG(z)} & 1 \end{bmatrix}, \quad z\in (a+i\ep, b+i\ep).
\endeq
Since the jumps for $\mathbf S_n(z)$ on $\realR+i(\mu+\ep)$ and $\realR-i\delta$ are exponentially close (in $n$) to the identity matrix, as is the jump for $\mathbf S^{(c)}_n(z)$ on $(a+i\ep, b+i\ep)$, we use the same model solutions for both the continuous and discrete versions of the Riemann--Hilbert problem. That is, we let
\eq
 {\bf X}_n(z) := {\bf S}_n(z) {\bf M}^{(0)}(z)^{-1}, \qquad  {\bf X}^{(c)}_n(z) := {\bf S}^{(c)}_n(z) {\bf M}^{(0)}(z)^{-1},
\endeq
where ${\bf M}^{(0)}(z)$ is defined in \eqref{M0-def}. Then
\eq
 {\bf X}_n(z) = \mathbb{I} + \bigO(n^{-1}), \qquad  {\bf X}^{(c)}_n(z) = \mathbb{I} + \bigO(n^{-1}).
 \endeq
Now consider the ratio ${\bf Y}_n(z):=  {\bf X}_n(z)  {\bf X}^{(c)}_n(z)^{-1}$. It only has jumps on the lines $\realR+i(\mu+\ep)$ and $\realR-i\delta$. On these lines, its jumps are given by
\eq
{\bf Y}_{n+}(z) = {\bf Y}_{n-}(z){\bf V}^{(\bf Y)}(z),
\endeq
where
\eq
 {\bf V}^{(\bf Y)}(z) = \left\{
 \begin{aligned}
 & {\bf X}^{(c)}_n(z) {\bf V}^{(\bf S)}(z) {\bf X}^{(c)}_n(z)^{-1}, \quad z\in \{\realR-i\delta\} \cup \{\{\realR + i(\mu+\ep)\} \setminus (a+i\ep, b+i\ep)\}, \\
 &{\bf X}^{(c)}_{n-}(z) {\bf V}^{(\bf S)}(z) {\bf X}^{(c)}_{n+}(z)^{-1}, \quad z\in (a+i\ep, b+i\ep),
 \end{aligned}\right.
 \endeq
 and ${\bf V}^{(\bf S)}(z)$ is as defined in \eqref{st3a}.  Since ${\bf V}^{(\bf S)}(z)=  \mathbb{I} + \bigO(e^{-cn})$ on these contours for some $c>0$, and ${\bf X}^{(c)}_n(z)$ is bounded, we find that
 \eq
 {\bf V}^{(\bf Y)}(z) =
 \mathbb{I} + \bigO(e^{-cn}), \quad z\in \{\realR-i\delta\} \cup \{\{\realR + i(\mu+\ep)\} \setminus (a+i\ep, b+i\ep)\}.
 \endeq
 Similarly, on the line segment $(a+i\ep, b+i\ep)$ we have
  \eq
\begin{aligned}
 {\bf V}^{(\bf Y)}(z) &=
 {\bf X}^{(c)}_{n-}(z) {\bf X}^{(c)}_{n+}(z)^{-1}+ \bigO(e^{-cn}).
\end{aligned}
     \endeq
 To see that this is also close to the identity matrix, note that here the jump matrix for $ {\bf X}^{(c)}_{n}(z)$ is given by  
   \eq\label{VXc_est}
 {\bf X}^{(c)}_{n-}(z)^{-1} {\bf X}^{(c)}_{n+}(z) = {\bf V}^{(\bf X^{(c)})}(z) =  \mathbb{I} + \bigO(e^{-cn}),\quad z\in (a+i\ep, b+i\ep),
  \endeq
  where $ {\bf V}^{(\bf X^{(c)})}(z)$ is given in \eqref{VXc}.
   Again using that   ${\bf X}^{(c)}_{n\pm}(z)$ is bounded, we find that \eqref{VXc_est} implies
     \eq
 {\bf V}^{(\bf Y)}(z) = {\bf X}^{(c)}_{n-}(z) {\bf X}^{(c)}_{n+}(z)^{-1} =\mathbb{I} + \bigO(e^{-cn}),  \quad z\in (a+i\ep, b+i\ep).
  \endeq
 Since ${\bf Y}_n(z)$ has jumps that are exponentially close to the identity matrix as $n\to\infty$, the nonlinear steepest-descent small-norm theory 
for Riemann--Hilbert problems \cite{DeiftZ:1993} gives
 \eq
 {\bf Y}_n(z) =  \mathbb{I} + \bigO(e^{-cn}),
 \endeq
 or equivalently
  \eq\label{Xdiff}
  {\bf X}_n(z)  = {\bf X}^{(c)}_n(z)+ \bigO(e^{-cn}),
  \endeq
  where we again use that  ${\bf X}^{(c)}_n(z)$ is bounded.

 Now consider the coefficient $c_{n,n,n-1}^{(T, \mu, \tau)}$ used to compute the winding number probabilities. The formula for $c_{n,n,n-1}^{(T, \mu, \tau)}$ in terms of the RHP is given in \eqref{cnnnm1} as
$$c_{n,n,n-1}^{(T, \mu, \tau)} = -ng_1 + [{\bf M}_1]_{11} + [{\bf X}_1]_{11}.$$
Let $d_{n,n,n-1}^{(T, \mu)}$ be the analogous coefficient for the continuous 
orthogonal polynomials, i.e., 
 \eq
 H_{n,k}^{(T,\mu)}(z) = z^k + d_{n,n,n-1}^{(T, \mu)} z^{k-1} + \cdots.
 \endeq
 The steepest-descent analysis of the Riemann--Hilbert problem for the continuous orthogonal polynomials gives
 \eq
 d_{n,n,n-1}^{(T, \mu)} = -ng_1 + [{\bf M}_1]_{11} + [{\bf X}^{(c)}_1]_{11};
 \endeq
 thus we have that the difference $   c_{n,n,n-1}^{(T, \mu, \tau)} -   d_{n,n,n-1}^{(T, \mu)}$ is
 \eq
    c_{n,n,n-1}^{(T, \mu, \tau)}-    d_{n,n,n-1}^{(T, \mu)} = [{\bf X}_1]_{11}- [{\bf X}^{(c)}_1]_{11} = \bigO(e^{-cn}),
    \endeq
 using \eqref{Xdiff}. Since $ H_{n,k}^{(T,\mu)}(z)$ is simply a rescaled and shifted Hermite polynomial, its coefficients are known exactly. In particular,
  \eq
 d_{n,n,n-1}^{(T, \mu)} = -in\mu,
 \endeq
 so we have 
 \eq
     c_{n,n,n-1}^{(T, \mu, \tau)} = -in\mu+\bigO(e^{-cn})
     \endeq
     for some $c>0$. It follows that the errors in 
\eqref{cnnnm1}--\eqref{Hn-over-Hn} and in the result of 
Lemma \ref{lemma-subcrit-winding-algebraic-error} can be improved to $\bigO(e^{-cn})$.
\end{proof}

\subsection{Orthogonal polynomial asymptotics in the subcritical regime}
\label{subsec-op-subcrit}

We now prove our results on orthogonal polynomials in the subcritical regime.

\begin{proof}[Proof of Theorem \ref{thm:polynomials_asymptotics_sub} (Orthogonal polynomial asymptotics in the subcritical regime)]
Taking into account that ${\bf D}_+^u(z)$ and ${\bf D}_-^u(z)$ are 
upper-triangular, \eqref{IP3}, \eqref{R-def}, \eqref{T-def}, \eqref{st2}, and 
\eqref{X-subcrit} give 
\eq
\label{p-ito-XM0}
p_{n,n}^{(T,\mu,\tau)}(z) = [{\bf P}_n(z)]_{11} = [{\bf X}_n(z){\bf M}^{(0)}(z)]_{11}e^{ng(z)}
\endeq
for $z\in\mathbb{C}\backslash(\Omega_+\cup\Omega_-)$.  Then \eqref{M0-def} and 
\eqref{M0-out} show 
\eq
\label{M011}
[{\bf M}^{(0)}(z)]_{11} = \frac{\gamma(z)+\gamma(z)^{-1}}{2}, \quad z\in\mathbb{C}\backslash(\mathbb{D}_a\cup\mathbb{D}_b).
\endeq
Along with \eqref{tt19}, this proves \eqref{sub-crit_asymptotics_outer} and part 
(a).

Now consider $z$ on the band and bounded away from the endpoints $a$ and $b$.  
The calculation leading to \eqref{p-ito-XM0} must be modified to account for the 
change of variables inside the lenses, so 
\eq
p_{n,n}^{(T,\mu,\tau)}(z) = [{\bf X}_n(z){\bf M}_-^{(0)}(z)]_{11}e^{ng_-(z)} - [{\bf X}_n(z){\bf M}_-^{(0)}(z)]_{12}e^{ng_+(z)}.
\endeq
Using \eqref{tt19} gives 
\eq
\label{p-ito-M11-M12-sub}
p_{n,n}^{(T,\mu,\tau)}(z) = \left([{\bf M}_-^{(0)}(z)]_{11}e^{ng_-(z)} - [{\bf M}_-^{(0)}(z)]_{12}e^{ng_+(z)}\right)(1+\mathcal{O}(n^{-1})).
\endeq
Then \eqref{M0-def} and \eqref{outer-band-jump} show
\eq
\label{M012-to-M011}
[{\bf M}_-^{(0)}(z)]_{12} = [{\bf M}_{0-}^\text{(out)}(z)]_{12} = -[{\bf M}_{0+}^\text{(out)}(z)]_{11}.
\endeq
Together, \eqref{p-ito-M11-M12-sub}, \eqref{M012-to-M011}, and \eqref{M011} prove 
\eqref{sub-crit_asymptotics_band} and part (b).
\end{proof}

\begin{proof}[Proof of Theorem \ref{thm-normalizations-subcrit} (Subcritical normalizing constants and recurrence coefficients)]
To find the formulas for $h_{n,n}^{(T,\mu,\tau)}$ and 
$\left(h_{n,n-1}^{(T,\mu,\tau)}\right)^{-1}$ we start by combining \eqref{IP4} 
and \eqref{P-at-infinity}:
\eq
h_{n,n}^{(T,\mu,\tau)} = -2\pi i\left(\left[{\bf M}_1\right]_{12} + \left[{\bf X}_{n,1}\right]_{12} \right)e^{n\ell}, \quad \left(h_{n,n-1}^{(T,\mu,\tau)}\right)^{-1} = \frac{-1}{2\pi i}\left(\left[{\bf M}_1\right]_{21} + \left[{\bf X}_{n,1}\right]_{21} \right)e^{-n\ell}.
\endeq
For $|z|$ sufficiently large, ${\bf M}^{(0)}(z)={\bf M}_0^\text{(out)}(z)$ by 
\eqref{M0-def}.  Thus, by expanding the off-diagonal entries of \eqref{M0-out}, 
\eq
\label{M112-M121}
[{\bf M}_1]_{12} = \frac{i}{\sqrt{T}}, \quad [{\bf M}_1]_{21} = -\frac{i}{\sqrt{T}}.
\endeq
Along with \eqref{tt19}, \eqref{eq:LWg-definition}, and \eqref{eq:def-g-function}, the previous two equations establish \eqref{h-subcrit}.  
Then the equation for $\left(\gamma_{n,n}^{(T,\mu,\tau)}\right)^2$ follows 
immediately from \eqref{h-subcrit} and \eqref{eq:defn_of_gamma_nk}.  To analyze 
$\be^{(T,\mu,\tau)}_{n, n-1}$, we start with \eqref{IP5}.  We previously 
computed 
\eq
\label{P111-P221}
[{\bf P}_{n,1}]_{11} = -in\mu+\mathcal{O}(n^{-1}), \quad [{\bf P}_{n,1}]_{21} = \frac{e^{-n\ell}}{2\pi\sqrt{T}}(1+\mathcal{O}(n^{-1}))
\endeq
(see \eqref{IP5a}, \eqref{cnnnm1}, \eqref{IP4}, and \eqref{h-subcrit}).  To 
obtain $[{\bf P}_{n,2}]_{21}$ we continue the expansion \eqref{P-at-infinity} to 
$\mathcal{O}(z^{-2})$ and use \eqref{tt19} to find
\eq
[{\bf P}_{n,2}]_{21} = -\frac{e^{-n\ell}}{2\pi i}\left([{\bf M}_2]_{21} - ng_1[{\bf M}_1]_{21} + \mathcal{O}(n^{-1})\right)
\endeq
(here ${\bf M}_2$ is the coefficient of $z^{-2}$ in the expansion of 
${\bf M}_0^\text{(out)}(z)$ as $z\to\infty$).  We recall $g_1=i\mu$ from 
\eqref{g1} and $[{\bf M}_1]_{21}=-i/\sqrt{T}$ from \eqref{M112-M121}, and expand 
\eqref{M0-out} as $z\to\infty$ to determine
\eq
[{\bf M}_2]_{21} = \frac{\mu}{\sqrt{T}}.
\endeq  
Together, these facts give
\eq
\label{P221}
[{\bf P}_{n,2}]_{21} = \frac{e^{-n\ell}}{2\pi i}\left(n\mu-\mu+\mathcal{O}(n^{-1})\right).
\endeq
Then \eqref{IP5}, \eqref{P111-P221}, and \eqref{P221} give 
$\beta_{n,n-1}^{(T,\mu,\tau)}=i\mu+\mathcal{O}(n^{-1})$, thereby 
completing the proof of the theorem.
\end{proof}

\section{Analysis in the Hermite regime}
\label{sec-hermite-analysis}

For fixed $T\in(0,\pi^2)$ and fixed $\mu\in[0,\mu_c(T))$, the jump matrices 
${\bf V}^{({\bf X})}(z)$ for ${\bf X}_n(z)$ are uniformly within 
$\mathcal{O}(n^{-1})$ of the identity.  
However, for $\mu=\mu_c(T)$ or $\mu$ approaching $\mu_c(T)$ as $n\to\infty$, 
this uniform decay breaks down on the contour $\mathbb{R}-i\delta$ close to 
the imaginary axis.  On $\mathbb{R}-i\delta$ we have the requirement 
\eq
\Re (\phi(z)) <0, \quad z\in\mathbb{R}-i\delta,
\endeq
where $\phi(z)$ is defined in \eqref{eq:def-phi}.  The smallest postive value 
of $\mu$ for which this condition fails, namely $\mu_c(T)$ defined in 
\eqref{mu-crit}, is specified by the conditions 
  \begin{equation}
  \frac{d\phi}{d z}(z) = 0, \quad \Re \phi(z) = 0,
\end{equation}
for $z$ in the lower half plane.  The $z$-value at which the transition 
occurs is 
\eq
\label{z-crit-fixed-mu}
z_c(T):=i\mu_c(T) - \frac{2i}{T}\sqrt{\pi^2-T}
\endeq
(note that $z_c(T) = -i\delta(\mu_c(T),T)$).
The signature charts for $\Re\phi(z)$ are shown in Figure \ref{sig-chart} in 
subcritical and critical situations.  
Also see Figure \ref{winding-plots}, where the transition from zero winding to 
positive winding is evident in plots of the exact winding number probabilities.
\begin{figure}[h]
\begin{center}
\includegraphics[height=2.1in]{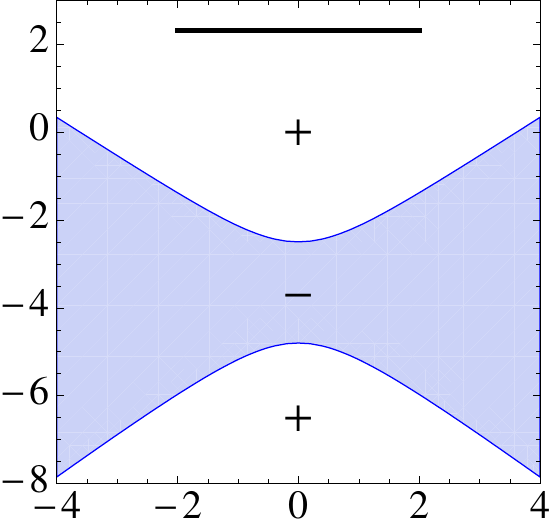}
\includegraphics[height=2.1in]{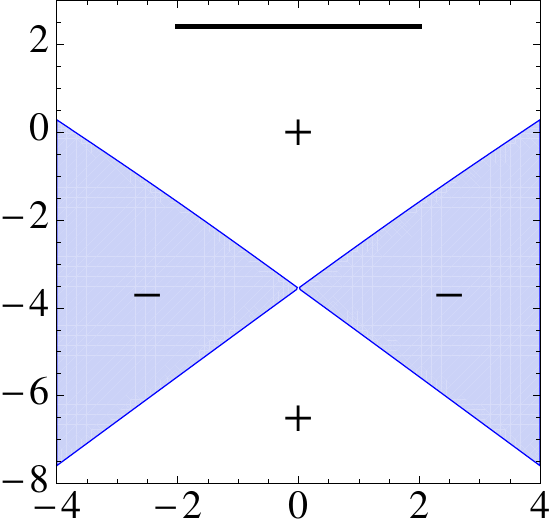}
\caption{Signature charts for $\Re\phi(z)$ with $T=1$ in the complex $z$-plane.  Left: $\mu=\mu_c(1)-0.1$ (subcritical).  Right: $\mu=\mu_c(1)\approx 2.402$ (critical).  The band $[a,b]$ is shown by a thick black line.}
\label{sig-chart}
\end{center}
\end{figure}
Now for $\mu$ close to $\mu_c(T)$ it is necessary to insert a new parametrix 
near $z_c(T)$ in terms of continuous Hermite polynomials.  We will find that it 
is not possible to match the outer 
parametrix ${\bf M}_0^{(\text{out})}(z)$ to this local parametrix.  The 
remedy is to introduce a pole into the outer model problem, which requires the 
construction of a new outer model solution.

\subsection{The outer model problem for the Hermite regime}
We begin by modifying the outer model problem by inserting a pole near 
$z_c(T)$.  The order $k$ of this pole, which depends on $\mu$, is given in 
\eqref{k-condition}.  Instead of placing the pole exactly at $z_c(T)$, it will be 
convenient to use the point $\widetilde{z}_c(\mu,T)$, as defined in \eqref{zc-def}. 
Note that $\widetilde{z}_c\to z_c$ as $\mu\to\mu_c$.  The new outer model 
problem is the following.
\begin{rhp}[The Hermite regime outer model problem]
\label{rhp-outer-hermite}
Fix $k\in\mathbb{N}$ and determine a $2\times 2$ 
matrix-valued function ${\bf M}_k^{(\rm{out})}(z)$ satisfying:
\begin{itemize}
\item[]{\bf Analyticity:} ${\bf M}_k^{(\rm{out})}(z)$ is meromorphic in $z$ 
off $[a,b]$ and is H\"older continuous up to $(a,b)$ with at worst quarter-root 
singularities at $a$ and $b$.  The only pole of ${\bf M}_k^{(\rm{out})}(z)$ is 
at $\widetilde{z}_c$, and 
${\bf M}_k^{(\rm{out})}(z)(z-\widetilde{z}_c)^{-k\sigma_3}$ is analytic at 
$z=\widetilde{z}_c$.  
\item[]{\bf Normalization:}  
\eq
\lim_{z\to\infty}{\bf M}_k^{(\rm{out})}(z) = \mathbb{I}.
\endeq
\item[]{\bf Jump condition:}  Orienting $[a,b]$ left-to-right, the solution 
satisfies \eqref{outer-band-jump}.
\end{itemize}
\end{rhp}
To solve this problem we follow \cite{BertolaL:2009,BuckinghamM:2015} and use the functions $D(z;\mu,T)$, $R(z;\mu,T)$, and $\alpha(z;\mu,T)$ defined in \eqref{def:D}, \eqref{def:R}, and \eqref{def:alpha}, respectively, as well as the constant $ \alpha_\infty$ defined in \eqref{def:alpha_infty}.
Note that $D_+(z)D_-(z)=-1$ for $z\in[a,b]$.  
We will use the following properties of $\alpha(z)$, which are shown by direct 
calculation.
\begin{lem}
\label{tau-properties-lem}
$\alpha(z)$ satisfies the following.
\begin{itemize}
\item[(a)] $\alpha(z)$ is analytic for $z\in\mathbb{C}\backslash[a,b]$.
\item[(b)] $\alpha_+(z)\alpha_-(z)=1$ for $z\in(a,b)$ oriented left-to-right.
\item[(c)] $\displaystyle \alpha(z) = -\frac{T^{3/2}}{4\pi^2}(z-\widetilde{z}_c) + \mathcal{O}((z-\widetilde{z}_c)^2)$ for $z\in\mathbb{D}_{\widetilde{z}_c}$.
\item[(d)] $\alpha(z)$ is bounded in a full neighborhood of $[a,b]$.
\item[(e)] $\displaystyle \alpha(z)= \alpha_\infty\left(1 + \frac{2\pi i}{Tz} + \mathcal{O}\left(\frac{1}{z^2}\right) \right)$ as $z\to\infty$.
\end{itemize}
\end{lem}
Now the solution to the outer model problem is 
\eq
\label{Mk-out}
{\bf M}_k^{(\rm{out})}(z):=\alpha_\infty^{-k\sigma_3}{\bf M}_0^{(\rm{out})}(z)\alpha(z)^{k\sigma_3}.
\endeq

\subsection{The local Hermite parametrix}
We start by expanding $\phi(z)$ about $\widetilde{z}_c$:
\eq
\label{phi-expansion}
\phi(z) = -i\pi + 2\pi(\mu-\mu_c) - \frac{T}{2\pi}\sqrt{\pi^2-T}(z-\widetilde{z}_c)^2 + \mathcal{O}\left((z-\widetilde{z}_c)^3\right).
\endeq
There are a couple of salient features to observe about this expansion.  First, 
the term $\mu-\mu_c$ in the order-zero coefficient is zero at criticality and 
grows as $\mu$ increases.  This term measures the distance into the transition 
region, and its magnitude will determine the power of the pole needed in the 
outer model problem.  For convenience we write 
\eq
\label{c-def}
\mathfrak{c}(\mu,T):=2\pi(\mu-\mu_c(T)).
\endeq
Second, the order-one term is identically zero, while 
the order-two term is nonzero (as long as $T<\pi^2$, as assumed).  Together, 
these facts dictate that the local parametrix is built out of Hermite 
polynomials.  In a neighborhood of $\widetilde{z}_c$, we define the conformal 
map $W(z;\mu,T)$ by the identity
\eq
\label{phi-to-W}
\phi(z) = -i\pi + \mathfrak{c} - W(z)^2
\endeq
and the condition that $W$ is real and increasing along $\mathbb{R}-i\delta$, with $\phi$ as defined in \eqref{eq:def-phi}.  
The function $W$ is analytic and one-to-one for $z$ sufficiently close to 
$\widetilde{z}_c$ since $\phi(z)+i\pi-\mathfrak{c}$ vanishes to second order at 
$\widetilde{z}_c$.  Define $\mathbb{D}_{\widetilde{z}_c}$ to be an 
$n$-independent disk centered at $\widetilde{z}_c$.  The radius of the disk 
is chosen to be less than $\delta$ (so the disk is entirely in the lower 
half-plane) and is sufficiently small so that $W(z)$ is analytic and one-to-one 
in the disk.  From \eqref{phi-expansion} and \eqref{phi-to-W} we see
\eq
\label{W-near-ztildec}
W(z) = \sqrt{\frac{T}{2\pi}}\sqrt[4]{\pi^2-T}(z-\widetilde{z}_c) + \mathcal{O}\left((z-\widetilde{z}_c)^2\right)
\endeq
in its domain of definition.  We seek a parametrix 
${\bf M}_k^{(\widetilde{z}_c)}(z)$ satisfying the following local 
Riemann--Hilbert problem.
\begin{rhp}[The Hermite regime inner model problem in $\mathbb{D}_{\widetilde{z}_c}$]
\label{rhp-inner-hermite}
Fix $k\in\mathbb{N}$.  Determine a $2\times 2$ matrix-valued function 
${\bf M}_k^{(\widetilde{z}_c)}(z)$ defined for 
$z\in\mathbb{D}_{\widetilde{z}_c}$ satisfying:
\begin{itemize}
\item[]{\bf Analyticity:} ${\bf M}_k^{(\widetilde{z}_c)}(z)$ is analytic for 
$z\in\mathbb{D}_{\widetilde{z}_c}$ off the contour $\mathbb{R}-i\delta$ with 
H\"older-continuous boundary values on the contour.  
\item[]{\bf Normalization:}  
\eq
{\bf M}_k^{(\widetilde{z}_c)}(z) = (\mathbb{I}+\mathcal{O}(n^{-1})){\bf M}_k^{(\rm{out})}(z) \text{ as }n\to\infty\text{ for }z\in\partial\mathbb{D}_{\widetilde{z}_c}.
\endeq
\item[]{\bf Jump condition:}  Orienting the jump contour left-to-right, the 
solution satisfies 
\eq
{\bf M}_{k+}^{(\widetilde{z}_c)}(z)={\bf M}_{k-}^{(\widetilde{z}_c)}(z) \bbm 1 & e^{2\pi i\tau-i\pi n + n\mathfrak{c} - nW(z)^2} \\ 0 & 1 \ebm, \quad z\in(\mathbb{R}-i\delta)\cap\mathbb{D}_{\widetilde{z}_c}.
\endeq
\end{itemize}
\end{rhp}
This problem is similar to the well-known Fokas--Its--Kitaev Riemann--Hilbert 
problem for (continuous) Hermite orthogonal polynomials \cite{FokasIK:1991}.  
Recall the Hermite polynomials $\mathfrak{h}_n(\zeta)$ and their leading 
coefficients $\kappa_k$ defined in \eqref{cont-hermite-def} and \eqref{eq:hk}, 
respectively.  
Then define the matrix-valued function 
\eq
\label{Hk-def}
{\bf H}_k(\zeta):=\begin{cases} \bbm 1 & \mathcal{C}(e^{-\zeta^2}) \\ 0 & 1 \ebm, \quad k=0, \vspace{.05in} \\ \bbm \frac{1}{\kappa_k}\mathfrak{h}_k(\zeta) & \frac{1}{\kappa_k}\mathcal{C}(\mathfrak{h}_k(\zeta)e^{-\zeta^2}) \\ -2\pi i \kappa_{k-1}\mathfrak{h}_{k-1}(\zeta) & -2\pi i \kappa_{k-1}\mathcal{C}(\mathfrak{h}_{k-1}(\zeta)e^{-\zeta^2}) \ebm, & k\geq 1, \end{cases}
\endeq
where $\mathcal{C}$ is the continuous Cauchy transform 
\eq
\mathcal{C}(f(\zeta)):=\frac{1}{2\pi i}\int_\mathbb{R}\frac{f(s)ds}{s-\zeta}.
\endeq
It is a now-classical fact that ${\bf H}_k(\zeta)$ is the unique solution to the 
following Riemann--Hilbert problem.
\begin{rhp}[The Fokas--Its--Kitaev problem for Hermite polynomials]
\label{rhp-fokas-its-kitaev}
Fix $k\in\mathbb{N}$.  Determine a $2\times 2$ matrix-valued function 
${\bf H}_k(\zeta)$ such that:
\begin{itemize}
\item[]{\bf Analyticity:} ${\bf H}_k(\zeta)$ is analytic for 
$\mathbb{C}\backslash\mathbb{R}$, with H\"older-continuous boundary values 
from the upper and lower half-planes.
\item[]{\bf Normalization:}  
\eq
\label{Hk-expansion}
{\bf H}_k(\zeta)\zeta^{-k\sigma_3} = \mathbb{I} + \mathcal{O}\left(\frac{1}{\zeta}\right)\text{ as }\zeta\to\infty.
\endeq
\item[]{\bf Jump condition:}  Orienting $\mathbb{R}$ left-to-right, the 
solution satisfies 
\eq
{\bf H}_{k+}(\zeta)={\bf H}_{k-}(\zeta) \bbm 1 & e^{-\zeta^2} \\ 0 & 1 \ebm, \quad \zeta\in\mathbb{R}.
\endeq
\end{itemize}
\end{rhp}
From the large-$\zeta$ expansion of \eqref{Hk-def}, we can improve the 
normalization condition \eqref{Hk-expansion} to 
\eq
\label{Hk-full-expansion}
{\bf H}_k(\zeta)\zeta^{-k\sigma_3} = \mathbb{I} + \bbm 0 & \frac{-1}{2\pi i\kappa_k^2} \\ -2\pi i\kappa_{k-1}^2 & 0 \ebm \frac{1}{\zeta} + \bbm \frac{k^2-k}{4} & 0 \\ 0 & -\frac{k^2+k}{4} \ebm \frac{1}{\zeta^2} + \mathcal{O}\left(\frac{1}{\zeta^3}\right)\text{ as }\zeta\to\infty
\endeq
(here $\kappa_{-1}:=0$).  
Before writing down the solution to the inner model problem we need to pull out 
a holomorphic prefactor from ${\bf M}_k^{(\text{out})}(z)$:
\eq
\label{Fk-def}
{\bf F}_k^{(\text{out})}(z):={\bf M}_k^{(\text{out})}(z)W(z)^{-k\sigma_3}, \quad z\in\mathbb{D}_{\widetilde{z}_c}.
\endeq
This function is analytic in its domain of definition from 
\eqref{W-near-ztildec} and the analyticity condition in Riemann--Hilbert Problem 
\ref{rhp-outer-hermite}.  We can now solve Riemann--Hilbert Problem 
\ref{rhp-inner-hermite}:
\eq
\label{Mk-zc}
{\bf M}_k^{(\widetilde{z}_c)}(z) = {\bf F}_k^{(\text{out})}(z)n^{-k\sigma_3/2}e^{i\pi\tau\sigma_3}e^{n(\mathfrak{c}-i\pi)\sigma_3/2}{\bf H}_k(n^{1/2}W(z))e^{n(i\pi-\mathfrak{c})\sigma_3/2}e^{-i\pi\tau\sigma_3}.
\endeq

\subsection{The global model and error problems in the Hermite regime}

We also note that there is an Airy parametrix ${\bf M}_k^{(a)}(z)$ 
defined for $z\in\mathbb{D}_a$ satisfying Riemann--Hilbert Problem \ref{rhp-airy} 
with the subscript 0 replaced with the subscript $k$ (and an analogous 
parametrix ${\bf M}_k^{(b)}(z)$ defined for $z\in\mathbb{D}_b$).  In 
particular, the parametrices satisfy 
\eq
\label{airy-matching}
{\bf M}_k^{(\#)}(z) = {\bf M}_k^{(\rm{out})}(z)(\mathbb{I}+\mathcal{O}(n^{-1})), \quad z\in\partial\mathbb{D}_{\#} \quad (\#\in\{a,b\}).
\endeq
We refer to 
\cite{BuckinghamM:2015} for discussion of a similar construction, as well as 
\cite{DeiftKMVZ:1999,BleherL:2011,BuckinghamM:2014} for more details.
We now set the global model solution to be 
\eq
\label{Mk-def}
{\bf M}^{(k)}(z):=\begin{cases} {\bf M}_k^\text{(out)}(z), & z\in\mathbb{C}\backslash(\mathbb{D}_a\cup\mathbb{D}_b\cup\mathbb{D}_{\widetilde{z}_c}), \\ {\bf M}_k^{(a)}(z), & z\in \mathbb{D}_a, \\ {\bf M}_k^{(b)}(z), & z\in \mathbb{D}_b, \\ {\bf M}_k^{(\widetilde{z}_c)}(z), & z\in \mathbb{D}_{\widetilde{z}_c}.  \end{cases}
\endeq
The error function is 
\eq
\label{X-hermite-def}
{\bf X}_n^{(k)}(z) := {\bf S}_n(z){\bf M}^{(k)}(z)^{-1}.
\endeq
This function has jumps on the contours $\Sigma_\text{Herm}^{({\bf X})}$ as 
shown in Figure \ref{fig-hermite-error-jumps}.  
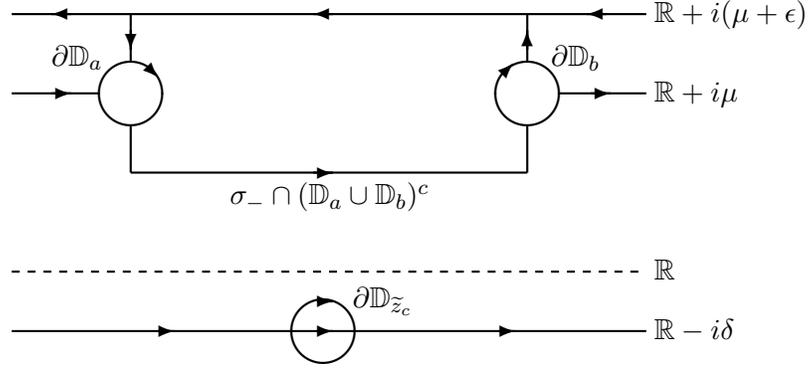
\begin{figure}[h]
\setlength{\unitlength}{1.5pt}
\begin{center}
\begin{picture}(100,80)(-50,-20)
\thicklines
\put(-80,65){\line(1,0){160}}
\put(-80,45){\line(1,0){22}}
\put(58,45){\line(1,0){22}}
\put(-50,25){\line(1,0){100}}
\multiput(-80,0)(4,0){40}{\line(1,0){2}}
\put(-80,-15){\line(1,0){160}}
\put(-50,25){\line(0,1){12}}
\put(-50,53){\line(0,1){12}}
\put(50,25){\line(0,1){12}}
\put(50,53){\line(0,1){12}}
\put(82,-2){$\mathbb{R}$}
\put(82,43){$\mathbb{R}+i\mu$}
\put(82,63){$\mathbb{R}+i(\mu+\epsilon)$}
\put(82,-17){$\mathbb{R}-i\delta$}
\put(-50,45){\circle{15}}
\put(50,45){\circle{15}}
\put(-1.5,-15){\circle{15}}
\put(-70,52){$\partial\mathbb{D}_a$}
\put(56,52){$\partial\mathbb{D}_b$}
\put(6,-9){$\partial\mathbb{D}_{\widetilde{z}_c}$}
\put(-25,17){$\sigma_-\cap(\mathbb{D}_a\cup\mathbb{D}_b)^c$}
\put(1,65){\vector(-1,0){5}}
\put(66,45){\vector(1,0){5}}
\put(-70,45){\vector(1,0){5}}
\put(70,65){\vector(-1,0){5}}
\put(-65,65){\vector(-1,0){5}}
\put(-4,25){\vector(1,0){5}}
\put(-3,-7.5){\vector(1,0){4}}
\put(-3,-15){\vector(1,0){4}}
\put(-43,-15){\vector(1,0){4}}
\put(43,-15){\vector(1,0){4}}
\put(50,56){\vector(0,1){5}}
\put(-50,61){\vector(0,-1){5}}
\put(-44,50){\vector(1,-1){1}}
\put(45.5,51.5){\vector(1,1){1}}
\end{picture}
\end{center}
\caption{\label{fig-hermite-error-jumps} The jump 
contours $\Sigma_\text{Herm}^{({\bf X})}$ along with their orientations.}
\end{figure}
We write the jump conditions as 
\eq
\label{X-hermite-jump}
{\bf X}_{n+}^{(k)}(z) = {\bf X}_{n-}^{(k)}(z) {\bf V}_\text{Herm}^{({\bf X})}(z), \quad z\in\Sigma_\text{Herm}^{({\bf X})}.
\endeq
The jumps are exponentially small (as $n\to\infty$) except on the three disk 
boundaries.  Special attention should be paid to the jump on 
$\mathbb{R}-i\delta$ inside $\mathbb{D}_{\widetilde{z}_c}$, on which 
\eq
\label{VHerm-on-R-near-zc}
\begin{split}
{\bf V}_\text{Herm}^{({\bf X})}(z) & = {\bf X}_{n-}^{(k)}(z)^{-1}{\bf X}_{n+}^{(k)}(z) \\ 
  & = {\bf M}_-^{(k)}(z){\bf S}_{n-}(z)^{-1}{\bf S}_{n+}(z){\bf M}_+^{(k)}(z)^{-1} \\ 
  & = {\bf M}_-^{(k)}(z){\bf V}^{({\bf S})}(z){\bf V}^{({\bf M})}(z)^{-1}{\bf M}_-^{(k)}(z)^{-1} \\
  & = {\bf M}_-^{(k)}(z) \bbm 1 & \left(\frac{1}{1-e^{-2\pi i(nz-\tau)}} - 1\right)e^{n\phi+2\pi i\tau} \\ 0 & 1 \ebm {\bf M}_-^{(k)}(z)^{-1}, \quad z\in(\mathbb{R}-i\delta)\cap\mathbb{D}_{\widetilde{z}_c}.
\end{split}
\endeq
The function ${\bf M}^{(k)}(z)={\bf M}_k^{(\widetilde{z}_c)}(z)$ is bounded, and 
the term in parentheses in the (12)-entry of the middle matrix in the last line 
of \eqref{VHerm-on-R-near-zc} is exponentially small.  Thus the jump on this 
segment is exponentially close to the identity as long as $e^{n\phi}$ is 
not exponentially growing (or growing at a slower exponential rate than the 
term in parentheses).  
The jumps on $\partial\mathbb{D}_a$ and $\partial\mathbb{D}_b$ are 
$\mathcal{O}(n^{-1})$ by \eqref{airy-matching}.  Thus
\eq
\label{V-herm-small-off-Dzc}
{\bf V}_\text{Herm}^{({\bf X})}(z) = \mathcal{O}\left(\frac{1}{n}\right), \quad z\in\Sigma_\text{Herm}^{({\bf X})}\backslash\mathbb{D}_{i\mu}.
\endeq
However, the jump on $\partial\mathbb{D}_{\widetilde{z}_c}$ is actually not 
small for certain values of $\mu$.  We will address this problem in the next 
section by introducing a \emph{parametrix for the error}.  See 
\cite{BuckinghamM:2012,BuckinghamM:2015,ClaeysG:2010} for similar 
constructions.

\subsection{Parametrix for the error}
We begin by gauging the size of the error jump on 
$\partial\mathbb{D}_{\widetilde{z}_c}$.  
Apply \eqref{X-hermite-jump}, \eqref{X-hermite-def}, and the fact that 
${\bf S}_n$ is continuous across $\partial\mathbb{D}_{\widetilde{z}_c}$ to see 
\eq
\left.{\bf V}_\text{Herm}^{({\bf X})}(z)\right\vert_{\partial\mathbb{D}_{\widetilde{z}_c}} = {\bf M}_k^{(\widetilde{z}_c)}(z){\bf M}_k^\text{(out)}(z)^{-1}.
\endeq
This is calculated using \eqref{Fk-def}, \eqref{Mk-zc}, and 
\eqref{Hk-full-expansion}:
\eq
\label{VHerm-near-zc-1}
\begin{split}
& \left.{\bf V}_\text{Herm}^{({\bf X})}(z)\right\vert_{\partial\mathbb{D}_{\widetilde{z}_c}} \\
& = {\bf F}_k^{(\text{out})}(z)n^{-k\sigma_3/2}e^{i\pi\tau\sigma_3}e^{n(\mathfrak{c}-i\pi)\sigma_3/2}{\bf H}_k(n^{1/2}W(z))W(z)^{-k\sigma_3}e^{n(i\pi-\mathfrak{c})\sigma_3/2}e^{-i\pi\tau\sigma_3}{\bf F}_k^{(\text{out})}(z)^{-1}  \\
& = {\bf F}_k^{(\text{out})}(z)n^{-k\sigma_3/2}e^{i\pi\tau\sigma_3}e^{n(\mathfrak{c}-i\pi)\sigma_3/2} \\
& \hspace{.2in} \times \left(\mathbb{I} + \frac{1}{n^{1/2}W(z)}\bbm 0 & \frac{-1}{2\pi i\kappa_k^2} \\ -2\pi i\kappa_{k-1}^2 & 0 \ebm + \frac{1}{nW(z)^2}\bbm \frac{k^2-k}{4} & 0 \\ 0 & -\frac{k^2+k}{4} \ebm  + \mathcal{O}\left(\frac{1}{n^{3/2}}\right)\right) \\ 
& \hspace{.2in} \times e^{n(i\pi-\mathfrak{c})\sigma_3/2} e^{-i\pi\tau\sigma_3} n^{k\sigma_3/2}{\bf F}_k^{(\text{out})}(z)^{-1}.
\end{split}
\endeq
In the last equality we used that $W(z)$ is independent of $n$.  Since the same 
is true for ${\bf F}_k^{(\text{out})}(z)$, we have 
\eq
\left.{\bf V}_\text{Herm}^{({\bf X})}(z)\right\vert_{\partial\mathbb{D}_{\widetilde{z}_c}} = \mathbb{I} + \bbm 0 &  \mathcal{O}\left(\frac{e^{n\mathfrak{c}}}{n^{k+(1/2)}}\right) \\ \mathcal{O}\left(\frac{n^{k-(1/2)}}{e^{n\mathfrak{c}}}\right) & 0 \ebm + \mathcal{O}\left(\frac{1}{n}\right)+ \mathcal{O}\left(\frac{e^{n\mathfrak{c}}}{n^{k+(3/2)}}\right) + \mathcal{O}\left(\frac{n^{k-(3/2)}}{e^{n\mathfrak{c}}}\right)
\endeq
if $k\geq 1$, and 
\eq
\left.{\bf V}_\text{Herm}^{({\bf X})}(z)\right\vert_{\partial\mathbb{D}_{\widetilde{z}_c}} = \mathbb{I} + \bbm 0 &  \mathcal{O}\left(\frac{e^{n\mathfrak{c}}}{n^{k+(1/2)}}\right) \\ 0 & 0 \ebm + \mathcal{O}\left(\frac{e^{n\mathfrak{c}}}{n^{k+(3/2)}}\right) + \mathcal{O}\left(\frac{n^{k-(3/2)}}{e^{n\mathfrak{c}}}\right)
\endeq
if $k=0$ (since $\kappa_{-1}\equiv 0$).  
For these jumps to be close to the identity, we need 
\eq
\left(k-\frac{1}{2}\right)\frac{\log n}{n} < \mathfrak{c} < \left(k+\frac{1}{2}\right)\frac{\log n}{n} \hspace{.1in} (k\geq 0).
\endeq
Note that if $\mathfrak{c}<-\frac{1}{2n}\log n$ then the analysis in the 
subcritical region goes through without change.  
Along with \eqref{c-def}, we now see that for given $n$ and $\mu$ the correct 
choice of $k$ is the integer such that \eqref{k-condition} holds.  
However, if $(\mu-\mu_c)2\pi n/\log n$ is exactly a positive half-integer, then it 
is not possible to choose $k$ so the jump is small on this circle.  Define 
\eq
\label{Qpm-def}
Q_+ \equiv Q_+(T,\mu,\tau,n) := \frac{e^{2i\pi \tau}}{2\pi i\kappa_k^2}\frac{(-1)^ne^{n\mathfrak{c}}}{n^{k+(1/2)}}, \quad Q_- \equiv Q_-(T,\mu,\tau,n) := \frac{2\pi i\kappa_{k-1}^2}{e^{2i\pi\tau}}\frac{n^{k-(1/2)}}{(-1)^n e^{n\mathfrak{c}}}.  
\endeq
We have
\eq
\label{Xherm-jump-expansion}
\left.{\bf V}_\text{Herm}^{({\bf X})}(z)\right\vert_{\partial\mathbb{D}_{\widetilde{z}_c}} = \mathbb{I} - \frac{1}{W(z)} {\bf F}_k^{(\rm{out})}(z) \bbm 0 & Q_+ \\ Q_- & 0 \ebm {\bf F}_k^{(\rm{out})}(z)^{-1} + \mathcal{O}\left(\frac{1}{n}\right).
\endeq
Note that
\eq
\label{Qpm-in-good-regions}
\begin{split}
Q_+ & = \mathcal{O}(n^{-1/2}), \quad \mathfrak{c} \le k\frac{\log n}{n}, \\
Q_- & = \mathcal{O}(n^{-1/2}), \quad k\frac{\log n}{n} \leq \mathfrak{c}.
\end{split}
\endeq
Disregarding all terms of $\mathcal{O}(n^{-1/2})$ or smaller, we arrive at the 
following approximate Riemann--Hilbert problem.
\begin{rhp}[The parametrix for the error]
\label{rhp-error-parametrix}
Determine the $2\times 2$ matrix-valued function ${\bf Y}_n^{(k)}(z)$ satisfying:
\begin{itemize}
\item[]{\bf Analyticity:} ${\bf Y}_n^{(k)}(z)$ is analytic for 
$z\in\mathbb{C}\backslash\partial\mathbb{D}_{\widetilde{z}_c}$ with 
H\"older-continuous boundary values on $\partial\mathbb{D}_{\widetilde{z}_c}$.
\item[]{\bf Normalization:}  
\eq
{\bf Y}_n^{(k)}(z) = \mathbb{I}+\mathcal{O}\left(\frac{1}{z}\right) \text{ as } z\to\infty.
\endeq
\item[]{\bf Jump condition:}  Orienting the jump contour negatively, the 
solution satisfies 
\eq
\label{Y-jump}
{\bf Y}_{n+}^{(k)}(z)={\bf Y}_{n-}^{(k)}(z){\bf V}_{\rm Herm}^{({\bf Y})}(z) \equiv {\bf Y}_{n-}^{(k)}(z) \left(\mathbb{I} - \frac{1}{W(z)} {\bf F}_k^{(\rm{out})}(z) {\bf Q} {\bf F}_k^{(\rm{out})}(z)^{-1} \right), \quad z\in\partial\mathbb{D}_{\widetilde{z}_c},
\endeq
where
\eq
\label{Q-def}
{\bf Q}:= \begin{cases} \bbm 0 & 0 \\ Q_- & 0 \ebm, & \displaystyle \left(k-\frac{1}{2}\right)\frac{\log n}{n} < \mathfrak{c} \le k\frac{\log n}{n}, \vspace{.05in} \\ \bbm 0 & Q_+ \\ 0 & 0 \ebm, & \displaystyle k\frac{\log n}{n} < \mathfrak{c} \le \left(k+\frac{1}{2}\right)\frac{\log n}{n}. \end{cases}
\endeq
\end{itemize}
\end{rhp}
The advantage of the function ${\bf Y}_n^{(k)}(z)$ is that the ratio
\eq
\label{Z-def}
{\bf Z}_n^{(k)}(z) := {\bf X}_n^{(k)}(z){\bf Y}_n^{(k)}(z)^{-1}
\endeq
satisfies the jump condition
\eq
\label{Z-hermite-jump}
{\bf Z}_{n+}^{(k)}(z) = {\bf Z}_{n-}^{(k)}(z) {\bf V}_\text{Herm}^{({\bf Z})}(z), \quad z\in\Sigma_\text{Herm}^{({\bf Z})}\equiv\Sigma_\text{Herm}^{({\bf X})},
\endeq
where ${\bf V}_\text{Herm}^{({\bf Z})}(z)$ is uniformly $o(1)$ as $n\to\infty$ 
(in the Hermite regime).  The jump $\Sigma_\text{Herm}^{({\bf Z})}$ is clearly 
controlled for 
$z\in\Sigma_\text{Herm}^{({\bf Z})}\backslash\partial\mathbb{D}_{\widetilde{z}_c}$ 
(since ${\bf V}_\text{Herm}^{({\bf X})}(z)$ is controlled and 
${\bf Y}_n^{(k)}(z)$ has no jump here).  On the other hand, we also have 
\eq
\begin{split}
\left.{\bf V}_\text{Herm}^{({\bf Z})}(z)\right\vert_{\partial\mathbb{D}_{\widetilde{z}_c}} & = {\bf Z}_{n-}^{(k)}(z)^{-1}{\bf Z}_{n+}^{(k)}(z) \\ 
  & = {\bf Y}_{n-}^{(k)}(z){\bf X}_{n-}^{(k)}(z)^{-1}{\bf X}_{n+}(z){\bf Y}_{n+}^{(k)}(z)^{-1} \\ 
  & = {\bf Y}_{n-}^{(k)}(z){\bf V}_\text{Herm}^{({\bf X})}(z){\bf V}_\text{Herm}^{({\bf Y})}(z)^{-1}{\bf Y}_{n-}^{(k)}(z)^{-1}.
\end{split}
\endeq
Thus, recalling \eqref{Xherm-jump-expansion} and \eqref{Y-jump}, 
for $z\in\mathbb{D}_{\widetilde{z}_c}$, 
\eq
\label{VZ-n-expansion}
{\bf V}_\text{Herm}^{({\bf Z})}(z) = \mathbb{I} - \frac{1}{W(z)}{\bf Y}_{n-}^{(k)}(z){\bf F}_k^{(\text{out})}(z) \widehat{\bf Q} {\bf F}_k^{(\text{out})}(z)^{-1}{\bf Y}_{n-}^{(k)}(z)^{-1} + \mathcal{O}\left(\frac{1}{n}\right), 
\endeq
where we have defined
\eq
\label{Qhat-def}
\widehat{\bf Q} := \begin{cases} \bbm 0 & Q_+ \\ 0 & 0 \ebm, & \displaystyle \left(k-\frac{1}{2}\right)\frac{\log n}{n} < \mathfrak{c} \le k\frac{\log n}{n}, \vspace{.05in} \\ \bbm 0 & 0 \\ Q_- & 0 \ebm, & \displaystyle k\frac{\log n}{n} < \mathfrak{c} \le \left(k+\frac{1}{2}\right)\frac{\log n}{n}. \end{cases}
\endeq
From \eqref{Qpm-in-good-regions} we see 
$\widehat{\bf Q} = \mathcal{O}(n^{-1/2})$, and thus 
\eq
\label{VZ-bound}
{\bf V}_\text{Herm}^{({\bf Z})}(z) = \mathcal{O}(n^{-1/2}) \quad \text{for all }
z\in\Sigma^{({\bf Z})}.
\endeq

We now solve Riemann--Hilbert Problem \ref{rhp-error-parametrix} exactly 
following \cite{BuckinghamM:2015}, \S3.5.2.  The function 
\eq
\label{Ytilde}
{\bf \widetilde{Y}}_n^{(k)}(z) := \begin{cases} {\bf Y}_n^{(k)}(z), & z\in\mathbb{C}\backslash\mathbb{D}_{\widetilde{z}_c}, \\ \displaystyle {\bf Y}_n^{(k)}(z)\left(\mathbb{I} - \frac{1}{W(z)} {\bf F}_k^{(\rm{out})}(z) {\bf Q} {\bf F}_k^{(\rm{out})}(z)^{-1} \right), & z\in\mathbb{D}_{\widetilde{z}_c} \end{cases} 
\endeq
is the (meromorphic) continuation of ${\bf Y}_n^{(k)}(z)$ from the exterior to 
the interior of $\mathbb{D}_{\widetilde{z}_c}$.  This function is meromorphic 
in the entire complex $z$-plane with exactly one simple pole at 
$\widetilde{z}_c$.  Furthermore, 
$\lim_{z\to\infty}{\bf Y}_n^{(k)}(z)=\mathbb{I}$, so we can write 
\eq
\label{Ytilde-ansatz}
{\bf \widetilde{Y}}_n^{(k)}(z) = \mathbb{I} + \frac{1}{z-\widetilde{z}_c}{\bf B},
\endeq
where ${\bf B}$ is independent of $z$.  
Our next goal is to determine the $2\times 2$ matrix 
${\bf B}\equiv{\bf B}(T,\mu,\tau,n)$.  Since ${\bf Y}_n^{(k)}(z)$ is analytic 
at $z=\widetilde{z}_c$, we have 
\eq
\label{Yn-explicit}
{\bf Y}_n^{(k)}(z) = \left( \mathbb{I} + \frac{1}{z-\widetilde{z}_c}{\bf B} \right) \left(\mathbb{I} + \frac{1}{W(z)} {\bf F}_k^{(\rm{out})}(z) {\bf Q} {\bf F}_k^{(\rm{out})}(z)^{-1} \right) = \mathcal{O}(1) \text{ for } z\to\widetilde{z}_c.
\endeq
We expand the middle quantity at $z=\widetilde{z}_c$.  Write
\eq
\label{Fout-expansion}
{\bf F}_k^{(\text{out})}(z) = {\bf F}_0 + {\bf F}_1(z-\widetilde{z}_c) + \mathcal{O}\left((z-\widetilde{z}_c)^2\right)
\endeq
and
\eq
\begin{split}
&W(z) = (z-\widetilde{z}_c)W_1 + (z-\widetilde{z}_c)^2W_2 + \mathcal{O}\left((z-\widetilde{z}_c)^3\right) \\
&\Rightarrow \frac{1}{W(z)} = \frac{1}{(z-\widetilde{z}_c)W_1} - \frac{W_2}{W_1^2} + \mathcal{O}(z-\widetilde{z}_c),
\end{split}
\endeq
where ${\bf F}_0$, ${\bf F}_1$, $W_1$, and $W_2$ are independent of $z$.  Note 
that 
\eq
\label{W1-def}
W_1 = \sqrt{\frac{T}{2\pi}}\sqrt[4]{\pi^2-T}
\endeq
from \eqref{W-near-ztildec}.  Then we must have 
\eq
\frac{{\bf B}{\bf F}_0{\bf Q}{\bf F}_0^{-1}}{(z-\widetilde{z}_c)^2 W_1} + \frac{1}{z-\widetilde{z}_c}\left( {\bf B} + \frac{{\bf F}_0{\bf Q}{\bf F}_0^{-1}}{W_1} + \frac{{\bf B}{\bf F}_1{\bf Q}{\bf F}_0^{-1}}{W_1} - \frac{{\bf B}{\bf F}_0{\bf Q}{\bf F}_0^{-1}{\bf F}_1{\bf F}_0^{-1}}{W_1} - \frac{W_2{\bf B}{\bf F}_0{\bf Q}{\bf F}_0^{-1}}{W_1^2}  \right) = {\bf 0}.
\endeq
Separating powers of $z-\widetilde{z}_c$ (and using the invertibility of 
${\bf F}_0$, which can be checked via explicit computation) gives the system of 
equations 
\eq
{\bf B}{\bf F}_0{\bf Q} = {\bf 0}, \quad 
{\bf B}{\bf F}_0 + \frac{1}{W_1}{\bf F}_0{\bf Q} + \frac{1}{W_1}{\bf B}{\bf F}_1{\bf Q} = {\bf 0}.
\endeq
The first equation shows that the first column of ${\bf B}{\bf F}_0$ is zero 
when ${\bf Q}$ is strictly upper triangular, and the second column of 
${\bf B}{\bf F}_0$ is zero when ${\bf Q}$ is strictly lower triangular.  
Solving the second equation then gives 
\eq
\label{B-formula}
{\bf B} = \begin{cases} \displaystyle \frac{-{\bf F}_0{\bf Q}{\bf F}_0^{-1}}{W_1 + [{\bf F}_0^{-1}{\bf F}_1]_{12}Q_-}, & \displaystyle \left(k-\frac{1}{2}\right)\frac{\log n}{n} < \mathfrak{c} \le k\frac{\log n}{n}, \vspace{.05in} \\ \displaystyle \frac{-{\bf F}_0{\bf Q}{\bf F}_0^{-1}}{W_1 + [{\bf F}_0^{-1}{\bf F}_1]_{21}Q_+}, & \displaystyle k\frac{\log n}{n} < \mathfrak{c} \le \left(k+\frac{1}{2}\right)\frac{\log n}{n}. \end{cases}
\endeq
We will need the (11)-entry of ${\bf B}$.  Start with 
\eq
[{\bf B}]_{11} = \begin{cases} \displaystyle \frac{-\frac{Q_-}{W_1}[{\bf F}_0]_{12}[{\bf F}_0]_{22}}{1 + \frac{Q_-}{W_1}[{\bf F}_0^{-1}{\bf F}_1]_{12}}, & \displaystyle \left(k-\frac{1}{2}\right)\frac{\log n}{n} < \mathfrak{c} \le k\frac{\log n}{n}, \vspace{.05in} \\ \displaystyle \frac{\frac{Q_+}{W_1}[{\bf F}_0]_{11}[{\bf F}_0]_{21}}{1 + \frac{Q_+}{W_1}[{\bf F}_0^{-1}{\bf F}_1]_{21}}, & \displaystyle k\frac{\log n}{n} < \mathfrak{c} \le \left(k+\frac{1}{2}\right)\frac{\log n}{n}. \end{cases}
\endeq
Observe that ${\bf F}_0\equiv{\bf F}_k^{(\text{out})}(\widetilde{z}_c)$.  
We introduce the notations
\eq
\label{gamma-star}
\gamma_* \equiv \gamma_*(T) := \gamma(\widetilde{z}_c) = \frac{1}{\sqrt{\pi}}\left( \sqrt{\pi^2-T} + i\sqrt{T} \right)^{1/2}
\endeq
(here the square root is the principal branch) and  
\eq
\label{alpha-on-W-at-zc}
\frac{\alpha_*}{W_*} := \lim_{z\to\widetilde{z}_c}\frac{\alpha(z)}{W(z)} = \frac{-T}{(2\pi)^{3/2}\sqrt[4]{\pi^2-T}}, 
\endeq
obtained via the local expansions near $z=\widetilde{z}_c$ in Lemma 
\ref{tau-properties-lem}(c) for $\alpha(z)$ and Equation \eqref{W-near-ztildec} 
for $W(z)$.  Then from \eqref{Fk-def}, \eqref{Mk-out}, and \eqref{M0-out} we 
have 
\eq
\label{F0-def}
{\bf F}_0 = \alpha_\infty^{-k\sigma_3} \bbm \displaystyle \frac{\gamma_*+\gamma_*^{-1}}{2} & \displaystyle \frac{\gamma_*-\gamma_*^{-1}}{-2i} \\ \displaystyle \frac{\gamma_*-\gamma_*^{-1}}{2i} & \displaystyle \frac{\gamma_*+\gamma_*^{-1}}{2} \ebm \left(\frac{\alpha_*}{W_*}\right)^{k\sigma_3}.
\endeq
Together, the previous three equations yield
\eq
\label{F11F21-F12F22}
\begin{split}
[{\bf F}_0]_{11} [{\bf F}_0]_{21} & = -\frac{i}{4}(\gamma_*^2-\gamma_*^{-2}) \left(\frac{\alpha_*}{W_*}\right)^{2k} = \frac{T^{(4k+1)/2}}{(2\pi)^{3k+1}(\pi^2-T)^{k/2}}, \\ [{\bf F}_0]_{12}[{\bf F}_0]_{22} & =  \frac{i}{4}(\gamma_*^2-\gamma_*^{-2})\left(\frac{\alpha_*}{W_*}\right)^{-2k} = -\frac{(2\pi)^{3k-1}(\pi^2-T)^{k/2}}{T^{(4k-1)/2}}.
\end{split}
\endeq
At this point we define the quantities 
\eq
\label{Rpm-def}
\begin{split}
R_+ \equiv R_+(T,\mu,n) & := \frac{iT^{2k}}{(2\pi)^{(6k+3)/2}(\pi^2-T)^{(2k+1)/4}\kappa_k^2}\cdot\frac{(-1)^{n+1}e^{n\mathfrak{c}}}{n^{k+(1/2)}}, \\ 
R_- \equiv R_-(T,\mu,n) & :=\frac{(2\pi)^{(6k+1)/2}(\pi^2-T)^{(2k-1)/4}\kappa_{k-1}^2}{iT^{2k}}\cdot\frac{n^{k-(1/2)}}{(-1)^{n+1}e^{n\mathfrak{c}}},
\end{split}
\endeq
so that (using \eqref{Qpm-def}, \eqref{W1-def}, and \eqref{F11F21-F12F22})  
\eq
\label{QWFF-ito-Rpm}
\frac{Q_+}{W_1}[{\bf F}_0]_{11}[{\bf F}_0]_{21} = R_+e^{2\pi i\tau}, \quad -\frac{Q_-}{W_1}[{\bf F}_0]_{12}[{\bf F}_0]_{22} = R_-e^{-2\pi i\tau}.
\endeq
We also have 
\eq
\begin{split}
{\bf F}_1\equiv \left.\frac{d{\bf F}_k^{(\text{out})}(z)}{dz}\right|_{z=\widetilde{z}_c} = & \alpha_\infty^{-k\sigma_3} \frac{\gamma^{\hspace{.02in}\prime}(\widetilde{z}_c)}{\gamma_*} \bbm \displaystyle \frac{\gamma_*-\gamma_*^{-1}}{2} & \displaystyle \frac{\gamma_*+\gamma_*^{-1}}{-2i} \\ \displaystyle \frac{\gamma_*+\gamma_*^{-1}}{2i} & \displaystyle \frac{\gamma_*-\gamma_*^{-1}}{2} \ebm \left(\frac{\alpha_*}{W_*}\right)^{k\sigma_3} \\ 
 & + \alpha_\infty^{-k\sigma_3} \bbm \displaystyle \frac{\gamma_*+\gamma_*^{-1}}{2} & \displaystyle \frac{\gamma_*-\gamma_*^{-1}}{-2i} \\ \displaystyle \frac{\gamma_*-\gamma_*^{-1}}{2i} & \displaystyle \frac{\gamma_*+\gamma_*^{-1}}{2} \ebm \frac{d}{dz}\left.\left(\left(\frac{\alpha(z)}{W(z)}\right)^{k\sigma_3}\right)\right|_{z=\widetilde{z}_c}.
\end{split}
\endeq
Direct calculation shows 
\eq
\frac{\gamma^{\hspace{.02in}\prime}(\widetilde{z}_c)}{\gamma_*} = \frac{T^{3/2}}{4\pi^2}.
\endeq
The preceding four equations show
\eq
\label{F0invF1-12}
\begin{split}
[{\bf F}_0^{-1}{\bf F}_1]_{12} & = i\frac{\gamma^{\hspace{.02in}\prime}(\widetilde{z}_c)}{\gamma_*}\left(\frac{\alpha_*}{W_*}\right)^{-2k} = \frac{(2\pi)^{3k-2}(\pi^2-T)^{k/2}i}{T^{(4k-3)/2}},\\
[{\bf F}_0^{-1}{\bf F}_1]_{21} & = -i\frac{\gamma^{\hspace{.02in}\prime}(\widetilde{z}_c)}{\gamma_*}\left(\frac{\alpha_*}{W_*}\right)^{2k}  = \frac{T^{(4k+3)/2}}{(2\pi)^{3k+2}(\pi^2-T)^{k/2}i}.
\end{split}
\endeq
Thus, from \eqref{Qpm-def}, \eqref{W1-def}, \eqref{Rpm-def}, and 
\eqref{F0invF1-12},
\eq
\label{QWF0F1}
\frac{Q_+}{W_1}[{\bf F}_0^{-1}{\bf F}_1]_{21} = \frac{TR_+}{2\pi i}e^{2\pi i\tau}
\quad\text{and}\quad 
\frac{Q_-}{W_1}[{\bf F}_0^{-1}{\bf F}_1]_{12} = -\frac{TR_-}{2\pi i}e^{-2\pi i\tau}.
\endeq
Therefore
\eq
\label{B11-formula}
[{\bf B}]_{11} = \begin{cases} \displaystyle \frac{R_-e^{-2\pi i\tau}}{1 - \frac{TR_-}{2\pi i}e^{-2\pi i\tau}}, & \displaystyle \left(k-\frac{1}{2}\right)\frac{\log n}{n} < \mathfrak{c} \le k\frac{\log n}{n}, \vspace{.05in} \\ \displaystyle \frac{R_+e^{2\pi i\tau}}{1 + \frac{TR_+}{2\pi i}e^{2\pi i\tau}}, & \displaystyle k\frac{\log n}{n} < \mathfrak{c} \le \left(k+\frac{1}{2}\right)\frac{\log n}{n}. \end{cases}
\endeq
Later we will need that fact, that follows directly from \eqref{B-formula} and 
\eqref{Q-def}, that 
\eq
\label{B-ito-Qpm}
{\bf B} = \begin{cases} \mathcal{O}(Q_-), & \displaystyle \left(k-\frac{1}{2}\right)\frac{\log n}{n} < \mathfrak{c} \le k\frac{\log n}{n}, \vspace{.05in} \\ \mathcal{O}(Q_+), & \displaystyle k\frac{\log n}{n} < \mathfrak{c} \le \left(k+\frac{1}{2}\right)\frac{\log n}{n}. \end{cases}
\endeq
Putting together \eqref{Ytilde} and \eqref{Ytilde-ansatz} shows
\eq
\label{Y-def}
{\bf Y}_n^{(k)}(z) := \begin{cases} \left( \mathbb{I} + \frac{1}{z-\widetilde{z}_c}{\bf B} \right), & z\in\mathbb{C}\backslash\mathbb{D}_{\widetilde{z}_c}, \\ \left( \mathbb{I} + \frac{1}{z-\widetilde{z}_c}{\bf B} \right)\left(\mathbb{I} + \frac{1}{W(z)} {\bf F}_k^{(\rm{out})}(z) {\bf Q} {\bf F}_k^{(\rm{out})}(z)^{-1} \right), & z\in\mathbb{D}_{\widetilde{z}_c} .\end{cases} 
\endeq

\subsection{Winding numbers in the Hermite regime}
\label{subsec-hermite-winding}
We now determine the winding numbers in the Hermite regime.
\begin{proof}[Proof of Theorem \ref{thm-hermite-winding} (Hermite regime 
winding number asymptotics)]
Recalling Lemma \ref{winding-probs-ito-P}, to compute the winding probabilities 
we will first determine $[{\bf P}_{n,1}]_{11}$.  Repeating the logic leading to 
Equation \eqref{P-ito-XM-subcrit} in the subcritical case shows
\begin{equation}
\label{P-ito-ZYM}
\mathbf P_n(z) =\bbm 1 & 0 \\ 0 & -2\pi i \ebm^{-1} e^{n\ell\sg_3/2} \mathbf Z_n^{(k)}(z)\mathbf Y_n^{(k)}(z) \mathbf M^{(k)}(z) e^{n(g(z) - \ell/2)\sg_3} \bbm 1 & 0 \\ 0 & -2\pi i \ebm
\end{equation}
for $z\in\mathbb{C}\backslash\{z|-\delta\leq \Im z \leq \mu+\epsilon\}$.
We have the large-$z$ expansions 
\eq
{\bf Y}_n^{(k)}(z) = \mathbb{I} + \frac{\bf B}{z} + \mathcal{O}\left(\frac{1}{z^2}\right), \quad {\bf Z}_n^{(k)}(z) = \mathbb{I} + \frac{{\bf Z}_{n,1}^{(k)}}{z} + \mathcal{O}\left(\frac{1}{z^2}\right),
\endeq
where the relevant entry of ${\bf Z}_{n,1}^{(k)}$ will be computed below.  
Therefore, using \eqref{g1} for $g_1$,
\eq
\label{P-large-z}
\begin{split}
\mathbb{I}+\frac{\mathbf P_{n,1}}{z}+\mathcal{O}\left(\frac{1}{z^2}\right) = 
 & \bbm 1 & 0 \\ 0 & -2\pi i \ebm^{-1} e^{n\ell\sg_3/2} \left(\mathbb{I}+\frac{\mathbf Z_{n,1}^{(k)} + {\bf B} + {\bf M}_1^{(k)} }{z}+\mathcal{O}\left(\frac{1}{z^2}\right)\right) \\
 &  \times \exp\left(-n\left(\frac{i\mu}{z}+\mathcal{O}\left(\frac{1}{z^2}\right)\right)\sg_3\right) e^{-(n\ell/2)\sg_3} \bbm 1 & 0 \\ 0 & -2\pi i \ebm.
\end{split}
\endeq
We see
\eq
\label{Pn1-11}
[\mathbf P_{n,1}]_{11} = -in\mu + \left[\mathbf Z_{n,1}^{(k)}\right]_{11} + \left[{\bf B}\right]_{11} + \left[\mathbf M_1^{(k)}\right]_{11}.
\endeq
We have already obtained a formula for $[{\bf B}]_{11}$ in \eqref{B11-formula}.  
Our next step is to find $\left[{\bf M}_1^{(k)}\right]_{11}$.  Now 
${\bf M}^{(k)}(z)={\bf M}_k^\text{(out)}(z)$ for $z$ sufficiently large by 
\eqref{Mk-def}.  From \eqref{Mk-out}, \eqref{M0-out}, \eqref{gamma-def}, and 
Lemma \ref{tau-properties-lem}(e), 
\eq
\left[{\bf M}_k^\text{(out)}(z)\right]_{11} = \frac{\gamma(z)+\gamma(z)^{-1}}{2}\left(\frac{\alpha(z)}{\alpha_\infty}\right)^k = \left(1+\mathcal{O}\left(\frac{1}{z^2}\right)\right)\left(1+\frac{2\pi ik}{Tz}+\mathcal{O}\left(\frac{1}{z^2}\right)\right).
\endeq
Therefore
\eq
\label{Mk1-11}
\left[{\bf M}_1^{(k)}\right]_{11} = \frac{2i\pi k}{T}.
\endeq

Now we find $\left[{\bf Z}_{n,1}^{(k)}\right]_{11}$.  As we have 
shown 
${\bf V}_\text{Herm}^{({\bf Z})}(z) = \mathcal{O}(n^{-1/2})$ for 
$z\in\partial\mathbb{D}_{\widetilde{z}_c}$ and is exponentially small for 
$z\in\Sigma_\text{Herm}^{({\bf Z})}\backslash\partial\mathbb{D}_{\widetilde{z}_c}$, 
it follows that ${\bf Z}_n(z)$ can be solved by a convergent Neumann series 
(see, for example, \cite[\S5]{Liechty:2012} for a similar analysis), and 
in particular
\eq
\label{Zk-integral-form}
{\bf Z}_{n,1}^{(k)} = \frac{-1}{2\pi i}\oint_{\partial\mathbb{D}_{\widetilde{z}_c}}\left({\bf V}_\text{Herm}^{({\bf Z})}(u)-\mathbb{I}\right) du + \mathcal{O}\left(\frac{1}{n}\right).
\endeq
A residue calcuation using \eqref{Zk-integral-form}, \eqref{VZ-n-expansion}, 
\eqref{Fout-expansion}, \eqref{W1-def}, and the fact that 
$\partial\mathbb{D}_{\widetilde{z}_c}$ is negatively oriented gives 
\eq
{\bf Z}_{n,1}^{(k)} = \frac{-1}{W_1}{\bf Y}_n^{(k)}(\widetilde{z}_c){\bf F}_0\widehat{\bf Q}{\bf F}_0^{-1}{\bf Y}_n^{(k)}(\widetilde{z}_c)^{-1} + \mathcal{O}\left(\frac{1}{n}\right).
\endeq
Equations \eqref{Yn-explicit} and \eqref{B-ito-Qpm} show us that 
\eq
{\bf Y}_n^{(k)}(\widetilde{z}_c) = \begin{cases} \mathbb{I} + \mathcal{O}(Q_-), & (k-\frac{1}{2})\frac{\log n}{n} < \mathfrak{c}<k\frac{\log n}{n}, \vspace{.05in} \\ \mathbb{I} + \mathcal{O}(Q_+), & k\frac{\log n}{n} < \mathfrak{c} < (k+\frac{1}{2})\frac{\log n}{n}. \end{cases}
\endeq
The definition \eqref{Qhat-def} of $\widehat{\bf Q}$ and the fact that 
$Q_+Q_-=\mathcal{O}(n^{-1})$ lead to 
\eq
\label{Z1-ito-FQFinv}
{\bf Z}_{n,1}^{(k)} = \frac{-1}{W_1}{\bf F}_0\widehat{\bf Q}{\bf F}_0^{-1} + \mathcal{O}\left(\frac{1}{n}\right).
\endeq
Direct calculation (using $\det{\bf F}_0\equiv 1$, which follows from 
${\bf F}_0 \equiv {\bf F}_k^{(\text{out})}(\widetilde{z}_c)$ and 
$\det{\bf F}_k^{(\text{out})}(z)\equiv 1$) now gives
\eq
\label{Z11-expression}
\left[{\bf Z}_{n,1}^{(k)}\right]_{11} = \begin{cases} \frac{Q_+}{W_1}\left[{\bf F}_0\right]_{11}\left[{\bf F}_0\right]_{21} + \mathcal{O}\left(\frac{1}{n}\right), & (k-\frac{1}{2})\frac{\log n}{n} < \mathfrak{c} \le k\frac{\log n}{n}, \vspace{.05in} \\ -\frac{Q_-}{W_1}\left[{\bf F}_0\right]_{12}\left[{\bf F}_0\right]_{22} + \mathcal{O}\left(\frac{1}{n}\right), & k\frac{\log n}{n} < \mathfrak{c} \le (k+\frac{1}{2})\frac{\log n}{n}. \end{cases}
\endeq
Equation \eqref{QWFF-ito-Rpm} gives
\eq
\left[{\bf Z}_{n,1}^{(k)}\right]_{11} = \begin{cases} R_+e^{2\pi i\tau} + \mathcal{O}\left(\frac{1}{n}\right), & (k-\frac{1}{2})\frac{\log n}{n} < \mathfrak{c}\le k\frac{\log n}{n}, \vspace{.05in} \\ R_-e^{-2\pi i\tau} + \mathcal{O}\left(\frac{1}{n}\right), & k\frac{\log n}{n} < \mathfrak{c} \le  (k+\frac{1}{2})\frac{\log n}{n}. \end{cases}
\endeq
This is similar to the expression \eqref{B11-formula} for $[{\bf B}]_{11}$.  
In fact, since (cf. \eqref{Qpm-in-good-regions})
\eq
\begin{split}
R_+ & = \mathcal{O}(n^{-1/2}), \quad \left(k-\frac{1}{2}\right)\frac{\log n}{n} < \mathfrak{c} \le  k\frac{\log n}{n}, \\
R_- & = \mathcal{O}(n^{-1/2}), \quad k\frac{\log n}{n} < \mathfrak{c} \le  \left(k+\frac{1}{2}\right)\frac{\log n}{n},
\end{split}
\endeq
we can add appropriate denominators with no extra error term:
\eq
\label{Zn1-11}
\left[{\bf Z}_{n,1}^{(k)}\right]_{11} = \begin{cases} \displaystyle \frac{R_+e^{2\pi i\tau}}{1 + \frac{TR_+}{2\pi i}e^{2\pi i\tau}} + \mathcal{O}\left(\frac{1}{n}\right), & (k-\frac{1}{2})\frac{\log n}{n} < \mathfrak{c}\le k\frac{\log n}{n}, \vspace{.05in} \\ \displaystyle \frac{R_-e^{-2\pi i\tau}}{1 - \frac{TR_-}{2\pi i}e^{-2\pi i\tau}} + \mathcal{O}\left(\frac{1}{n}\right), & k\frac{\log n}{n} < \mathfrak{c} \le  (k+\frac{1}{2})\frac{\log n}{n}. \end{cases}
\endeq
The advantage is that now, from \eqref{Pn1-11}, \eqref{B11-formula}, 
\eqref{Mk1-11}, and \eqref{Zn1-11}, $[{\bf P}_{n,1}]_{11}$ takes the simplified 
form 
\eq
\label{Pn1-11-formula}
[{\bf P}_{n,1}]_{11} = -in\mu + \frac{2i\pi k}{T} + \frac{R_+e^{2\pi i\tau}}{1 + \frac{TR_+}{2\pi i}e^{2\pi i\tau}} + \frac{R_-e^{-2\pi i\tau}}{1 - \frac{TR_-}{2\pi i}e^{-2\pi i\tau}} + \mathcal{O}\left(\frac{1}{n}\right).
\endeq
The $\mathcal{O}(n^{-1})$ error terms are uniformly bounded in $\tau$ (and 
therefore integrable with respect to $\tau$).  Combining Proposition 
\ref{Hankel-integral-prop}, \eqref{IP5a}, and \eqref{Pn1-11-formula} shows 
\eq
\begin{split}
\log &\left(\frac{\Hankel_n(T,\mu,\tau) }{\Hankel_n(T,\mu,\hsgn{n})}\right) = \int_{\hsgn{n}}^\tau \left(inT\mu + T [\mathbf P_{n,1}(T,\mu,v)]_{11}\right)\,dv \\ 
& \hspace{.2in} = \int_{\hsgn{n}}^\tau \left[2\pi ik + \frac{TR_+e^{2\pi iv}}{1+\frac{TR_+}{2\pi i}e^{2\pi iv}} + \frac{TR_-e^{-2\pi iv}}{1-\frac{TR_-}{2\pi i}e^{-2\pi iv}}\right] \,dv + \mathcal{O}\left(\frac{1}{n}\right) \\
& \hspace{.2in} = 2ik\pi(\tau-\hsgn{n}) + \log\left(1+\frac{TR_+}{2\pi i}e^{2\pi i\tau}\right) - \log\left(1+\frac{TR_+}{2\pi i}e^{2\pi i\epsilon(n)}\right) \\ 
& \hspace{.45in} + \log\left(1-\frac{TR_-}{2\pi i}e^{-2\pi i\tau}\right) - \log\left(1-\frac{TR_-}{2\pi i}e^{-2\pi i\epsilon(n)}\right) + \mathcal{O}\left(\frac{1}{n}\right).
\end{split}
\endeq
Exponentiating and using $R_+R_-=\mathcal{O}(n^{-1})$ from \eqref{Rpm-def} and 
$e^{2\pi i\hsgn{n}}=(-1)^{n+1}$ from \eqref{eq:def_hsgn} gives
\eq
\frac{\Hankel_n(T,\mu,\tau) }{\Hankel_n(T,\mu,\hsgn{n})} = \frac{(-1)^{-(n+1)k}}{1 + \frac{(-1)^{n+1}T}{2\pi i}(R_+-R_-)}  \left(1 + \frac{TR_+}{2\pi i}e^{2\pi i\tau} - \frac{TR_-}{2\pi i}e^{-2\pi i\tau} \right)e^{2ik\pi\tau} + \mathcal{O}\left(\frac{1}{n}\right).
\endeq
We can now compute the winding numbers using \eqref{eq:total_offset_formula}:
\eq
\begin{split}
& \Prob(\mathcal{W}_n(T,\mu) = \om) \\
& = \frac{(-1)^{(n+1)(\omega-k)}}{1 + \frac{(-1)^{n+1}T}{2\pi i}(R_+-R_-)} \int_0^1 \left(1 + \frac{TR_+}{2\pi i}e^{2\pi i\tau} - \frac{TR_-}{2\pi i}e^{-2\pi i\tau} \right)e^{2\pi i\tau(k-\omega)}d\tau + \mathcal{O}\left(\frac{1}{n}\right) \\
& = \begin{cases} 
\displaystyle \frac{\frac{1}{2\pi i}(-1)^nTR_-}{1 + \frac{1}{2\pi i}(-1)^{n+1}TR_+ + \frac{1}{2\pi i}(-1)^nTR_-} + \mathcal{O}\left(\frac{1}{n}\right), & \omega=k-1, \vspace{.05in} \\ 
\displaystyle \frac{1}{1 + \frac{1}{2\pi i}(-1)^{n+1}TR_+ + \frac{1}{2\pi i}(-1)^nTR_-} + \mathcal{O}\left(\frac{1}{n}\right), & \omega=k, \vspace{.05in} \\ 
\displaystyle \frac{\frac{1}{2\pi i}(-1)^{n+1}TR_+}{1 + \frac{1}{2\pi i}(-1)^{n+1}TR_+ + \frac{1}{2\pi i}(-1)^nTR_-} + \mathcal{O}\left(\frac{1}{n}\right), & \omega=k+1, \vspace{.1in} \\ 
\displaystyle \mathcal{O}\left(\frac{1}{n}\right), & \text{otherwise}. \end{cases} 
\end{split}
\endeq
This can be written in the simplified form \eqref{hermite-winding} using the 
formulas \eqref{c-def} and \eqref{eq:hk} for $\mathfrak{c}$ and $\kappa_k$.  This 
completes the proof of Theorem \ref{thm-hermite-winding}.
\end{proof}

\subsection{Orthogonal polynomial asymptotics in the Hermite regime}
\label{subsec-hermite-norms}

We now prove the results for orthogonal polynomials in the Hermite regime, 
starting with Theorem \ref{thm:polynomials_asymptotics_crit}.
\begin{proof}[Proof of Theorem \ref{thm:polynomials_asymptotics_crit} (Orthogonal 
polynomials in the Hermite regime)]
From \eqref{IP3}, \eqref{R-def}, \eqref{T-def}, \eqref{st2}, 
\eqref{X-hermite-def}, and \eqref{Z-def},
\eq
\label{p-ito-ZYM}
p_{n,n}^{(T,\mu,\tau)}(z) = [{\bf P}_n(z)]_{11} = [{\bf Z}_n^{(k)}(z){\bf Y}_n^{(k)}(z){\bf M}^{(k)}(z)]_{11}e^{ng(z)}
\endeq
for $z\in\mathbb{C}\backslash(\Omega_+\cup\Omega_-)$ 
(here we have used the fact that ${\bf D}_+^u(z)$ and ${\bf D}_-^u(z)$ are 
upper-triangular to obtain the result for $-\delta\leq\Im z\leq\mu$).
From \eqref{VZ-bound} and the small-norm theory of Riemann--Hilbert problems 
\cite{DeiftZ:1993}, we have 
\eq
\label{Z-bound}
{\bf Z}_n^{(k)}(z) = \mathbb{I} + \mathcal{O}(n^{-1/2}).  
\endeq
Furthermore, from 
\eqref{Y-def}, \eqref{B-ito-Qpm}, \eqref{Q-def}, and \eqref{Qpm-def}, we also 
have 
\eq
\label{Y-bound}
{\bf Y}_n^{(k)}(z) = \mathbb{I} + \mathcal{O}\left(\frac{e^{2\pi n(\mu-\mu_c)}}{n^{k+\frac{1}{2}}}\right) + \mathcal{O}\left(\frac{n^{k-\frac{1}{2}}}{e^{2\pi n(\mu-\mu_c)}} \right).
\endeq
Thus we have
\eq
\label{p-ito-M11}
p_{n,n}^{(T,\mu,\tau)}(z) = e^{ng(z)}[{\bf M}^{(k)}(z)]_{11}\left(1 + \mathcal{O}\left(\frac{e^{2\pi n(\mu-\mu_c)}}{n^{k+\frac{1}{2}}}\right) + \mathcal{O}\left(\frac{n^{k-\frac{1}{2}}}{e^{2\pi n(\mu-\mu_c)}} \right)\right).
\endeq
By \eqref{Mk-def}, \eqref{Mk-out}, and \eqref{M0-out}, 
\eq
\label{Mk-11}
[{\bf M}^{(k)}(z)]_{11} = \left(\frac{\alpha(z)}{\alpha_\infty}\right)^k\left(\frac{\ga(z)+\ga(z)^{-1}}{2}\right), \quad z\in\mathbb{C}\backslash(\mathbb{D}_a\cup\mathbb{D}_b\cup\mathbb{D}_{\widetilde{z}_c}),
\endeq
which establishes \eqref{hermite_asymptotics_outer} and proves part (a).

For $z$ on the band but outside the Airy neighborhoods $\mathbb{D}_a$ and 
$\mathbb{D}_b$, we have (similar to \eqref{p-ito-ZYM} but taking into account 
the jumps on the lens boundaries)
\eq
p_{n,n}^{(T,\mu,\tau)}(z) = [{\bf Z}_n^{(k)}(z){\bf Y}_n^{(k)}(z){\bf M}_-^{(k)}(z)]_{11}e^{ng_-(z)} - [{\bf Z}_n^{(k)}(z){\bf Y}_n^{(k)}(z){\bf M}_-^{(k)}(z)]_{12}e^{ng_+(z)}.
\endeq
From \eqref{Z-bound} and \eqref{Y-bound},
\eq
\label{p-ito-M11-M12}
\begin{split}
& p_{n,n}^{(T,\mu,\tau)}(z) \\ 
& = \left([{\bf M}_-^{(k)}(z)]_{11}e^{ng_-(z)} - [{\bf M}_-^{(k)}(z)]_{12}e^{ng_+(z)}\right)\left(1 + \mathcal{O}\left(\frac{e^{2\pi n(\mu-\mu_c)}}{n^{k+\frac{1}{2}}}\right) + \mathcal{O}\left(\frac{n^{k-\frac{1}{2}}}{e^{2\pi n(\mu-\mu_c)}} \right)\right).
\end{split}
\endeq
Combining \eqref{Mk-def}, \eqref{Mk-out}, \eqref{outer-band-jump}, 
\eqref{M0-out}, and $\alpha_+\alpha_-=1$ from Lemma \ref{tau-properties-lem}(b) 
shows 
\eq
\label{Mk-12}
[{\bf M}_-^{(k)}(z)]_{12} = \alpha_\infty^{-k}\alpha_-(z)^{-k}[{\bf M}_{0-}^\text{(out)}(z)]_{12} = -\left(\frac{\alpha_+(z)}{\alpha_\infty}\right)^k[{\bf M}_{0+}^\text{(out)}(z)]_{11}.
\endeq
Inserting \eqref{Mk-11} and \eqref{Mk-12} into \eqref{p-ito-M11-M12} and using 
\eqref{M0-out} shows \eqref{hermite_asymptotics_band} and proves part (b).  

For $z\in\mathbb{D}_{\widetilde{z}_c}$ we may again use \eqref{p-ito-M11}, only 
now \eqref{Mk-def}, \eqref{Mk-zc}, \eqref{Fk-def}, \eqref{Mk-out}, 
\eqref{M0-out}, and \eqref{Hk-def} give
\eq
[{\bf M}^{(k)}]_{11} = \frac{\alpha^k}{\alpha_\infty^k W^k}\left(\frac{\gamma+\gamma^{-1}}{2}\right)\frac{\mathfrak{h}_k(n^{1/2}W)}{\kappa_k n^{k/2}} + \frac{\pi W^k}{\alpha_\infty^k\alpha^k}(\gamma-\gamma^{-1})\kappa_{k-1}\mathfrak{h}_{k-1}(n^{1/2}W)\frac{n^{k/2}}{e^{n\mathfrak{c}}}.
\endeq 
Looking at the $n$-dependent terms in the second summand, we have
\eq
\mathfrak{h}_{k-1}(n^{1/2}W)\frac{n^{k/2}}{e^{n\mathfrak{c}}} = \mathcal{O}\left(\frac{n^{k-\frac{1}{2}}}{e^{n\mathfrak{c}}}\right),
\endeq
which is the same as the second error term in \eqref{p-ito-M11}.  This implies 
\eqref{hermite_asymptotics_zc}, and the proof of part (c) is complete.
\end{proof}

We now prove Theorem \ref{thm-hermite-norms}.  
\begin{proof}[Proof of Theorem \ref{thm-hermite-norms} (Normalizing constants and recurrence coefficients in the Hermite regime)]
Recall from \eqref{IP4} that 
$h^{(T,\mu,\tau)}_{n, n}=[\mathbf P_{n,1}]_{12}$ and 
$\left(h^{(T,\mu,\tau)}_{n, n-1}\right)^{-1}=[\mathbf P_{n,1}]_{21}$.  From the 
off-diagonal entries of \eqref{P-large-z}, we have 
\eq
\label{P112-P121}
\begin{split}
[{\bf P}_{n,1}]_{12} & = -2\pi i\left(\left[{\bf M}_1^{(k)}\right]_{12} + \left[{\bf Z}_{n,1}^{(k)}\right]_{12} + [{\bf B}]_{12} \right)e^{n\ell},\\
[{\bf P}_{n,1}]_{21} & = \frac{-1}{2\pi i}\left(\left[{\bf M}_1^{(k)}\right]_{21} + \left[{\bf Z}_{n,1}^{(k)}\right]_{21} + [{\bf B}]_{21} \right)e^{-n\ell}.
\end{split}
\endeq
To find the necessary entries of ${\bf M}_1$, recall from \eqref{Mk-def} that 
${\bf M}^{(k)}(z)={\bf M}_k^\text{(out)}(z)$ for $|z|$ large enough.  
Combining \eqref{Mk-out}, \eqref{M0-out}, \eqref{gamma-def}, and 
Lemma \ref{tau-properties-lem}(e) gives
\eq
\begin{split}
\left[{\bf M}_k^\text{(out)}(z)\right]_{12} & = \frac{\gamma(z)-\gamma(z)^{-1}}{-2i}\alpha_\infty^{-k}\alpha(z)^{-k} = \left(\frac{i}{\sqrt{T}z}+\mathcal{O}\left(\frac{1}{z^2}\right)\right)\left(1+\mathcal{O}\left(\frac{1}{z}\right)\right)\alpha_\infty^{-2k}, \\
\left[{\bf M}_k^\text{(out)}(z)\right]_{21} & = \frac{\gamma(z)-\gamma(z)^{-1}}{2i}\alpha_\infty^k\alpha(z)^k = \left(\frac{-i}{\sqrt{T}z}+\mathcal{O}\left(\frac{1}{z^2}\right)\right)\left(1+\mathcal{O}\left(\frac{1}{z}\right)\right)\alpha_\infty^{2k}.
\end{split}
\endeq
Therefore
\eq
\label{Mk1-offdiag}
\left[{\bf M}_1^{(k)}\right]_{12} = \frac{i}{\sqrt{T}}\alpha_\infty^{-2k}, \quad \left[{\bf M}_1^{(k)}\right]_{21} = \frac{-i}{\sqrt{T}}\alpha_\infty^{2k}.
\endeq
We now calculate $\left[{\bf Z}_{n,1}^{(k)}\right]_{12} + [{\bf B}]_{12}$ and 
$\left[{\bf Z}_{n,1}^{(k)}\right]_{21} + [{\bf B}]_{21}$.  From 
\eqref{Z1-ito-FQFinv}, \eqref{Qhat-def}, and the fact that 
$\widetilde{\bf Q}=\mathcal{O}(n^{-1/2})$, we can write 
\eq
{\bf Z}_{n,1}^{(k)} = \begin{cases} \displaystyle \frac{-{\bf F}_0\widehat{\bf Q}{\bf F}_0^{-1}}{W_1 + Q_+[{\bf F}_0^{-1}{\bf F}_1]_{21}} + \mathcal{O}\left(\frac{1}{n}\right), & \displaystyle \left(k-\frac{1}{2}\right)\frac{\log n}{n} < \mathfrak{c} \le k\frac{\log n}{n}, \vspace{.05in} \\ \displaystyle \frac{-{\bf F}_0\widehat{\bf Q}{\bf F}_0^{-1}}{W_1 + Q_-[{\bf F}_0^{-1}{\bf F}_1]_{12}} + \mathcal{O}\left(\frac{1}{n}\right), & \displaystyle k\frac{\log n}{n} < \mathfrak{c} \le \left(k+\frac{1}{2}\right)\frac{\log n}{n}. \end{cases}
\endeq
Along with \eqref{B-formula}, this means we can write the simplified formulas
\eq
\begin{split}
\left[{\bf Z}_{n,1}^{(k)}\right]_{12} + [{\bf B}]_{12} & = 
\frac{\frac{-Q_+}{W_1}[{\bf F}_0]_{11}^2}{1 + \frac{Q_+}{W_1}[{\bf F}_0^{-1}{\bf F}_1]_{21}} 
+ \frac{\frac{Q_-}{W_1}[{\bf F}_0]_{12}^2}{1 + \frac{Q_-}{W_1}[{\bf F}_0^{-1}{\bf F}_1]_{12}} 
+ \mathcal{O}\left(\frac{1}{n}\right) \\ 
  & = 
\frac{-\frac{Q_+}{W_1}[{\bf F}_0]_{11}[{\bf F}_0]_{21}\cdot\frac{[{\bf F}_0]_{11}}{[{\bf F}_0]_{21}}}{1 + \frac{Q_+}{W_1}[{\bf F}_0^{-1}{\bf F}_1]_{21}} 
+ \frac{\frac{Q_-}{W_1}[{\bf F}_0]_{12}[{\bf F}_0]_{22}\cdot\frac{[{\bf F}_0]_{12}}{[{\bf F}_0]_{22}}}{1 + \frac{Q_-}{W_1}[{\bf F}_0^{-1}{\bf F}_1]_{12}} 
+ \mathcal{O}\left(\frac{1}{n}\right), \\ 
\left[{\bf Z}_{n,1}^{(k)}\right]_{21} + [{\bf B}]_{21} & = 
\frac{\frac{Q_+}{W_1}[{\bf F}_0]_{21}^2}{1 + \frac{Q_+}{W_1}[{\bf F}_0^{-1}{\bf F}_1]_{21}} 
-
\frac{\frac{Q_-}{W_1}[{\bf F}_0]_{22}^2}{1 + \frac{Q_-}{W_1}[{\bf F}_0^{-1}{\bf F}_1]_{12}} 
+ \mathcal{O}\left(\frac{1}{n}\right)\\
  & = 
\frac{\frac{Q_+}{W_1}[{\bf F}_0]_{11}[{\bf F}_0]_{21}\cdot\frac{[{\bf F}_0]_{21}}{[{\bf F}_0]_{11}}}{1 + \frac{Q_+}{W_1}[{\bf F}_0^{-1}{\bf F}_1]_{21}} 
-
\frac{\frac{Q_-}{W_1}[{\bf F}_0]_{12}[{\bf F}_0]_{22}\cdot\frac{[{\bf F}_0]_{22}}{[{\bf F}_0]_{12}}}{1 + \frac{Q_-}{W_1}[{\bf F}_0^{-1}{\bf F}_1]_{12}} 
+ \mathcal{O}\left(\frac{1}{n}\right).
\end{split}
\endeq 
Putting together \eqref{QWF0F1}, \eqref{Rpm-def}, \eqref{Fk-scalar-def}, and 
\eqref{Gk-def} leads us to 
\eq
\frac{Q_+}{W_1}[{\bf F}_0^{-1}{\bf F}_1]_{21} = (-1)^{n+1}F_k e^{2\pi i\tau} = G_k, \quad \frac{Q_-}{W_1}[{\bf F}_0^{-1}{\bf F}_1]_{12} = \frac{1}{(-1)^{n+1}F_{k-1}e^{2\pi i\tau}} = G_{k-1}^{-1}.
\endeq
Similarly, combining \eqref{QWFF-ito-Rpm}, \eqref{Rpm-def}, \eqref{Fk-scalar-def}, 
and \eqref{Gk-def} shows
\eq
\frac{Q_+}{W_1}[{\bf F}_0]_{11}[{\bf F}_0]_{21} = \frac{2\pi i}{T}G_k, \quad -\frac{Q_-}{W_1}[{\bf F}_0]_{12}[{\bf F}_0]_{22} = -\frac{2\pi i}{T}G_{k-1}^{-1}.
\endeq
In addition, from \eqref{F0-def}, \eqref{gamma-star}, and \eqref{lambda-def},
\eq
\frac{[{\bf F}_0]_{11}}{[{\bf F}_0]_{21}} = i\frac{\gamma_*+\gamma_*^{-1}}{\gamma_*-\gamma_*^{-1}}\alpha_\infty^{-2k} = -\lambda\alpha_\infty^{-2k}, \quad \frac{[{\bf F}_0]_{12}}{[{\bf F}_0]_{22}} = i\frac{\gamma_*-\gamma_*^{-1}}{\gamma_*+\gamma_*^{-1}}\alpha_\infty^{-2k} = \lambda^{-1}\alpha_\infty^{-2k}.
\endeq  
Together, the previous four equations give
\eq
\label{Z1pB-offdiag}
\begin{split}
\left[{\bf Z}_{n,1}^{(k)}\right]_{12} + [{\bf B}]_{12} & = 
\left(\frac{2\pi i}{T}\frac{G_k}{1 + G_k} \lambda 
+ \frac{2\pi i}{T}\frac{G_{k-1}^{-1}}{1 + G_{k-1}^{-1}} \lambda^{-1}
+ \mathcal{O}\left(\frac{1}{n}\right) \right)\alpha_\infty^{-2k}, \\ 
\left[{\bf Z}_{n,1}^{(k)}\right]_{21} + [{\bf B}]_{21} & = 
\left(- \frac{2\pi i}{T}\frac{G_k}{1 + G_k}\lambda^{-1}
-\frac{2\pi i}{T}\frac{G_{k-1}^{-1}}{1 + G_{k-1}^{-1}}\lambda
 + \mathcal{O}\left(\frac{1}{n}\right)\right)\alpha_\infty^{2k}.
\end{split}
\endeq 
Finally, putting together \eqref{P112-P121}, \eqref{Mk1-offdiag}, 
\eqref{Z1pB-offdiag}, \eqref{eq:LWg-definition}, and \eqref{eq:def-g-function} 
gives \eqref{hnn-asymptotics} and 
\eqref{hnnm1-asymptotics}, as desired.  

To prove \eqref{gamma1-asymptotics}, we simply take the product of \eqref{hnn-asymptotics} and 
\eqref{hnnm1-asymptotics} (see \eqref{eq:defn_of_gamma_nk}), noting that for $k>0$,
\eq
\frac{G_k}{G_{k-1}} = \bigO(n^{-1}),
\endeq
by \eqref{Gk-def} and \eqref{Fk-scalar-def}.  This concludes the proof.
\end{proof}

\end{document}